\documentclass[final,5p,times]{elsarticle}

\usepackage[T1]{fontenc}
\usepackage[utf8]{inputenc}
\usepackage{booktabs}
\usepackage{microtype}
\usepackage{adjustbox}
\usepackage{amsmath,amssymb,amsfonts,amsthm,mathtools,bm}
\usepackage{aliascnt}
\usepackage{enumitem}
\usepackage{float}
\usepackage{xcolor}
\usepackage[colorlinks=true,linkcolor=blue!55!black,urlcolor=blue!55!black]{hyperref}
\usepackage[nameinlink,capitalise,noabbrev]{cleveref}
\biboptions{sort&compress}
\newcommand{\fitfigure}[1]{%
  \includegraphics[width=\columnwidth,height=0.72\textheight,keepaspectratio]{#1.png}%
}

\newenvironment{fittabular}[1]
  {\begin{adjustbox}{max width=\columnwidth}\begin{tabular}{#1}}
  {\end{tabular}\end{adjustbox}}

\numberwithin{equation}{section}

\theoremstyle{plain}
\newtheorem{theorem}{Theorem}[section]
\newaliascnt{proposition}{theorem}
\newtheorem{proposition}[proposition]{Proposition}
\aliascntresetthe{proposition}
\newaliascnt{lemma}{theorem}

\aliascntresetthe{lemma}
\newaliascnt{corollary}{theorem}
\newtheorem{corollary}[corollary]{Corollary}
\aliascntresetthe{corollary}

\theoremstyle{definition}
\newaliascnt{definition}{theorem}
\newtheorem{definition}[definition]{Definition}
\aliascntresetthe{definition}
\newaliascnt{assumption}{theorem}
\newtheorem{assumption}[assumption]{Assumption}
\aliascntresetthe{assumption}
\newaliascnt{remark}{theorem}
\newtheorem{remark}[remark]{Remark}
\aliascntresetthe{remark}

\crefname{theorem}{theorem}{theorems}
\Crefname{theorem}{Theorem}{Theorems}
\crefname{proposition}{proposition}{propositions}
\Crefname{proposition}{Proposition}{Propositions}
\crefname{lemma}{lemma}{lemmas}
\Crefname{lemma}{Lemma}{Lemmas}
\crefname{corollary}{corollary}{corollaries}
\Crefname{corollary}{Corollary}{Corollaries}
\crefname{definition}{definition}{definitions}
\Crefname{definition}{Definition}{Definitions}
\crefname{assumption}{assumption}{assumptions}
\Crefname{assumption}{Assumption}{Assumptions}
\crefname{remark}{remark}{remarks}
\Crefname{remark}{Remark}{Remarks}
\crefname{appendix}{appendix}{appendices}
\Crefname{appendix}{Appendix}{Appendices}
\crefname{subappendix}{appendix}{appendices}
\Crefname{subappendix}{Appendix}{Appendices}
\crefformat{appendix}{#2#1#3}
\Crefformat{appendix}{#2#1#3}
\crefmultiformat{appendix}{#2#1#3}{ and~#2#1#3}{, #2#1#3}{, and~#2#1#3}
\Crefmultiformat{appendix}{#2#1#3}{ and~#2#1#3}{, #2#1#3}{, and~#2#1#3}
\crefrangeformat{appendix}{#3#1#4--#5#2#6}
\Crefrangeformat{appendix}{#3#1#4--#5#2#6}
\crefformat{subappendix}{#2#1#3}
\Crefformat{subappendix}{#2#1#3}
\crefmultiformat{subappendix}{#2#1#3}{ and~#2#1#3}{, #2#1#3}{, and~#2#1#3}
\Crefmultiformat{subappendix}{#2#1#3}{ and~#2#1#3}{, #2#1#3}{, and~#2#1#3}
\crefrangeformat{subappendix}{#3#1#4--#5#2#6}
\Crefrangeformat{subappendix}{#3#1#4--#5#2#6}

\newcommand{\R}{\mathbb{R}}
\newcommand{\OmegaBar}{\overline{\Omega}}
\newcommand{\dd}{\,\mathrm{d}}
\newcommand{\argmin}{\operatorname*{arg\,min}}
\newcommand{\Prob}{\mathbb{P}}
\newcommand{\Cov}{\operatorname{Cov}}

\DeclarePairedDelimiter{\norm}{\lVert}{\rVert}
\DeclarePairedDelimiter{\abs}{\lvert}{\rvert}

\newcommand{\Lop}{\mathcal{L}}
\newcommand{\Bop}{\mathcal{B}}
\newcommand{\Fop}{\mathcal{F}}

\newcommand{\Utheta}{U_\theta}
\newcommand{\ustar}{u^\star}
\newcommand{\utheta}{u_\theta}

\newcommand{\Ih}{\mathcal{I}_h}
\newcommand{\Uhadm}{\mathcal{U}_h}
\newcommand{\Th}{\mathcal{T}_h}
\newcommand{\Nh}{\mathcal{N}_h}
\newcommand{\Vhpone}{V_h^{P_1}}
\newcommand{\uhtheta}{u_{h,\theta}}

\definecolor{techgold}{HTML}{B3A369}
\definecolor{techblue}{HTML}{003057}
\definecolor{buzzgold}{HTML}{EAAA00}
\definecolor{greymatter}{HTML}{54585A}
\definecolor{uforange}{HTML}{FA4616}
\definecolor{ufblue}{HTML}{0021A5}
\definecolor{shale}{HTML}{54585A}

\newif\ifshowcomments
\showcommentsfalse

\newcommand{\na}[1]{%
    \ifshowcomments
        \textcolor{techgold}{[NA]: #1}%
    \fi
}

\journal{Computer Methods in Applied Mechanics and Engineering}

\begin{document}

\begin{frontmatter}

\title{Pointwise Error Estimates for Numerical Physics-Informed Neural Networks}

\author[gt]{Nivar Anwer\fnref{fn1}}

\author[cs,transvalor]{Marien Chenaud\fnref{fn1}}

\author[transvalor]{Jos{\'e} Alves}

\author[cs]{Fr{\'e}d{\'e}ric Magoul{\`e}s}

\affiliation[gt]{
  organization={School of Computer Science, Georgia Institute of Technology, Atlanta, GA, USA}
}

\affiliation[cs]{
  organization={MICS, CentraleSup\'elec, Paris Saclay University, 9 Rue Joliot-Curie, Gif sur Yvette, 91400, France}
}

\affiliation[transvalor]{
  organization={Transvalor SA, 955 Avenue Roumanille, Biot, 06904, France}
}

\fntext[fn1]{These authors contributed equally to this work.}

\begin{abstract}
Physics-informed neural networks are often evaluated by residual losses sampled at finitely many points, which do not by themselves certify pointwise values of a partial differential equation solution. In this work, deterministic pointwise error intervals are developed for mesh-based, piecewise-linear numerical physics-informed neural networks. The proposed error estimation is given for a compatible field, which is the finite-element reconstruction of an admissible prediction on a mesh. The certifying residual is then obtained by applying the finite-dimensional numerical system to this compatible field. For compatible square linear systems, the pointwise error relative to the discrete target has an exact adjoint Green representation, and the computed signed error recovers the finite element solution exactly. Norm-based, inexact, localized, and randomized variants provide computable intervals when the exact correction computation is impractical. The extension from the discrete target to the continuous solution is supplied by comparison estimates. For a one-dimensional coercive reaction-diffusion class, this transfer layer is made fully computable by an explicit residual-based a posteriori estimator with querywise constants. The error bound derivation is extended to nonlinear residual systems with explicit Taylor remainders. Numerical experiments assess compatibility and calibration on manufactured examples, on a large-scale public three-dimensional elasticity benchmark, and on projected neural load families on the same benchmark. 
\end{abstract}

    \end{frontmatter}

\crefname{section}{Section}{Sections}
\Crefname{section}{Section}{Sections}

\section{Introduction}
\label{sec:introduction}
\na{[note to self] add more relevant citations to intro}
Physics-informed neural networks (PINNs) are commonly trained using sampled PDE residuals, boundary terms, and data constraints \cite{Raissi2019PINNs,KarniadakisEtAl2021PIML}. However, the numerical values of the corresponding loss terms do not constitute a certificate for the predicted spatial field. In this work, we provide deterministic interval certifications for values of the exact PDE solution at specified query points. Such deterministic certificate must guarantee that the exact solution value is contained in the reported interval. A small empirical residual loss at finitely many samples does not provide such a guarantee; rigorous error control for PINN-type methods depends on stability, sampling, approximation, optimization, and loss consistency assumptions rather than on the sampled residual value alone \cite{DeRyckMishra2024Acta,BonitoDeVorePetrovaSiegel2026ConsistentPINNs}.

The inability to construct such certificate from loss values computed at fixed collocation points is structural. For an unprojected collocation PINN trained using finitely many automatic-differentiation residual samples, the field may be altered away from sampled locations without changing the sampled loss, while an unsampled point value may be changed arbitrarily. Finite collocation sampling therefore does not control unsampled point values deterministically, unless additional structure is imposed. \Cref{prop:finite_collocation_no_point_control} formalizes this limitation, and the compatibility condition introduced next provides the finite-dimensional structure needed for certification.

The certified model class is a mesh-based, piecewise-linear numerical PINN. The trainable model outputs an admissible coefficient vector in a finite-dimensional space. The resulting continuous field is the finite-element interpolation of that vector. The certifying residual that is used to derive the error bounds is the residual of the finite-dimensional numerical system applied to the same vector. This alignment is called \emph{compatibility}. Compatibility is the central structural condition, since it makes the certified object well defined: the field, the residual, and the query all act on one admissible discrete state.

The error estimate is conducted as follows. First, for a compatible linear residual system, the error admits an exact discrete adjoint Green identity relative to a discrete numerical target. The signed correction recovers that target exactly, while norm-based, localized, randomized, and inexact variants yield computable approximate intervals when full correction is impractical. The transfer from the discrete target to the exact PDE solution is handled separately by rigorous comparison estimates. Finally, the same procedure extends to nonlinear residual systems. The stationary-target variant is given in \Cref{appsec:nonlinear_stationary_variant}. The linear square-system formulation assumes an invertible reduced operator after admissible reduction, covering the standard conforming Galerkin setting. A full-column-rank overdetermined extension is given in \Cref{appsec:linear_rectangular_extension}. Rank-deficient reduced operators are outside the certified scope.

The main contributions are as follows.
\begin{enumerate}[label=\arabic*.,leftmargin=2.0em]
    \item Compatibility is identified as the condition that makes deterministic pointwise certification well defined for a predicted numerical PINN field.
    \item The square-system linear certificate is derived from the discrete target, through the exact identity and transfer step, to the final certified interval; a full-column-rank overdetermined extension is given in \Cref{appsec:linear_rectangular_extension}.
    \item Discrete certification is separated from discrete-to-continuous transfer, so that the final error bound decomposes into distinct error contributions, namely: discrete correction or residual, query sensitivity, transfer, and representation mismatch.
    \item A fully computable a posteriori transfer theorem is proved for a concrete coercive class, namely one-dimensional reaction--diffusion with conforming \(P_1\) Galerkin discretization, explicit residual estimators, and querywise constants.
    \item The same procedure is extended to localized, randomized, and nonlinear certificates, with certified inexact-solve tolerances and a stationary-target extension provided in the appendices.
    \item The theoretical claims are evaluated on manufactured pointwise benchmarks in two and three dimensions, on a public three-dimensional elasticity benchmark, and on projected neural load families, using scalar compatible queries, query-dependent adjoint Green sensitivities, corrected centers, raw-output intervals after projection, and comparison identities based on the released benchmark field.
\end{enumerate}

For theoretical reasons stated in \Cref{prop:finite_collocation_no_point_control}, the proposed method does not provide a certificate for unprojected collocation PINNs trained only on finitely many sampled residuals. The certificate applies after a compatible finite-dimensional numerical system is specified and the predicted field is represented within that system. Second, in a single linear instance, exact correction recovers the target of the chosen numerical discretization. The numerical PINN is therefore not presented as a single-instance substitute for a direct classical resolution, but as a source of admissible states that can be certified, approximately corrected, or reused across parametric, many-instance, and nonlinear settings.

The paper is organized as follows. \Cref{sec:related_work} positions the work relative to the closest PINN, finite-element PINN, and goal-oriented finite-element literatures. \Cref{sec:framework} defines the certified object, the notation, and the scope assumptions. \Cref{sec:linear_main} gives the square-system linear certificate. \Cref{sec:linear_refinements} gives downstream computational variants and finite-element specializations, including a fully computable a posteriori transfer theorem for one-dimensional coercive reaction--diffusion. \Cref{appsec:linear_rectangular_extension} gives the full-column-rank overdetermined linear extension. \Cref{sec:nonlinear} gives the exact-root nonlinear certificate, while the stationary-target extension is given in \Cref{appsec:nonlinear_stationary_variant}. \Cref{sec:computational_procedure} provides an implementation procedure. \Cref{sec:experiments} organizes the experiments by theorem-level claim. Additional experiments and reporting details are given in \Cref{app:additional_experiments}. Auxiliary analytical material is given in \Cref{appsec:main_relocated_material}, and the proofs are given in \Cref{app:proofs}.

\section{Related Work}
\label{sec:related_work}

The problem that is addressed here is the computation of certified error bounds for PINN predictions. Recent works develop residual-based analysis of physics-informed models, and rigorous control of neural PDE solvers. Generalization, stability, and residual-to-solution estimates are developed in \cite{MishraMolinaro2023Generalization,MishraMolinaro2022Inverse,DeRyckMishra2024Acta,ZeinhoferMasriMardal2025UnifiedFramework}, with Navier--Stokes extensions in \cite{DeRyckJagtapMishra2024NavierStokes}. For elliptic problems, \cite{BonitoDeVorePetrovaSiegel2026ConsistentPINNs} show that provable error control depends critically on loss consistency. Rigorous a posteriori certification has also begun to appear, including \cite{HillebrechtUnger2025RigorousAPosterioriPINNs,XuYaguchi2025APosterioriNavierStokesPINNs}. These studies are closest in scope because they connect residual information to rigorous solution control. Other works, such as \cite{doumeche2023convergence}, leverage statistical learning theory to derive probabilistic error bounds, rather than deterministic guarantees. Finally, the authors of \cite{zhu2022numerical} derive an error decomposition for neural differential equation solvers, separating the contribution of the time discretization from that of the approximation error. However, their analysis focuses on neural ordinary differential equation solvers embedded within a numerical time-stepping framework, rather than on direct predictions produced by PINNs.


A second related line introduces variational, finite-element, or mesh-based structure into PINNs. hp-VPINNs, cPINNs, and related weak formulations build neural trial spaces into weighted-residual or domain-decomposition schemes \cite{KharazmiZhangKarniadakis2021hpVPINNs,JagtapKharazmiKarniadakis2020cPINNs}. Quadrature and test-space effects are analyzed in \cite{BerroneCanutoPintore2022VPINNQuadratures}, and reliable residual-type a posteriori estimators in the energy norm are given in \cite{BerroneCanutoPintore2022VPINNAPosteriori}. Finite element informed neural networks (FEINNs) and their variants associate PINNs with increasingly explicit finite-element structure \cite{BadiaLiMartin2024FEINN,BadiaLiMartin2025CompatibleFEINN,BadiaLiMartin2025AdaptiveFEINN, ChenaudMagoulesAlves2023PIGCN,ChenaudMagoulesAlves2024PIGMN}. The Deep Ritz literature provides a related variational approach with its own approximation and certification results \cite{EYu2018DeepRitz,MinakowskiRichter2023DeepRitz,GrekasMakridakis2025DeepRitzFEM}. These works endow neural approximations with numerical structure. The present work is different in objective: it does not propose a new mesh-based neural solver, but a deterministic pointwise certification layer for a released admissible numerical field whose residual and query map are compatible with the same finite-dimensional system.

A separate issue is whether a mesh-based numerical PINN has a role beyond a direct finite-element solve. For a single well-posed coercive linear instance requiring high accuracy, no advantage over a direct finite-element solve is claimed; under exact correction, the certificate recovers the target of the chosen numerical discretization. The relevant regime is one in which a trainable model produces admissible states that are reused across loads, parameters, geometries, or queries, corrected only approximately, or used as local base points for nonlinear verification. In that regime, the finite-element structure is essential. It supplies the admissible coordinates, query map, and compatible residual needed to define which numerical object is certified. This object is not defined by an unprojected collocation PINN alone.

The transfer layer and the pointwise viewpoint draw on classical finite-element analysis rather than on PINN training theory. Goal-oriented and dual-weighted-residual methods for pointwise or scalar quantities of interest were developed in \cite{PrudhommeOden1999GoalOrientedPointwise,OdenPrudhomme2001GoalOriented,BeckerRannacher2001OptimalControlAPosteriori}, while classical pointwise finite-element estimates date to \cite{RannacherScott1982OptimalPointwise}. These techniques serve a different purpose in the present setting. The approximate state is supplied by a numerical PINN, the discrete certificate is attached to the compatible residual of the same discrete system that defines the certified field. The adaptation to the continuous problem is completed only after a separate discrete-to-continuous transfer step.


The main contribution of our work, relative to the aforementioned studies, can be articulated along three complementary aspects. First, the certified quantity is defined through the compatibility between the field, the residual, and the query evaluated on a common admissible vector. Second, the passage from the discrete target to the exact PDE solution is formulated as an independent certified layer. Finally, our primary objective is the deterministic certification of scalar compatible queries, including pointwise evaluations, rather than the sole control of residuals or energy norms.

\section{Framework and Compatibility}
\label{sec:framework}

This section defines the scope used in this work to derive error bounds. First, the discrete, target solution is projected onto the exact PDE solution space. The error bounds are derived from this projection. An alternative construction based on direct continuous residual-to-point stability is given in \Cref{appsec:continuous_residual_route}.

\subsection{Definition and objective}
\label{subsec:problem_statement}

Let \(\Omega\subset\R^d\) be a bounded domain with closure \(\OmegaBar\). The goal is to certify the values of the exact PDE solution \(\ustar\) at prescribed query points
\begin{equation}
    \mathcal{Q}=\{x_1,\dots,x_K\}\subset\OmegaBar.
\end{equation}
For each query point, define the point-evaluation functional
\begin{equation}
    \lambda_i(u):=u(x_i),
    \qquad i=1,\dots,K.
\end{equation}
The desired deterministic certificate is an interval \(I_i\subset\R\) satisfying
\begin{equation}
    \ustar(x_i)\in I_i,
    \qquad i=1,\dots,K.
    \label{eq:deterministic_target}
\end{equation}
When randomized Green-sensitivity estimation is employed, the corresponding simultaneous certificate is
\begin{equation}
    \Prob_{\Xi}
    \left\{
    \ustar(x_i)\in I_i\ \text{for all } i=1,\dots,K
    \right\}
    \ge 1-\delta,
    \label{eq:random_target}
\end{equation}
where the probability is taken only over the auxiliary randomization \(\Xi\). This probability statement is not a probabilistic model for \(\ustar\).

The certification problem therefore has two stages. A compatible discrete target is certified first; the resulting statement is then transferred to the continuous PDE solution through independently certified comparison bounds.

\begin{table}[t]
    \centering
    \caption{Core objects used throughout the analysis.}
    \begin{fittabular}{@{}ll@{}}
        \toprule
        Symbol & Meaning \\
        \midrule
        \(\ustar\) & exact PDE solution \\
        \(\Utheta=U_{0,h}+Z_hy_\theta\) & admissible coefficient vector produced by the numerical PINN \\
        \(\uhtheta=\Ih\Utheta\) & certified reconstruction attached to \(\Utheta\) \\
        \(U_h^\star\) & compatible discrete target of the chosen numerical system \\
        \(U_h^\pi\) & admissible comparison state used for transfer to \(\ustar\) \\
        \(\utheta\) & unprojected neural field before extraction, when present \\
        \(\zeta_{i,h}^{\theta}\) & querywise mismatch between \(\utheta\) and \(\uhtheta\) \\
        \(r_h=F_h(\Utheta)\) or \(A_h\Utheta-b_h\) & compatible residual of the same admissible vector \\
        \(\ell_{i,h}\) & compatible coefficient-space query vector \\
        \bottomrule
    \end{fittabular}
    \label{tab:notation_map}
\end{table}

\subsection{Certified numerical object and compatibility}
\label{subsec:p1_framework}
\label{subsec:compatible_system}

The certificate is attached to one admissible coefficient vector and to one discrete query/residual system built on the same coordinates. The algebra is stated first for a general scalar compatible query. Nodal and off-grid \(P_1\) point values are then immediate special cases.

Let \(V_h\) be a finite-dimensional trial space, let \(\Uhadm\subseteq\R^{N_h}\) be the admissible coefficient set, let
\[
    \Ih:\R^{N_h}\to V_h
\]
be a linear reconstruction map, and let
\[
    F_h:\Uhadm\to\R^{M_h}
\]
be the discrete residual map.

For each prescribed query, assume there exists a vector \(\ell_{i,h}\in\R^{N_h}\) such that
\begin{equation}
    \Lambda_{i,h}(U):=\ell_{i,h}^{\top}U
    \qquad\text{for every }U\in\Uhadm.
    \label{eq:general_discrete_query}
\end{equation}
In the main pointwise setting,
\begin{equation}
    \Lambda_{i,h}(U)=(\Ih U)(x_i),
    \qquad i=1,\dots,K,
    \label{eq:point_query_special_case}
\end{equation}
Other scalar compatible queries can be assembled in the same basis.

\begin{definition}[Compatible finite-dimensional certification system]
A compatible certification system consists of \((\Uhadm,\Ih,F_h,\{\ell_{i,h}\}_{i=1}^K)\) such that
\begin{enumerate}[label=(\roman*),leftmargin=2.2em]
    \item \(\Ih\) is linear;
    \item \(\Uhadm\subseteq \R^{N_h}\) is the admissible coefficient set;
    \item the reported query value is \(\Lambda_{i,h}(U)=\ell_{i,h}^{\top}U\);
    \item the certifying residual is computed from the same admissible vector, namely
    \[
        r_h(U):=F_h(U).
    \]
\end{enumerate}
Compatibility means that the reported value \(\Lambda_{i,h}(\Utheta)\), the certified reconstruction \(\Ih\Utheta\), and the certifying residual \(r_h(\Utheta)\) are all attached to the same admissible vector \(\Utheta\in\Uhadm\).
\end{definition}

This definition specifies the certified scope: admissible trainable models must output, or be explicitly mapped to, a coefficient vector in a specified finite-dimensional space. The certificate is attached to that vector and to the numerical system built on it, not to the training loss alone.

The linear results in \Cref{sec:linear_main} specialize to \(F_h(U)=A_hU-b_h\). The nonlinear results in \Cref{sec:nonlinear} use the same compatibility requirement after reduction to admissible coordinates.

A mesh-based numerical PINN is any trainable map that outputs such an admissible vector \(\Utheta\in\Uhadm\). The certified field is
\begin{equation}
    \uhtheta:=\Ih \Utheta.
    \label{eq:certified_field_general}
\end{equation}
If an unprojected neural field \(\utheta\) is trained first, then an explicit extraction map
\[
    \Pi_h:\utheta\mapsto \Utheta\in\Uhadm
\]
must be specified. Common choices include nodal interpolation, projection, and constrained projection. The certificate is attached to \(\uhtheta=\Ih \Utheta\), not to the unprojected neural field.

For the mesh-based \(P_1\) case, we use standard conforming simplicial finite-element notation \cite{BrennerScott2008FEM}. Let \(\Th\) be a conforming, shape-regular simplicial mesh of \(\Omega\) with nodal set
\[
    \Nh=\{a_1,\dots,a_{N_h}\}.
\]
If a fitted mesh is not used and the computation is performed on \(\Omega_h\neq\Omega\), the resulting geometric consistency terms are absorbed later into the transfer bounds.

Let
\[
    \Vhpone
    :=
    \left\{
    v_h\in C^0(\OmegaBar):
    v_h|_T\in P_1(T)\ \text{for every }T\in\Th
    \right\},
\]
with nodal basis \(\{\phi_j\}_{j=1}^{N_h}\). The reconstruction is
\begin{equation}
    \Ih U=\sum_{j=1}^{N_h}U_j\phi_j.
    \label{eq:p1_reconstruction_map}
\end{equation}
Hence the point-query vector at \(x\in\OmegaBar\) is
\begin{equation}
    \ell_h(x):=\bigl(\phi_1(x),\dots,\phi_{N_h}(x)\bigr)^\top,
    \qquad
    (\Ih U)(x)=\ell_h(x)^\top U.
    \label{eq:p1_query_vector}
\end{equation}
Consequently:
\begin{enumerate}[label=(\alph*),leftmargin=2.2em]
    \item if \(x=a_k\) is a mesh node, then \(\ell_h(x)=e_k\) and the query is a nodal value;
    \item if \(x\) lies in an element \(T=\operatorname{conv}\{a_{j_0},\dots,a_{j_d}\}\), then the nonzero entries of \(\ell_h(x)\) are exactly the barycentric coordinates of \(x\) in \(T\), so the query is an off-grid point value;
    \item if a different scalar linear functional is needed later, its compatible query vector is assembled once in the same coefficient basis and used by the same algebra.
\end{enumerate}
For the prescribed point set \(\mathcal Q=\{x_1,\dots,x_K\}\), define
\begin{equation}
    \ell_{i,h}:=\ell_h(x_i),
    \qquad i=1,\dots,K.
    \label{eq:query_vectors}
\end{equation}

The compatible residual of the certified vector is
\begin{equation}
    r_h:=F_h(\Utheta).
    \label{eq:compatible_discrete_residual}
\end{equation}
In the linear setting this becomes \(r_h=A_h\Utheta-b_h\); in the nonlinear setting the reduced residual is introduced in \Cref{sec:nonlinear}. A sampled automatic-differentiation residual may replace \eqref{eq:compatible_discrete_residual} only if it is proved to equal the residual of this same compatible finite-dimensional system applied to an admissible vector.

If the final reported center is an unprojected neural output \(\utheta(x_i)\) rather than the certified reconstruction \(\uhtheta(x_i)=\ell_{i,h}^{\top}\Utheta\), we define the representation mismatch
\begin{equation}
    \zeta_{i,h}^{\theta}
    :=
    \abs{\utheta(x_i)-\ell_{i,h}^{\top}\Utheta}.
    \label{eq:representation_mismatch}
\end{equation}

The compatibility chain is summarized in \Cref{fig:certificate-path}.

\begin{figure}[t]
    \centering
    \includegraphics[
        width=\columnwidth,
        height=0.36\textheight,
        keepaspectratio
    ]{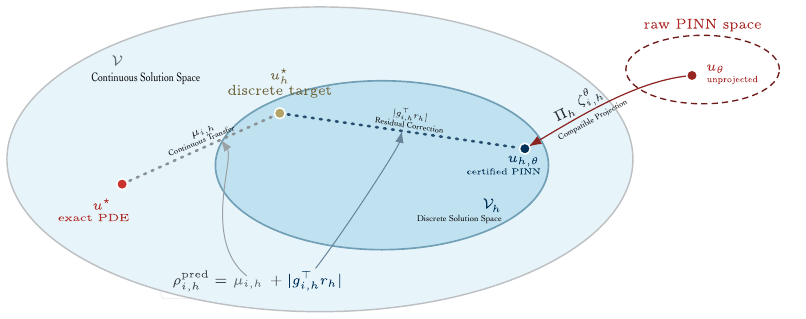}
    \caption{Compatibility path: projection, discrete correction, and transfer to the PDE target.}
    \label{fig:certificate-path}
\end{figure}

For linear systems, the primary formulation assumes square operators after admissible reduction; the full-column-rank overdetermined extension is given in \Cref{appsec:linear_rectangular_extension}.

For vector-valued problems, such as elasticity, the coefficient-level algebra is unchanged: replace \(V_h\) by the corresponding vector-valued finite-element space, stack the coefficients in \(U\), and choose \(\ell_{i,h}\) to extract the desired scalar component or scalar linear functional.

\subsection{Finite collocation residuals evaluation do not control unsampled point values}
\label{subsec:collocation_obstruction}

The following proposition justifies the impossibility of constructing deterministic error bounds from finitely many collocation points, and therefore justifies the framework proposed here.

\begin{proposition}[Finite collocation samples do not certify an unsampled point value]
\label{prop:finite_collocation_no_point_control}
Let \(S=\{z_1,\dots,z_m\}\subset\Omega\cup\partial\Omega\) be the finite set of collocation and boundary sample points used by an unprojected automatic-differentiation PINN loss. Assume the sampled loss depends only on finitely many values of \(u\) and its derivatives up to finite order at points of \(S\).

Fix an interior point \(x_0\in\Omega\setminus S\). Then there exists \(\varphi\in C_c^\infty(\Omega)\) such that
\begin{equation}
    \begin{multlined}
    \varphi(x_0)=1,
    \qquad
    D^\alpha\varphi(z_j)=0
    \quad
    \text{for every sample point }z_j\\
    \text{and every derivative }D^\alpha
    \text{ used by the loss}.
    \end{multlined}
\end{equation}
Consequently, for every candidate field \(u\) and every \(\beta\in\R\), the sampled loss takes the same value at \(u\) and \(u+\beta\varphi\), while
\[
    (u+\beta\varphi)(x_0)=u(x_0)+\beta.
\]
Hence finitely many collocation residual samples alone cannot yield a deterministic certificate for the unsampled point value at \(x_0\).
\end{proposition}

This proposition identifies a potential failure mode at the level of the underlying smooth problem, even if the specific perturbation  \(u+\beta\varphi\) may not be representable by a given PINN architecture.

\begin{remark}[Scope of mesh-based restrictions]
\label{rem:mesh_space_scope}
\Cref{prop:finite_collocation_no_point_control} explains why deterministic certification requires a compatible finite-dimensional numerical representation. Restricting the output to a mesh-based space, for example a \(P_1\) finite-element space, removes uncontrolled between-sample variation because every query value becomes the fixed linear functional \(\ell_{i,h}^{\top}U\). That restriction is necessary in this framework when starting from unprojected collocation losses, but it is not sufficient by itself. Valid certification still requires the compatible residual \(A_hU-b_h\), correct boundary and constraint handling, discrete stability on the admissible space, and a certified transfer from the discrete target to the continuous PDE solution.
\end{remark}

\subsection{Continuous PDE target, transfer interface, and scope}
\label{subsec:continuous_target}

The exact certification target is the continuous PDE solution \(\ustar\). The main theorems first certify a compatible discrete target and only then transfer that statement to \(\ustar\). Only the transfer interface required for the subsequent interval results is stated here; the alternative direct route based on continuous residual-to-point stability is given in \Cref{appsec:continuous_residual_route}.

Let \(V\) denote the admissible continuous solution class. For each reported query, the corresponding linear functional must be well defined on the class used in the transfer argument. For point queries in \(d\ge 2\), this is not provided by \(H_0^1(\Omega)\) alone and therefore requires additional regularity or a transfer argument stated in a stronger class. Smoothed scalar linear functionals, such as fixed-radius averages, do not require that pointwise boundedness assumption. Let
\begin{equation}
    \Lop:V\to Y_\Omega,
    \qquad
    \Bop:V\to Y_{\partial\Omega},
\end{equation}
and consider the boundary-value problem
\begin{equation}
    \Lop(\ustar)=f\quad\text{in }\Omega,
    \qquad
    \Bop(\ustar)=g\quad\text{on }\partial\Omega.
    \label{eq:continuous_pde}
\end{equation}

\begin{assumption}[Well-posed continuous target for transfer]
\label{ass:continuous_target_well_posed}
The problem \eqref{eq:continuous_pde} has a unique solution \(\ustar\in V\) in the admissible class. If only local uniqueness is available, every later transfer statement is understood on the certified neighborhood where that solution is unique.
\end{assumption}

The linear and nonlinear interval theorems below certify \(\ustar(x_i)\) by combining a discrete certificate for the compatible numerical target with a separate transfer term. That transfer term is supplied by an admissible comparison state together with rigorous residual and evaluation bounds; see \Cref{subsec:linear_transfer} for the linear construction and \Cref{prop:nonlinear_root_discrete_to_continuous} for the nonlinear exact-root analogue.

\subsection{Main scope assumptions}
\label{subsec:scope_assumptions}

The following assumptions determine the certified scope of the results.

\begin{center}
\begin{minipage}{0.94\columnwidth}
\begin{enumerate}[label=(\roman*),leftmargin=2.0em]
    \item The continuous target is well posed, and each reported query functional is meaningful on the class used in the transfer argument.
    \item The predicted field is tied to a compatible finite-dimensional system \((\Uhadm,\Ih,F_h,\{\ell_{i,h}\})\).
    \item Boundary conditions, algebraic constraints, quadrature choices, and geometry approximations are handled consistently in the admissible space and residual.
    \item In the linear square-system setting, the reduced operator is square and invertible after admissible reduction; \Cref{appsec:linear_rectangular_extension} treats the full-column-rank overdetermined extension. Rank-deficient reduced operators are outside the scope of this work.
    \item The transfer from the discrete target to the PDE solution is supplied by rigorous comparison bounds.
    \item For nonlinear systems, the relevant neighborhood, invertibility or full-rank conditions, and Taylor remainder bounds are verified.
\end{enumerate}
\end{minipage}
\end{center}

These assumptions are part of the certified scope. They define the class of released numerical PINN fields for which the certificates apply.

\section{Linear Deterministic Certification}
\label{sec:linear_main}

This section presents the square reduced linear certificate: the compatible discrete target is defined, the exact residual-to-query identity is derived, the target is transferred to the continuous PDE solution, and the deterministic interval theorem is stated. Localized, randomized, and inexact linear variants are given in \Cref{sec:linear_refinements}; the nonlinear generalization is given in \Cref{sec:nonlinear}. The proofs of the propositions and theorems of this section are given in \Cref{appsec:linear_main_proofs}.

\subsection{Admissible coefficient space and discrete target}
\label{subsec:linear_discrete_target}

Let the discrete admissible set be the affine space
\begin{equation}
    \Uhadm:=U_{0,h}+Z_h\R^{n_h}\subseteq\R^{N_h},
    \label{eq:admissible_affine_space}
\end{equation}
where \(Z_h\in\R^{N_h\times n_h}\) has full column rank. This representation covers strongly imposed boundary conditions, gauge constraints, quotient-space parametrizations, and nullspace removal. In the unconstrained case, \(U_{0,h}=0\), \(Z_h=I_{N_h}\), and \(n_h=N_h\).

Assume that the trained numerical PINN vector is admissible:
\begin{equation}
    \Utheta=U_{0,h}+Z_hy_\theta.
\end{equation}
If the unconstrained PINN vector is not admissible, it must first be replaced by a specified admissible projection, and the resulting computable mismatch must be included in the final interval.

For the linear discrete residual
\begin{equation}
    F_h(U)=A_hU-b_h,
    \qquad
    A_h\in\R^{M_h\times N_h},
    \qquad
    b_h\in\R^{M_h},
\end{equation}
define
\begin{equation}
    \widehat A_h:=A_hZ_h,
    \qquad
    \widehat b_h:=b_h-A_hU_{0,h}.
\end{equation}
Then
\begin{equation}
    F_h(U_{0,h}+Z_hy)=\widehat A_hy-\widehat b_h.
\end{equation}
Let \(W_h\in\R^{M_h\times M_h}\) be symmetric positive definite and define the residual inner product
\begin{equation}
    \langle r,s\rangle_{W_h}:=r^\top W_h s,
    \qquad
    \norm{r}_{W_h}:=(r^\top W_h r)^{1/2}.
\end{equation}
Set
\begin{equation}
    \widehat H_h:=\widehat A_h^\top W_h\widehat A_h.
    \label{eq:reduced_normal_matrix_spd}
\end{equation}

\begin{assumption}[Square reduced linear system]
\label{ass:discrete_square_main}
The reduced operator \(\widehat A_h\in\R^{n_h\times n_h}\) is square and invertible. The full-column-rank overdetermined case is treated in \Cref{appsec:linear_rectangular_extension}. Rank-deficient reduced operators are outside the certified scope.
\end{assumption}

Different choices of \(W_h\) specify different residual geometries (Euclidean, quadrature-weighted, or energy-compatible) but in the square reduced setting they do not alter the exact identities below.

Define the discrete numerical target by
\begin{equation}
    y_h^\star := \widehat A_h^{-1}\widehat b_h,
    \qquad
    U_h^\star := U_{0,h}+Z_hy_h^\star.
    \label{eq:admissible_discrete_target}
\end{equation}
The compatible residual of the certified vector is
\begin{equation}
    r_h:=A_h\Utheta-b_h
    =
    \widehat A_hy_\theta-\widehat b_h.
    \label{eq:admissible_pinn_residual}
\end{equation}

\subsection{Exact query identities and sensitivities}
\label{subsec:linear_exact_identities}

For a discrete quantity-of-interest vector \(\ell_h\in\R^{N_h}\), we consider its restriction to the fixed admissible space. Define
\begin{equation}
    \widehat\ell_h:=Z_h^\top\ell_h\in\R^{n_h}.
    \label{eq:reduced_qoi_vector}
\end{equation}
Let \(q_{\ell,h}\in\R^{n_h}\) solve
\begin{equation}
    \widehat H_hq_{\ell,h}=\widehat\ell_h,
    \label{eq:reduced_adjoint_equation}
\end{equation}
and define the residual-side Green vector
\begin{equation}
    g_{\ell,h}:=W_h\widehat A_hq_{\ell,h}\in\R^{M_h}.
    \label{eq:reduced_residual_green_vector}
\end{equation}
Also set
\begin{equation}
    \sigma_{\ell,h}
    :=
    \norm{g_{\ell,h}}_{W_h^{-1}}.
    \label{eq:general_sigma_definition}
\end{equation}
In the square reduced setting, \(g_{\ell,h}=\widehat A_h^{-\top}\widehat\ell_h\), independently of \(W_h\); the appendix gives the overdetermined full-column-rank interpretation.

\begin{proposition}[Exact discrete residual-to-error identity]
\label{prop:admissible_exact_identity}
Assume \Cref{ass:discrete_square_main}. Then
\begin{equation}
    \Utheta-U_h^\star
    =
    Z_h\widehat A_h^{-1}r_h.
    \label{eq:admissible_exact_discrete_identity}
\end{equation}
\end{proposition}

\begin{proposition}[Adjoint Green identity]
\label{prop:admissible_adjoint_green_identity}
Assume \Cref{ass:discrete_square_main}. For every \(\ell_h\in\R^{N_h}\),
\begin{equation}
    \ell_h^\top(\Utheta-U_h^\star)=g_{\ell,h}^\top r_h.
    \label{eq:admissible_pointwise_identity}
\end{equation}
Moreover,
\begin{equation}
    \sigma_{\ell,h}
    =
    \left(
    \widehat\ell_h^\top\widehat H_h^{-1}\widehat\ell_h
    \right)^{1/2}.
    \label{eq:admissible_sigma_identity}
\end{equation}
Consequently,
\begin{equation}
    \abs{\ell_h^\top(\Utheta-U_h^\star)}
    \le
    \sigma_{\ell,h}\norm{r_h}_{W_h}.
    \label{eq:admissible_pointwise_bound}
\end{equation}
\end{proposition}

For the query set \(\mathcal Q\), define the reduced query quantities by
\begin{equation}
    \begin{aligned}
    \widehat\ell_{i,h}&:=Z_h^\top\ell_{i,h},
    &q_{i,h}&:=q_{\ell_{i,h},h},\\
    g_{i,h}&:=g_{\ell_{i,h},h},
    &\sigma_{i,h}&:=\sigma_{\ell_{i,h},h}.
    \end{aligned}
    \label{eq:query_green_quantities}
\end{equation}
Equivalently,
\begin{equation}
    \sigma_{i,h}
    =
    \left(
    \widehat\ell_{i,h}^{\top}\widehat H_h^{-1}\widehat\ell_{i,h}
    \right)^{1/2}.
    \label{eq:query_sigma_reduced}
\end{equation}

\subsection{Transfer to the continuous target}
\label{subsec:linear_transfer}

To transfer the discrete certificate to the continuous PDE solution, choose any admissible comparison vector
\begin{equation}
    U_h^\pi = U_{0,h}+Z_h y_h^\pi \in \Uhadm.
    \label{eq:admissible_comparison_vector}
\end{equation}
This vector need not arise from a canonical coefficient map of \(\ustar\). It may be an interpolant, a constrained lift, a projection, or any other admissible numerical representation for which certified comparison estimates are available.

Define its consistency residual by
\begin{equation}
    \tau_h^\pi
    :=
    A_h U_h^\pi - b_h
    =
    \widehat A_h y_h^\pi - \widehat b_h.
    \label{eq:comparison_consistency_residual}
\end{equation}
Assume that certified bounds are available such that
\begin{equation}
    T_h\ge\norm{\tau_h^\pi}_{W_h},
    \label{eq:certified_consistency_bound}
\end{equation}
and
\begin{equation}
    \varepsilon_{i,h}^{\mathrm{eval}}
    \ge
    \abs{\ell_{i,h}^{\top}U_h^\pi-\ustar(x_i)}.
    \label{eq:certified_eval_error}
\end{equation}
Define
\begin{equation}
    \mu_{i,h}:=\sigma_{i,h}T_h+\varepsilon_{i,h}^{\mathrm{eval}}.
    \label{eq:admissible_mu}
\end{equation}

\begin{proposition}[Exact comparison identity and discrete-to-continuous bias]
\label{prop:admissible_discretization_bias}
Assume \Cref{ass:discrete_square_main}. For every admissible comparison vector \(U_h^\pi=U_{0,h}+Z_h y_h^\pi\),
\begin{equation}
    \ell_{i,h}^{\top}U_h^\star-\ustar(x_i)
    =
    -g_{i,h}^{\top}\tau_h^\pi
    +
    \bigl(\ell_{i,h}^{\top}U_h^\pi-\ustar(x_i)\bigr).
    \label{eq:exact_comparison_bias_identity}
\end{equation}
Consequently,
\begin{equation}
    \abs{\ell_{i,h}^{\top}(U_h^\star-U_h^\pi)}
    =
    \abs{g_{i,h}^{\top}\tau_h^\pi}
    \le
    \sigma_{i,h}\norm{\tau_h^\pi}_{W_h},
    \label{eq:sharp_comparison_bias}
\end{equation}
and, if \eqref{eq:certified_consistency_bound}--\eqref{eq:certified_eval_error} hold,
\begin{equation}
    \begin{multlined}
    \abs{\ell_{i,h}^{\top}U_h^\star-\ustar(x_i)}
    \le
    \abs{g_{i,h}^{\top}\tau_h^\pi}
    +
    \varepsilon_{i,h}^{\mathrm{eval}}\\
    \le
    \mu_{i,h}.
    \end{multlined}
    \label{eq:admissible_discretization_bias_bound}
\end{equation}
\end{proposition}

\subsection{Deterministic pointwise intervals}
\label{subsec:deterministic_intervals}

\begin{theorem}[Linear deterministic certificate for a compatible point query]
\label{thm:linear_deterministic_interval}
Assume \Cref{ass:discrete_square_main}, \eqref{eq:certified_consistency_bound}, and \eqref{eq:certified_eval_error}. Let \(r_h=A_h\Utheta-b_h\), and define the corrected center
\begin{equation}
    c_{i,h}^{\mathrm{corr}}
    :=
    \ell_{i,h}^{\top}\Utheta-g_{i,h}^{\top}r_h.
    \label{eq:corrected_center}
\end{equation}
Then
\begin{equation}
    \ustar(x_i)\in
    \left[
    c_{i,h}^{\mathrm{corr}}-\mu_{i,h},
    c_{i,h}^{\mathrm{corr}}+\mu_{i,h}
    \right].
    \label{eq:corrected_interval}
\end{equation}
Equivalently, the interval centered at the numerical representation is
\begin{equation}
    \begin{multlined}
    \ustar(x_i)\in\\
    \left[
    \ell_{i,h}^{\top}\Utheta-\rho_{i,h}^{\mathrm{pred}},
    \ell_{i,h}^{\top}\Utheta+\rho_{i,h}^{\mathrm{pred}}
    \right],\\
    \rho_{i,h}^{\mathrm{pred}}
    :=
    \abs{g_{i,h}^{\top}r_h}+\mu_{i,h}.
    \end{multlined}
    \label{eq:pinn_representation_interval}
\end{equation}
If the final reported center is an unprojected neural output \(\utheta(x_i)\), define
\begin{equation}
    \rho_{i,h}^{\theta}
    :=
    \abs{\utheta(x_i)-c_{i,h}^{\mathrm{corr}}}+\mu_{i,h}.
    \label{eq:actual_pinn_output_radius}
\end{equation}
Then
\begin{equation}
    \ustar(x_i)\in
    \left[
    \utheta(x_i)-\rho_{i,h}^{\theta},
    \utheta(x_i)+\rho_{i,h}^{\theta}
    \right].
    \label{eq:actual_pinn_output_interval}
\end{equation}
Moreover,
\begin{equation}
    \rho_{i,h}^{\theta}
    \le
    \zeta_{i,h}^{\theta}
    +
    \rho_{i,h}^{\mathrm{pred}}.
\end{equation}
\end{theorem}

\begin{remark}[Immediate \(P_1\) pointwise special cases]
\label{rem:p1_special_cases_after_linear_theorem}
In the mesh-based \(P_1\) setting, the theorem is applied with \(\ell_{i,h}=\ell_h(x_i)\) from \eqref{eq:p1_query_vector}. If \(x_i=a_k\) is a mesh node, then \(\ell_{i,h}=e_k\) and the certificate is a nodal-value certificate. If \(x_i\) lies inside an element, then \(\ell_{i,h}\) is the barycentric-coordinate vector on that element and the same theorem gives the off-grid point certificate. Thus nodal and off-grid point values are direct instances of the same compatible-query algebra with different choices of \(\ell_{i,h}\).
\end{remark}

\begin{proposition}[Exact correction calibrates the center]
\label{prop:exact_correction_recovers_discrete_target}
Assume \Cref{ass:discrete_square_main}. Let \(e_\theta\in\R^{n_h}\) solve
\begin{equation}
    \widehat A_h e_\theta = r_h.
\end{equation}
Then
\begin{equation}
    \Utheta-Z_h e_\theta=U_h^\star.
\end{equation}
Consequently,
\begin{equation}
    c_{i,h}^{\mathrm{corr}} = \ell_{i,h}^{\top}U_h^\star.
\end{equation}
\end{proposition}

\begin{remark}[Interpretation of exact linear correction]
\label{rem:pinn_role_linear_exact_correction}
For a linear compatible residual system, exact correction centers the interval at the compatible discrete target of the chosen square reduced system:
\[
    c_{i,h}^{\mathrm{corr}}
    =
    \ell_{i,h}^\top U_h^\star.
\]
In that regime the certificate calibrates the trained field to a specified numerical solution. The trained admissible state remains useful when exact correction is avoided or when certification is reused: norm-only certification, early-stopped or inexact correction, localized or randomized certificates, nonlinear certification around a local base point, or many-instance amortized evaluation. In all of those regimes the model supplies the admissible coefficient vector whose compatible residual is being certified.
\end{remark}

\begin{remark}[What remains non-computable after the discrete correction]
\label{rem:a_priori_error}
For the linear certificate, the discrete quantities \((g_{i,h})\), \((\sigma_{i,h})\), \(g_{i,h}^{\top}r_h\), and \(c_{i,h}^{\mathrm{corr}}\) are computable from the compatible discrete system. The non-computable part of the final interval is the transfer layer, namely \(T_h\) and \(\varepsilon_{i,h}^{\mathrm{eval}}\). The next subsection first gives an a priori specialization for coercive Galerkin problems and then one fully computable a posteriori transfer theorem in a concrete coercive class.
\end{remark}

\section{Linear Variants and Specializations}
\label{sec:linear_refinements}

This section gives computational variants for the square reduced setting: norm-only certificates, a standard coercive Galerkin specialization, localization, randomized sensitivity estimation, and inexact linear algebra. The effective-residual projector and the full-column-rank overdetermined extension are given in \Cref{appsec:linear_rectangular_extension}.

\subsection{Norm-only certificates and width decomposition}
\label{subsec:width_decomposition}

In the square reduced setting,
\begin{equation}
    \abs{g_{i,h}^{\top}r_h}
    \le
    \sigma_{i,h}\norm{r_h}_{W_h}.
    \label{eq:projected_norm_bound}
\end{equation}

Define the three radii
\begin{align}
    \rho_{i,h}^{\mathrm{corr}}
    &:=
    \sigma_{i,h}T_h+\varepsilon_{i,h}^{\mathrm{eval}},
    \label{eq:corrected_radius_decomposition}\\
    \rho_{i,h}^{\mathrm{pred}}
    &:=
    \abs{g_{i,h}^{\top}r_h}
    +
    \sigma_{i,h}T_h
    +
    \varepsilon_{i,h}^{\mathrm{eval}},
    \label{eq:p1_radius_decomposition}\\
    \rho_{i,h}^{\mathrm{norm}}
    &:=
    \sigma_{i,h}\norm{r_h}_{W_h}
    +
    \sigma_{i,h}T_h
    +
    \varepsilon_{i,h}^{\mathrm{eval}}.
    \label{eq:norm_radius_decomposition}
\end{align}
Then
\begin{align}
    \abs{c_{i,h}^{\mathrm{corr}}-\ustar(x_i)}
    &\le
    \rho_{i,h}^{\mathrm{corr}},
    \label{eq:corrected_radius_bound}\\
    \abs{\ell_{i,h}^{\top}\Utheta-\ustar(x_i)}
    &\le
    \rho_{i,h}^{\mathrm{pred}}
    \le
    \rho_{i,h}^{\mathrm{norm}},
    \label{eq:prediction_centered_radius_chain}\\
    \abs{\utheta(x_i)-\ustar(x_i)}
    &\le
    \zeta_{i,h}^{\theta}+\rho_{i,h}^{\mathrm{norm}}.
    \label{eq:raw_output_radius_chain}
\end{align}
If only an a priori transfer rate \(T_h\le C_\tau h^\alpha\) is available, it can be directly substituted into \eqref{eq:corrected_radius_decomposition}--\eqref{eq:norm_radius_decomposition}.

The width has three mathematically distinct sources:
\begin{equation}
    \begin{multlined}
    \rho_{i,h}^{\mathrm{norm}}
    =
    \underbrace{\sigma_{i,h}\norm{r_h}_{W_h}}_{\text{discrete compatible error bound}}\\
    +
    \underbrace{\sigma_{i,h}T_h}_{\text{discrete-to-continuous transfer}}
    +
    \underbrace{\varepsilon_{i,h}^{\mathrm{eval}}}_{\text{evaluation comparison}}.
    \end{multlined}
    \label{eq:norm_radius_source_split}
\end{equation}
The width also includes the additional term
\(
    \zeta_{i,h}^{\theta}
\)
only when the reported center is an unprojected neural output rather than the certified reconstruction. For the signed prediction-centered interval, replace the first term in \eqref{eq:norm_radius_source_split} by \(\abs{g_{i,h}^{\top}r_h}\).

A large interval radius may result from large sensitivity \(\sigma_{i,h}\), large discrete compatible error, large transfer residual \(T_h\), or large representation/evaluation mismatch. These contributions should be reported separately. The sharper overdetermined refinement based on the effective-residual projector \(P_h^{\mathrm{eff}}\) is given in \Cref{appsec:linear_rectangular_extension}.

\subsection{Finite-element application for coercive Dirichlet elliptic problems}
\label{subsec:coercive_galerkin_specialization}

The abstract transfer term becomes explicit after \(T_h\) and \(\varepsilon_{i,h}^{\mathrm{eval}}\) are instantiated for a PDE class. The finite-element ingredients used below are standard conforming Galerkin theory, Céa-type quasi-optimality, and \(P_1\) interpolation estimates on shape-regular meshes \cite{BrennerScott2008FEM}. Classical pointwise finite-element estimates and goal-oriented a posteriori estimators may supply pointwise transfer bounds \cite{RannacherScott1982OptimalPointwise,PrudhommeOden1999GoalOrientedPointwise,OdenPrudhomme2001GoalOriented,BeckerRannacher2001OptimalControlAPosteriori}. A standard linear specialization is the conforming \(P_1\) Galerkin discretization of a coercive second-order Dirichlet problem.\newline

Assume \(V=H_0^1(\Omega)\) and that the continuous problem is
\begin{equation}
    a(u^\star,v)=L(v)
    \qquad\text{for all }v\in V,
    \label{eq:coercive_dirichlet_weak_form}
\end{equation}
with
\begin{equation}
    a(u,v)
    :=
    \int_\Omega \mathbf A(x)\nabla u\cdot\nabla v + c(x)uv\,\dd x,
    \label{eq:coercive_bilinear_form}
\end{equation}
where \(\mathbf A\in L^\infty(\Omega;\R_{\mathrm{sym}}^{d\times d})\) is uniformly elliptic and \(c\in L^\infty(\Omega)\) satisfies \(c\ge 0\). Assume
\begin{equation}
    \alpha_a \norm{v}_{H_0^1(\Omega)}^2
    \le
    a(v,v),
    \qquad
    \abs{a(u,v)}
    \le
    M_a \norm{u}_{H_0^1(\Omega)}\norm{v}_{H_0^1(\Omega)}.
\end{equation}
Let \(u_h^{\mathrm{FE}}\in V_h\) denote the conforming \(P_1\) Galerkin solution satisfying
\begin{equation}
    a(u_h^{\mathrm{FE}},v_h)=L(v_h)
    \qquad\text{for all }v_h\in V_h.
    \label{eq:galerkin_fe_problem}
\end{equation}
Write
\[
    u_h^{\mathrm{FE}}=\Ih U_h^{\mathrm{FE}}
    =
    \Ih(U_{0,h}+Z_h y_h^{\mathrm{FE}}).
\]
After elimination of essential constraints, the reduced Galerkin matrix \(\widehat A_h\) is square and symmetric positive definite. Set
\begin{equation}
    \norm{v}_{a}:=a(v,v)^{1/2}.
\end{equation}

\begin{proposition}[Exact correction equals the Galerkin solution]
\label{prop:coercive_galerkin_transfer}
In the square coercive Galerkin setting above,
\begin{equation}
    U_h^\star=U_h^{\mathrm{FE}},
    \qquad
    c_{i,h}^{\mathrm{corr}}=u_h^{\mathrm{FE}}(x_i).
\end{equation}
Hence any rigorous pointwise finite-element bound
\begin{equation}
    \varepsilon_{i,h}^{\mathrm{FE}}
    \ge
    \abs{u_h^{\mathrm{FE}}(x_i)-u^\star(x_i)}
\end{equation}
immediately yields the corrected certificate
\begin{equation}
    u^\star(x_i)
    \in
    \left[
    c_{i,h}^{\mathrm{corr}}-\varepsilon_{i,h}^{\mathrm{FE}},
    \;
    c_{i,h}^{\mathrm{corr}}+\varepsilon_{i,h}^{\mathrm{FE}}
    \right].
    \label{eq:coercive_galerkin_corrected_interval}
\end{equation}
\end{proposition}

Thus, in the exact-correction regime, the remaining continuous-transfer task is the pointwise finite-element error problem for the chosen PDE class. The certificate uses such bounds as external transfer inputs and does not alter the finite-dimensional discrete correction.

Let \(\mathcal J_h:C^0(\OmegaBar)\to V_h\) denote the nodal interpolation operator and set
\[
    u_h^\pi:=\mathcal J_h u^\star,
    \qquad
    U_h^\pi \text{ the coefficient vector of }u_h^\pi.
\]

The next bound is an a priori instantiation of the transfer interface obtained by combining the comparison identity above with the standard finite-element estimates cited at the start of this subsection.

\begin{theorem}[Interpolation-based explicit transfer bound]
\label{thm:interpolation_transfer_bound}
Assume that the mesh family is quasi-uniform, that \(u^\star\in H^2(\Omega)\), and that for each query point \(x_i\) one has \(u^\star\in W^{2,\infty}(\omega_i)\), where \(\omega_i\) is the nodal patch of any element containing \(x_i\). Assume also that \(W_h=\widehat A_h^{-1}\). Then
\begin{equation}
    T_h
    :=
    \norm{\tau_h^\pi}_{W_h}
    \le
    C_{\mathrm{int}} h \abs{u^\star}_{H^2(\Omega)},
    \label{eq:explicit_transfer_residual_bound}
\end{equation}
where \(C_{\mathrm{int}}>0\) depends only on \(\alpha_a\), \(M_a\), and the shape regularity of the mesh family. Moreover, for every query point \(x_i\),
\begin{equation}
    \varepsilon_{i,h}^{\mathrm{eval}}
    :=
    \abs{u_h^\pi(x_i)-u^\star(x_i)}
    \le
    C_{\mathrm{pt}} h^2 \abs{u^\star}_{W^{2,\infty}(\omega_i)}.
    \label{eq:explicit_transfer_eval_bound}
\end{equation}
If \(x_i\) is a mesh node, then \(\varepsilon_{i,h}^{\mathrm{eval}}=0\). Consequently,
\begin{equation}
    \begin{multlined}
    \abs{\ell_{i,h}^\top U_h^\star-u^\star(x_i)}
    \le
    C_{\mathrm{int}}\sigma_{i,h} h \abs{u^\star}_{H^2(\Omega)}\\
    +
    C_{\mathrm{pt}} h^2 \abs{u^\star}_{W^{2,\infty}(\omega_i)}.
    \end{multlined}
    \label{eq:coercive_transfer_split}
\end{equation}
\end{theorem}

\begin{corollary}[A priori certificate decomposition for coercive Galerkin problems]
\label{cor:coercive_galerkin_certificate_split}
Under the hypotheses of \Cref{thm:interpolation_transfer_bound},
\begin{equation}
    \abs{\uhtheta(x_i)-u^\star(x_i)}
    \le
    \abs{g_{i,h}^\top r_h}
    +
    C_{\mathrm{int}}\sigma_{i,h} h \abs{u^\star}_{H^2(\Omega)}
    +
    C_{\mathrm{pt}} h^2 \abs{u^\star}_{W^{2,\infty}(\omega_i)},
    \label{eq:coercive_signed_split}
\end{equation}
and, using only residual norms,
\begin{equation}
    \begin{multlined}
    \abs{\uhtheta(x_i)-u^\star(x_i)}
    \le
    \sigma_{i,h}\norm{r_h}_{W_h}
    +
    C_{\mathrm{int}}\sigma_{i,h} h \abs{u^\star}_{H^2(\Omega)}\\
    +
    C_{\mathrm{pt}} h^2 \abs{u^\star}_{W^{2,\infty}(\omega_i)}.
    \end{multlined}
    \label{eq:coercive_norm_split}
\end{equation}
If the interval is centered at the unprojected neural output \(u_\theta(x_i)\), add \(\zeta_{i,h}^{\theta}\).
\end{corollary}

\begin{remark}[Point queries in three dimensions]
\label{rem:3d_point_queries}
For point queries in \(d\ge 2\), the factor \(\sigma_{i,h}\) in \eqref{eq:coercive_transfer_split} and \eqref{eq:coercive_norm_split} may grow under refinement. The signed term \(\abs{g_{i,h}^{\top}\tau_h^\pi}\) and the norm route \(\sigma_{i,h}T_h\) should therefore be interpreted separately. The refinement study in \Cref{subsec:exp-3d-pointwise} shows this separation directly for point queries in three dimensions: the signed term tracks the pointwise bias, while the norm route is wider. Fixed-radius averages remain better behaved on the same meshes and are the reported quantities in the large-scale SimJEB study.
\end{remark}

\begin{remark}[A priori versus fully computable transfer]
\label{rem:apriori_transfer_computability}
\Cref{thm:interpolation_transfer_bound,cor:coercive_galerkin_certificate_split} are rigorous a priori specializations. For a fully computable numerical certificate, the Sobolev seminorms of \(u^\star\) appearing there must themselves be upper-bounded by verified regularity estimates or replaced by rigorous a posteriori transfer bounds. The next theorem gives such a fully computable a posteriori instantiation in a concrete coercive class.
\end{remark}

\begin{theorem}[Fully computable a posteriori transfer for one-dimensional coercive reaction--diffusion]
\label{thm:1d_fully_computable_transfer}
Assume \(\Omega=(a,b)\subset\R\), \(V=H_0^1(a,b)\), and
\begin{equation}
    a(u,v)
    :=
    \int_a^b \kappa(x)u'(x)v'(x)+c(x)u(x)v(x)\,\dd x,
\end{equation}
where \(\kappa\in L^\infty(a,b)\) is piecewise constant on \(\Th\), with
\[
    \kappa|_K=\kappa_K>0
    \qquad\text{for every }K\in\Th,
\]
and \(c\in L^\infty(a,b)\) satisfies \(c\ge 0\). Let \(u^\star\in V\) solve
\begin{equation}
    a(u^\star,v)
    =
    \int_a^b f(x)v(x)\,\dd x
    \qquad\text{for all }v\in V,
\end{equation}
with \(f\in L^2(a,b)\). Let \(u_h^{\mathrm{FE}}\in V_h\) be the conforming \(P_1\) Galerkin solution satisfying
\begin{equation}
    a(u_h^{\mathrm{FE}},v_h)
    =
    \int_a^b f(x)v_h(x)\,\dd x
    \qquad\text{for all }v_h\in V_h.
\end{equation}
For each element \(K\in\Th\), write \(h_K:=|K|\) and define
\begin{equation}
    R_K
    :=
    f-c\,u_h^{\mathrm{FE}}
    \qquad\text{on }K,
\end{equation}
\begin{equation}
    \eta_{K,h}^{\mathrm{tr}}
    :=
    \frac{h_K}{\pi\sqrt{\kappa_K}}
    \norm{R_K}_{L^2(K)},
    \qquad
    \eta_h^{\mathrm{tr}}
    :=
    \left(
    \sum_{K\in\Th}
    \bigl(\eta_{K,h}^{\mathrm{tr}}\bigr)^2
    \right)^{1/2}.
\end{equation}
Then
\begin{equation}
    \norm{u^\star-u_h^{\mathrm{FE}}}_{a}
    \le
    \eta_h^{\mathrm{tr}}.
\end{equation}
Moreover, for every query point \(x_i\in[a,b]\),
\begin{equation}
    C_{i,h}^{\mathrm{pt}}
    :=
    \min\left\{
    \left(\int_a^{x_i}\kappa(x)^{-1}\,\dd x\right)^{1/2},
    \left(\int_{x_i}^{b}\kappa(x)^{-1}\,\dd x\right)^{1/2}
    \right\}
\end{equation}
satisfies
\begin{equation}
    \abs{u^\star(x_i)-u_h^{\mathrm{FE}}(x_i)}
    \le
    C_{i,h}^{\mathrm{pt}}\,\eta_h^{\mathrm{tr}}.
\end{equation}
In the notation of \Cref{subsec:linear_transfer}, one may take
\begin{equation}
    U_h^\pi = U_h^{\mathrm{FE}} = U_h^\star,
    \qquad
    T_h = 0,
    \qquad
    \varepsilon_{i,h}^{\mathrm{eval}}
    =
    C_{i,h}^{\mathrm{pt}}\,\eta_h^{\mathrm{tr}}.
    \label{eq:1d_computable_transfer_terms}
\end{equation}
Consequently,
\begin{equation}
    u^\star(x_i)\in
    \left[
    c_{i,h}^{\mathrm{corr}}-C_{i,h}^{\mathrm{pt}}\,\eta_h^{\mathrm{tr}},
    \;
    c_{i,h}^{\mathrm{corr}}+C_{i,h}^{\mathrm{pt}}\,\eta_h^{\mathrm{tr}}
    \right].
    \label{eq:1d_computable_transfer_interval}
\end{equation}
\end{theorem}

Asymptotic contraction criteria are presented in \Cref{appsec:linear_auxiliary_material}. They aid interpretation of the radii but are not needed for the main interval construction.

\subsection{Localized residual certificates}
\label{subsec:localized_linear_certificates}

Localization is query-dependent. In the square reduced setting, the exact discrete error is represented directly by \(r_h\), so the local/tail split is written in terms of the full compatible residual.

Fix a query point \(x_i\) and a radius parameter \(R\). Let
\[
    \Pi_{i,R}:\R^{M_h}\to\R^{M_h},
    \qquad
    \Pi_{i,R}^2=\Pi_{i,R},
\]
be any linear projector whose range is the subspace of residual rows designated as local to \(x_i\) inside radius \(R\). Define
\begin{equation}
    g_{i,h}^{\mathrm{loc}}(R):=\Pi_{i,R}^\top g_{i,h},
    \qquad
    g_{i,h}^{\mathrm{tail}}(R):=(I-\Pi_{i,R})^\top g_{i,h},
\end{equation}
and
\begin{equation}
    r_h^{\mathrm{loc}}(R):=\Pi_{i,R}r_h,
    \qquad
    r_h^{\mathrm{tail}}(R):=(I-\Pi_{i,R})r_h.
\end{equation}

\begin{proposition}[Projector-based local/tail certificate]
\label{prop:local_tail_projector_certificate}
For every query \(x_i\) and every projector \(\Pi_{i,R}\),
\begin{equation}
    \ell_{i,h}^{\top}(\Utheta-U_h^\star)
    =
    \bigl(g_{i,h}^{\mathrm{loc}}(R)\bigr)^\top r_h^{\mathrm{loc}}(R)
    +
    \bigl(g_{i,h}^{\mathrm{tail}}(R)\bigr)^\top r_h^{\mathrm{tail}}(R).
    \label{eq:exact_local_tail_identity}
\end{equation}
Consequently,
\begin{equation}
    \begin{multlined}
    \abs{\ell_{i,h}^{\top}(\Utheta-U_h^\star)}
    \le
    \abs{\bigl(g_{i,h}^{\mathrm{loc}}(R)\bigr)^\top r_h^{\mathrm{loc}}(R)}\\
    +
    \norm{g_{i,h}^{\mathrm{tail}}(R)}_{W_h^{-1}}
    \norm{r_h^{\mathrm{tail}}(R)}_{W_h}.
    \end{multlined}
    \label{eq:certified_local_green_tail_bound}
\end{equation}
A pure norm split gives
\begin{equation}
    \begin{multlined}
    \abs{\ell_{i,h}^{\top}(\Utheta-U_h^\star)}
    \le
    \norm{g_{i,h}^{\mathrm{loc}}(R)}_{W_h^{-1}}
    \norm{r_h^{\mathrm{loc}}(R)}_{W_h}\\
    +
    \norm{g_{i,h}^{\mathrm{tail}}(R)}_{W_h^{-1}}
    \norm{r_h^{\mathrm{tail}}(R)}_{W_h}.
    \end{multlined}
    \label{eq:pure_norm_local_green_tail_bound}
\end{equation}
\end{proposition}

Define the local and tail sensitivity masses by
\begin{equation}
    \chi_{i,h}(R):=\norm{g_{i,h}^{\mathrm{loc}}(R)}_{W_h^{-1}},
    \qquad
    \vartheta_{i,h}(R):=\norm{g_{i,h}^{\mathrm{tail}}(R)}_{W_h^{-1}}.
    \label{eq:local_tail_sensitivity_masses}
\end{equation}
The quantity \(\chi_{i,h}(R)\) is a rigorous sensitivity index for the local residual block. The product \(\vartheta_{i,h}(R)\norm{r_h^{\mathrm{tail}}(R)}_{W_h}\) is the required deterministic tail contribution. Omitting the tail term removes the deterministic certificate and leaves only a heuristic indicator.

If \(\Pi_{i,R}\) is \(W_h\)-orthogonal, for example a coordinate selector when \(W_h\) is diagonal or block diagonal by residual cell, then
\begin{equation}
    \sigma_{i,h}^2
    =
    \chi_{i,h}(R)^2+\vartheta_{i,h}(R)^2.
    \label{eq:orthogonal_local_tail_sigma_split}
\end{equation}
Indeed, the cross term vanishes because
\[
    \bigl(g_{i,h}^{\mathrm{loc}}(R)\bigr)^\top W_h^{-1} g_{i,h}^{\mathrm{tail}}(R)=0
\]
whenever \(\Pi_{i,R}^\top W_h=W_h\Pi_{i,R}\).

\begin{remark}[Deterministic localization and statistical post-processing]
\label{rem:localized_conformal_overlay}
The deterministic certificate is \eqref{eq:certified_local_green_tail_bound} together with the continuous-transfer terms. The local sensitivity mass \(\chi_{i,h}(R)\), the tail mass \(\vartheta_{i,h}(R)\), the signed local correction \(\bigl(g_{i,h}^{\mathrm{loc}}(R)\bigr)^\top r_h^{\mathrm{loc}}(R)\), and the tail product \(\vartheta_{i,h}(R)\norm{r_h^{\mathrm{tail}}(R)}_{W_h}\) are useful features for adaptivity and, when an exchangeable family of PDE instances is available, for statistical post-processing such as conformal calibration. Such conformal post-processing can provide population-level empirical coverage over the calibration distribution, but it does not replace the deterministic single-instance certificate proved here.
\end{remark}

The full-column-rank overdetermined version, where the exact discrete error depends on \(P_h^{\mathrm{eff}}r_h\), is given in \Cref{appsec:linear_rectangular_extension}.

\subsection{Randomized sensitivity estimation}
\label{subsec:randomized_sensitivity}

When the number of queries is large, exact computation of all \(\sigma_{i,h}\) may dominate cost. Random probing for diagonal and quadratic-form estimation is standard in randomized numerical linear algebra \cite{HallmanIpsenSaibaba2023Diagonal}. Here Gaussian probes are used because they give an exact \(\chi^2\) calibration for the scalar quadratic forms below. Gaussian probing estimates
\begin{equation}
    \sigma_{i,h}^2
    =
    \widehat\ell_{i,h}^{\top}\widehat H_h^{-1}\widehat\ell_{i,h}.
\end{equation}
Let \(W_h^{1/2}\) denote the symmetric positive-definite square root of \(W_h\). Let \(z^{(t)}\sim\mathcal{N}(0,I_{M_h})\) be independent, and solve
\begin{equation}
    \widehat H_h \psi^{(t)}
    =
    \widehat A_h^\top W_h^{1/2}z^{(t)},
    \qquad
    t=1,\dots,m.
    \label{eq:reduced_random_probe_solve}
\end{equation}
Then
\begin{equation}
    \Cov(\psi^{(t)})=\widehat H_h^{-1},
\end{equation}
and therefore
\begin{equation}
    \widehat\ell_{i,h}^{\top}\psi^{(t)}
    \sim
    \mathcal{N}(0,\sigma_{i,h}^2).
\end{equation}
Define
\begin{equation}
    \widehat\sigma_{i,h}^{\,2}
    :=
    \frac1m
    \sum_{t=1}^m
    \left(\widehat\ell_{i,h}^{\top}\psi^{(t)}\right)^2.
    \label{eq:reduced_random_sigma_estimator}
\end{equation}
When \(\sigma_{i,h}>0\),
\begin{equation}
    \frac{m\widehat\sigma_{i,h}^{\,2}}{\sigma_{i,h}^2}
    \sim
    \chi_m^2.
\end{equation}
The zero-sensitivity case is immediate.

Let \(q_{\alpha,m}\) be the lower \(\alpha\)-quantile of \(\chi_m^2\):
\begin{equation}
    \Prob\{\chi_m^2\le q_{\alpha,m}\}=\alpha.
\end{equation}
For \(K\) query points, define
\begin{equation}
    \widehat\sigma_{i,h}^{+}
    :=
    \left(
    \frac{m\widehat\sigma_{i,h}^{\,2}}{q_{\delta/K,m}}
    \right)^{1/2}.
    \label{eq:simultaneous_sigma_upper_bound}
\end{equation}
By the union bound,
\begin{equation}
    \Prob
    \left\{
    \sigma_{i,h}\le\widehat\sigma_{i,h}^{+}
    \ \text{for all }i=1,\dots,K
    \right\}
    \ge
    1-\delta.
    \label{eq:simultaneous_sigma_event}
\end{equation}

For sharp simultaneous intervals, we compute the exact discrete correction once by solving
\begin{equation}
    \widehat A_h e_\theta = r_h.
    \label{eq:one_primal_correction_solve}
\end{equation}
Then
\begin{equation}
    \Utheta-U_h^\star=Z_he_\theta,
\end{equation}
and
\begin{equation}
    c_{i,h}^{\mathrm{corr}}
    =
    \ell_{i,h}^{\top}(\Utheta-Z_he_\theta).
    \label{eq:randomized_corrected_center}
\end{equation}
On the event \eqref{eq:simultaneous_sigma_event}, the intervals
\begin{equation}
    \begin{multlined}
    I_i^{\mathrm{corr,rand}}
    :=\\
    \left[
    c_{i,h}^{\mathrm{corr}}
    -
    \widehat\sigma_{i,h}^{+}T_h
    -
    \varepsilon_{i,h}^{\mathrm{eval}},\right.\\
    \left.
    c_{i,h}^{\mathrm{corr}}
    +
    \widehat\sigma_{i,h}^{+}T_h
    +
    \varepsilon_{i,h}^{\mathrm{eval}}
    \right]
    \end{multlined}
    \label{eq:randomized_corrected_interval}
\end{equation}
satisfy
\begin{equation}
    \Prob
    \left\{
    \ustar(x_i)\in I_i^{\mathrm{corr,rand}}
    \ \text{for all }i=1,\dots,K
    \right\}
    \ge
    1-\delta.
\end{equation}

Without the correction solve \eqref{eq:one_primal_correction_solve}, the looser representation-centered interval is
\begin{equation}
    \begin{multlined}
    I_i^{\mathrm{norm,rand}}
    :=\\
    \left[
    \ell_{i,h}^{\top}\Utheta
    -
    \widehat\sigma_{i,h}^{+}
    \bigl(\norm{r_h}_{W_h}+T_h\bigr)
    -
    \varepsilon_{i,h}^{\mathrm{eval}},\right.\\
    \left.
    \ell_{i,h}^{\top}\Utheta
    +
    \widehat\sigma_{i,h}^{+}
    \bigl(\norm{r_h}_{W_h}+T_h\bigr)
    +
    \varepsilon_{i,h}^{\mathrm{eval}}
    \right].
    \end{multlined}
    \label{eq:randomized_norm_interval}
\end{equation}
Rademacher probes remain unbiased for \(\sigma_{i,h}^2\), but the exact \(\chi_m^2\) calibration above no longer applies; separate concentration bounds are required in that setting.

\subsection{Inexact linear algebra}
\label{subsec:inexact_linear_algebra}

Certified allowance formulas for inexact correction are given in \Cref{appsec:linear_auxiliary_material}. In the square reduced setting, primal correction and adjoint solves should use \(\widehat A_h\) and \(\widehat A_h^\top\) directly, not the normal matrix \(\widehat H_h\). The full-column-rank overdetermined extension and its matrix-free least-squares solvers are given in \Cref{appsec:linear_rectangular_extension}.

\section{Nonlinear Certification by Verified Local Roots}
\label{sec:nonlinear}

This section presents the exact-root certificate. A local discrete exact root is certified first; the resulting discrete target is then transferred to the PDE solution. The stationary-target variant is given in \Cref{appsec:nonlinear_stationary_variant}, and one explicit verified-ball construction for square systems is given in \Cref{appsec:linear_auxiliary_material}.

All nonlinear statements are written after admissible constraints have been eliminated. Fix a norm \(\norm{\cdot}_{X_h}\) on \(\R^{n_h}\); all local radii and operator norms below are measured with respect to \(\norm{\cdot}_{X_h}\) on the state space, \(\norm{\cdot}_{W_h}\) on the residual space \(\R^{M_h}\), and the dual norm \(\norm{\cdot}_{X_h^\ast}\) on \((\R^{n_h})^\ast\). The reduced unknown is \(y\in\R^{n_h}\), the corresponding admissible full vector is
\begin{equation}
    U(y):=U_{0,h}+Z_h y,
\end{equation}
and the query map at \(x_i\) is
\begin{equation}
    L_{i,h}(y):=\ell_{i,h}^{\top}U(y)
    =
    \ell_{i,h}^{\top}U_{0,h}+\widehat\ell_{i,h}^{\top}y,
    \qquad
    \widehat\ell_{i,h}:=Z_h^\top \ell_{i,h}.
    \label{eq:nonlinear_reduced_query_map}
\end{equation}
This is the reduced-coordinate form of the compatible query \(\Lambda_{i,h}\).
Let
\begin{equation}
    F_h:\R^{n_h}\to\R^{M_h}
\end{equation}
be continuously differentiable on a neighborhood of the relevant points, and let the trained PINN induce \(y_\theta\in\R^{n_h}\) with residual
\begin{equation}
    r_h:=F_h(y_\theta).
\end{equation}

\subsection{Certified exact-root targets}
\label{subsec:nonlinear_exact_root}

A discrete exact-root target is a vector \(y_h^\star\) satisfying
\begin{equation}
    F_h(y_h^\star)=0.
\end{equation}
When \(M_h>n_h\), full column rank of \(DF_h(y)\) does not certify existence of such a root; exact-root certificates require a separately verified local zero.

For any base point \(y\) at which
\begin{equation}
    J_h(y):=DF_h(y)
\end{equation}
has full column rank, define
\begin{equation}
    H_h^{\mathrm{root}}(y):=J_h(y)^\top W_h J_h(y).
\end{equation}
For a reduced functional \(\widehat\ell_h\in\R^{n_h}\), let \(q_{\ell,h}^{\mathrm{root}}(y)\) solve
\begin{equation}
    H_h^{\mathrm{root}}(y)q_{\ell,h}^{\mathrm{root}}(y)=\widehat\ell_h,
\end{equation}
and set
\begin{equation}
    g_{\ell,h}^{\mathrm{root}}(y):=W_hJ_h(y)q_{\ell,h}^{\mathrm{root}}(y),
    \qquad
    \sigma_{\ell,h}^{\mathrm{root}}(y):=\norm{g_{\ell,h}^{\mathrm{root}}(y)}_{W_h^{-1}}.
\end{equation}
Equivalently,
\begin{equation}
    \sigma_{\ell,h}^{\mathrm{root}}(y)
    =
    \left(
    \widehat\ell_h^\top H_h^{\mathrm{root}}(y)^{-1}\widehat\ell_h
    \right)^{1/2}.
\end{equation}
For a deviation \(d\in\R^{n_h}\), define the base-point Taylor remainder
\begin{equation}
    R_h^{\mathrm{root}}(y;d)
    :=
    F_h(y-d)-F_h(y)+J_h(y)d.
    \label{eq:base_point_root_remainder}
\end{equation}
If \(DF_h\) is Lipschitz on the ball \(B_{X_h}(y,t)\) with constant \(L_h(y,t)\) in the operator norm from \((\R^{n_h},\norm{\cdot}_{X_h})\) to \((\R^{M_h},\norm{\cdot}_{W_h})\), then
\begin{equation}
    \norm{R_h^{\mathrm{root}}(y;d)}_{W_h}
    \le
    \frac{L_h(y,t)}{2}\norm{d}_{X_h}^2,
    \qquad
    \norm{d}_{X_h}\le t,
    \label{eq:root_remainder_lipschitz_bound}
\end{equation}
because
\begin{equation}
    R_h^{\mathrm{root}}(y;d)
    =
    \int_0^1 \bigl(J_h(y)-J_h(y-sd)\bigr)d\,\dd s.
\end{equation}

\begin{proposition}[Local exact-root correction at a base point]
\label{prop:nonlinear_root_base_point}
Let \(y\) be a base point such that \(J_h(y)\) has full column rank. Assume the exact-root target \(y_h^\star\) satisfies
\begin{equation}
    \norm{y-y_h^\star}_{X_h}\le t_h^{\mathrm{root}}(y),
\end{equation}
and that a certified remainder bound
\begin{equation}
    \mathcal{R}_h^{\mathrm{root}}(y;t_h^{\mathrm{root}}(y))
    \ge
    \norm{R_h^{\mathrm{root}}(y;y-y_h^\star)}_{W_h}
    \label{eq:certified_root_remainder_base}
\end{equation}
is available. Then, for every reduced functional \(\widehat\ell_h\in\R^{n_h}\),
\begin{equation}
    \abs{
    \widehat\ell_h^\top (y-y_h^\star)
    -
    \bigl(g_{\ell,h}^{\mathrm{root}}(y)\bigr)^\top F_h(y)
    }
    \le
    \sigma_{\ell,h}^{\mathrm{root}}(y)\,
    \mathcal{R}_h^{\mathrm{root}}(y;t_h^{\mathrm{root}}(y)).
    \label{eq:nonlinear_root_base_point_certificate}
\end{equation}
In particular, for query \(x_i\),
\begin{equation}
    \abs{
    L_{i,h}(y)-L_{i,h}(y_h^\star)
    -
    \bigl(g_{i,h}^{\mathrm{root}}(y)\bigr)^\top F_h(y)
    }
    \le
    \sigma_{i,h}^{\mathrm{root}}(y)\,
    \mathcal{R}_h^{\mathrm{root}}(y;t_h^{\mathrm{root}}(y)),
    \label{eq:local_root_query_certificate}
\end{equation}
where \(g_{i,h}^{\mathrm{root}}(y)\) and \(\sigma_{i,h}^{\mathrm{root}}(y)\) correspond to \(\widehat\ell_{i,h}\).
\end{proposition}

Such neighborhoods can be certified by standard Newton--Kantorovich verification arguments \cite{OrtegaRheinboldt2000,Deuflhard2011Newton}. For square systems, \Cref{appsec:linear_auxiliary_material} records the corresponding verified ball in the norms used here.

\subsection{Discrete-to-continuous comparison for exact-root targets}
\label{subsec:nonlinear_root_transfer}

\begin{proposition}[Exact-root discrete-to-continuous comparison]
\label{prop:nonlinear_root_discrete_to_continuous}
Let \(y_h^\pi\in\R^{n_h}\) be an admissible comparison vector, set
\begin{equation}
    U_h^\pi:=U_{0,h}+Z_h y_h^\pi,
    \qquad
    \tau_h^\pi:=F_h(y_h^\pi),
\end{equation}
and assume
\begin{equation}
    \varepsilon_{i,h}^{\mathrm{eval}}
    \ge
    \abs{L_{i,h}(y_h^\pi)-\ustar(x_i)}.
    \label{eq:nonlinear_root_eval_comparison}
\end{equation}
Assume also that \(J_h(y_h^\pi)\) has full column rank and that the same exact-root target \(y_h^\star\) satisfies
\begin{equation}
    \norm{y_h^\pi-y_h^\star}_{X_h}\le t_h^{\mathrm{root}}(y_h^\pi),
\end{equation}
together with a certified remainder bound
\begin{equation}
    \mathcal{R}_h^{\mathrm{root}}(y_h^\pi;t_h^{\mathrm{root}}(y_h^\pi))
    \ge
    \norm{R_h^{\mathrm{root}}(y_h^\pi;y_h^\pi-y_h^\star)}_{W_h}.
\end{equation}
Then
\begin{equation}
    \abs{L_{i,h}(y_h^\star)-\ustar(x_i)}
    \le
    \mu_{i,h}^{\mathrm{root,disc}},
\end{equation}
where
\begin{equation}
    \mu_{i,h}^{\mathrm{root,disc}}
    :=
    \abs{
    \bigl(g_{i,h}^{\mathrm{root}}(y_h^\pi)\bigr)^\top \tau_h^\pi
    }
    +
    \sigma_{i,h}^{\mathrm{root}}(y_h^\pi)\,
    \mathcal{R}_h^{\mathrm{root}}(y_h^\pi;t_h^{\mathrm{root}}(y_h^\pi))
    +
    \varepsilon_{i,h}^{\mathrm{eval}}.
    \label{eq:nonlinear_root_disc_bias}
\end{equation}
If only a norm bound \(T_h^\pi\ge \norm{\tau_h^\pi}_{W_h}\) is available, then
\begin{equation}
    \mu_{i,h}^{\mathrm{root,disc}}
    \le
    \sigma_{i,h}^{\mathrm{root}}(y_h^\pi)
    \bigl(
    T_h^\pi
    +
    \mathcal{R}_h^{\mathrm{root}}(y_h^\pi;t_h^{\mathrm{root}}(y_h^\pi))
    \bigr)
    +
    \varepsilon_{i,h}^{\mathrm{eval}}.
\end{equation}
\end{proposition}

\subsection{Nonlinear exact-root certificate}
\label{subsec:nonlinear_root_interval}

\begin{theorem}[Nonlinear exact-root pointwise certificate]
\label{thm:nonlinear_root_interval}
Assume the hypotheses of \Cref{prop:nonlinear_root_base_point} at the computable base point \(y_\theta\) and of \Cref{prop:nonlinear_root_discrete_to_continuous} at the comparison point \(y_h^\pi\), with the same exact-root target \(y_h^\star\). Define
\begin{equation}
    c_{i,h}^{\mathrm{root,corr}}
    :=
    L_{i,h}(y_\theta)
    -
    \bigl(g_{i,h}^{\mathrm{root}}(y_\theta)\bigr)^\top r_h.
    \label{eq:nonlinear_root_corrected_center}
\end{equation}
Then
\begin{equation}
    \begin{multlined}
    \abs{c_{i,h}^{\mathrm{root,corr}}-\ustar(x_i)}
    \le
    \sigma_{i,h}^{\mathrm{root}}(y_\theta)\,
    \mathcal{R}_h^{\mathrm{root}}(y_\theta;t_h^{\mathrm{root}}(y_\theta))\\
    +
    \mu_{i,h}^{\mathrm{root,disc}}.
    \end{multlined}
    \label{eq:nonlinear_root_interval_bound}
\end{equation}
Hence
\begin{equation}
    \ustar(x_i)\in
    \left[
    c_{i,h}^{\mathrm{root,corr}}-\rho_{i,h}^{\mathrm{root}},
    \;
    c_{i,h}^{\mathrm{root,corr}}+\rho_{i,h}^{\mathrm{root}}
    \right],
\end{equation}
where
\begin{equation}
    \begin{multlined}
    \rho_{i,h}^{\mathrm{root}}
    :=\\
    \sigma_{i,h}^{\mathrm{root}}(y_\theta)\,
    \mathcal{R}_h^{\mathrm{root}}(y_\theta;t_h^{\mathrm{root}}(y_\theta))\\
    +
    \mu_{i,h}^{\mathrm{root,disc}}.
    \end{multlined}
\end{equation}
Moreover,
\begin{equation}
    \begin{multlined}
    \abs{L_{i,h}(y_\theta)-\ustar(x_i)}
    \le
    \sigma_{i,h}^{\mathrm{root}}(y_\theta)
    \Bigl(
    \norm{r_h}_{W_h}
    +
    \mathcal{R}_h^{\mathrm{root}}(y_\theta;t_h^{\mathrm{root}}(y_\theta))
    \Bigr)\\
    +
    \mu_{i,h}^{\mathrm{root,disc}},
    \end{multlined}
    \label{eq:nonlinear_root_norm_only}
\end{equation}
and therefore
\begin{equation}
    \begin{multlined}
    \abs{u_\theta(x_i)-\ustar(x_i)}
    \le
    \zeta_{i,h}^{\theta}
    +
    \sigma_{i,h}^{\mathrm{root}}(y_\theta)
    \Bigl(
    \norm{r_h}_{W_h}
    +
    \mathcal{R}_h^{\mathrm{root}}(y_\theta;t_h^{\mathrm{root}}(y_\theta))
    \Bigr)\\
    +
    \mu_{i,h}^{\mathrm{root,disc}}.
    \end{multlined}
\end{equation}
\end{theorem}

\begin{remark}[Stationary-target variant]
\label{rem:nonlinear_ls_appendix_bridge}
When a verified discrete exact root is unavailable but a verified local stationary point of the nonlinear least-squares functional can be certified, the same two-stage argument can be applied to the stationarity map
\[
    S_h(y)=DF_h(y)^\top W_hF_h(y).
\]
The stationary-target construction, including the local correction identity and the discrete-to-continuous comparison theorem, is given in \Cref{appsec:nonlinear_stationary_variant}.
\end{remark}

\section{Computational Procedure and Reporting for \texorpdfstring{\(P_1\)}{P1} Numerical PINNs}
\label{sec:computational_procedure}

For the linear compatible \(P_1\) case, the certificate can be assembled as follows.

The following procedure applies to the square reduced setting.

\begin{enumerate}[label=\arabic*.,leftmargin=1.8em]
    \item \textbf{Assemble the compatible system.}
    Build the mesh \(\Th\), the admissible set \(\Uhadm=U_{0,h}+Z_h\R^{n_h}\), the reconstruction \(\Ih\), and the compatible residual \(F_h(U)=A_hU-b_h\). Form
    \[
        \widehat A_h=A_hZ_h,
        \qquad
        \widehat b_h=b_h-A_hU_{0,h},
    \]
    and form \(\widehat H_h=\widehat A_h^\top W_h\widehat A_h\) only when it is specifically needed, for example in the randomized sensitivity procedure.

    \item \textbf{Represent the trained model by an admissible coefficient vector.}
    Obtain \(\Utheta=U_{0,h}+Z_hy_\theta\in\Uhadm\), either directly or through an explicit extraction map \(\Pi_h\). Compute the compatible residual
    \[
        r_h=A_h\Utheta-b_h.
    \]

    \item \textbf{Assemble the query vectors once.}
    For point queries, use \(\ell_h(x)=\bigl(\phi_1(x),\dots,\phi_{N_h}(x)\bigr)^\top\). Nodal values are the special case \(\ell_h(a_k)=e_k\). Off-grid point values use barycentric weights on the containing element. Any other scalar linear functional is treated by assembling its coefficient vector \(\ell_{i,h}\) in the same basis.

    \item \textbf{Choose the discrete correction regime.}
    Use the direct square-system formulas from \Cref{sec:linear_main}. Primal correction and adjoint solves should use \(\widehat A_h\) and \(\widehat A_h^\top\) directly; explicit formation of \(\widehat H_h\) is unnecessary except when the randomized sensitivity procedure is used. The full-column-rank overdetermined extension is given in \Cref{appsec:linear_rectangular_extension}.

    \item \textbf{Supply the transfer layer.}
    Choose an admissible comparison vector \(U_h^\pi\in\Uhadm\) and rigorous bounds \(T_h\) and \(\varepsilon_{i,h}^{\mathrm{eval}}\) as in \Cref{subsec:linear_transfer}. For coercive conforming Galerkin problems, use the specialization in \Cref{subsec:coercive_galerkin_specialization}. In the fully computable one-dimensional reaction--diffusion specialization of \Cref{thm:1d_fully_computable_transfer}, take \(U_h^\pi=U_h^\star=U_h^{\mathrm{FE}}\), so \(T_h=0\) and \(\varepsilon_{i,h}^{\mathrm{eval}}\) is given explicitly by \eqref{eq:1d_computable_transfer_terms}.

    \item \textbf{Report the interval.}
    Use the corrected interval from \Cref{thm:linear_deterministic_interval} when exact correction is available, the randomized intervals from \Cref{subsec:randomized_sensitivity} when many queries are present, the localized form \eqref{eq:certified_local_green_tail_bound} when locality is exploited, or the norm-only radius \eqref{eq:norm_radius_decomposition} when only residual-norm information is desired. Certified inexact-solve allowances are given in \Cref{appsec:linear_auxiliary_material}. Add \(\zeta_{i,h}^{\theta}\) only when the reported center is an unprojected neural output rather than the certified reconstruction.
\end{enumerate}


\section{Experiments}
\label{sec:experiments}

The experiments are organized by claim rather than by benchmark. The objective is to evaluate certification mechanisms rather than prediction error in isolation. Manufactured tests provide direct validation for nodal and off-grid point queries and isolate the refinement behavior of point queries in three dimensions. The SimJEB experiments address whether the compatible-query algebra remains interpretable and computable in a large three-dimensional elasticity benchmark, both for compatible linear states and for raw neural predictions after projection into the admissible space. For that reason, this section reports the SimJEB query geometry, adjoint Green structure, corrected centers, released-field comparison identity, projected-neural raw-output intervals, and randomized sensitivity results; secondary SimJEB diagnostics such as raw-versus-corrected fields, localization sweeps, inexact correction, and additional load-family studies are given in \Cref{app:additional_experiments}. Details on the main claims tested by each experiment are gathered in \Cref{tab:experiment_map}.

All linear experiments in the square-system formulation use square reduced systems after admissible reduction. The full-column-rank overdetermined extension is illustrated in \Cref{appsubsec:exp-rectangular-extension}.

Unless stated otherwise, all reported linear intervals use the compatible residual \(r_h=A_h\Utheta-b_h\), the corresponding query vectors \(\ell_{i,h}\), and the transfer inputs from \Cref{subsec:linear_transfer}. In the SimJEB released-field experiment, the released nodal displacement field is not treated as a certified continuum reference solution. It is used only as an external admissible comparison vector \(U_h^\pi\) on the same mesh, so that the discrete comparison identity from \Cref{prop:admissible_discretization_bias} can be checked directly against a public benchmark field.

\begin{table*}[t]
    \centering
    \caption{Experiment map. Each block corresponds to one theorem-level or implementation claim.}
    \begin{adjustbox}{max width=\textwidth}
    \begin{tabular}{@{}p{0.19\textwidth}p{0.18\textwidth}p{0.19\textwidth}p{0.19\textwidth}p{0.21\textwidth}@{}}
        \toprule
        Claim tested & Theory reference & Benchmark & Metric & Main comparison \\
        \midrule
        Finite collocation does not control unsampled point values & \Cref{prop:finite_collocation_no_point_control} & synthetic hidden-bump test & hidden-point error, certified radii & sampled loss versus hidden-point change \\
        Linear correction calibrates the compatible discrete target & \Cref{thm:linear_deterministic_interval,prop:exact_correction_recovers_discrete_target} & manufactured linear benchmark & querywise error, corrected-center error, width decomposition & uncorrected center, corrected center, compatible discrete target \\
        Point queries in three dimensions & \Cref{prop:admissible_discretization_bias,cor:coercive_galerkin_certificate_split} & manufactured 3D refinement study & sensitivity growth, pointwise bias, signed transfer term, and norm route & nodal point queries versus fixed-radius averages \\
        Large-scale compatible-query correction, sensitivity structure, and transfer & \Cref{prop:admissible_adjoint_green_identity,prop:admissible_discretization_bias} & SimJEB bracket 148 & patch-average displacement error, signed correction, sensitivity concentration, and released-field bias & query geometry, adjoint Green maps, signed correction, and norm comparison bound \\
        Raw-output intervals after admissibility projection & \Cref{thm:linear_deterministic_interval,prop:exact_correction_recovers_discrete_target} & SimJEB bracket 148 load family & coverage, raw error versus raw radius, representation mismatch, and runtime & raw output, corrected center, compatible discrete target, and fresh CG solve \\
        Localization requires an explicit tail term & \Cref{prop:local_tail_projector_certificate} & dumbbell-domain benchmark & local term, tail term, certified local+tail quantity & actual query error across radius sweeps \\
        Randomized sensitivities preserve simultaneous coverage & \Cref{subsec:randomized_sensitivity} & SimJEB bracket and synthetic appendix tests & all-query coverage, inflation, runtime & exact sensitivities from adjoint solves \\
        Nonlinear exact-root certificates produce non-vacuous intervals & \Cref{thm:nonlinear_root_interval} & manufactured nonlinear benchmark & corrected error, norm-only error, remainder split & reported radii versus exact errors \\
        \bottomrule
    \end{tabular}
    \end{adjustbox}
    \label{tab:experiment_map}
\end{table*}

\subsection{Necessity of compatibility}
\label{subsec:exp-compatibility}

\Cref{fig:compatibility-failure} illustrates a failure mode excluded by the theory, and discussed in \Cref{prop:finite_collocation_no_point_control}. A smooth bump is inserted between collocation points while the field remains unchanged at every sampled location. The sampled loss is therefore unchanged, but the hidden query value changes substantially. A compatible certificate reflects this change because the certified object is a mesh-based numerical field whose reported value and residual are defined on the same degrees of freedom. The behavior is consistent with \Cref{prop:finite_collocation_no_point_control}: the sampled loss remains unchanged, while the hidden-point error increases and the compatible mesh and unprojected-output radii reflect the hidden error.

\begin{figure}[t]
    \centering
    \fitfigure{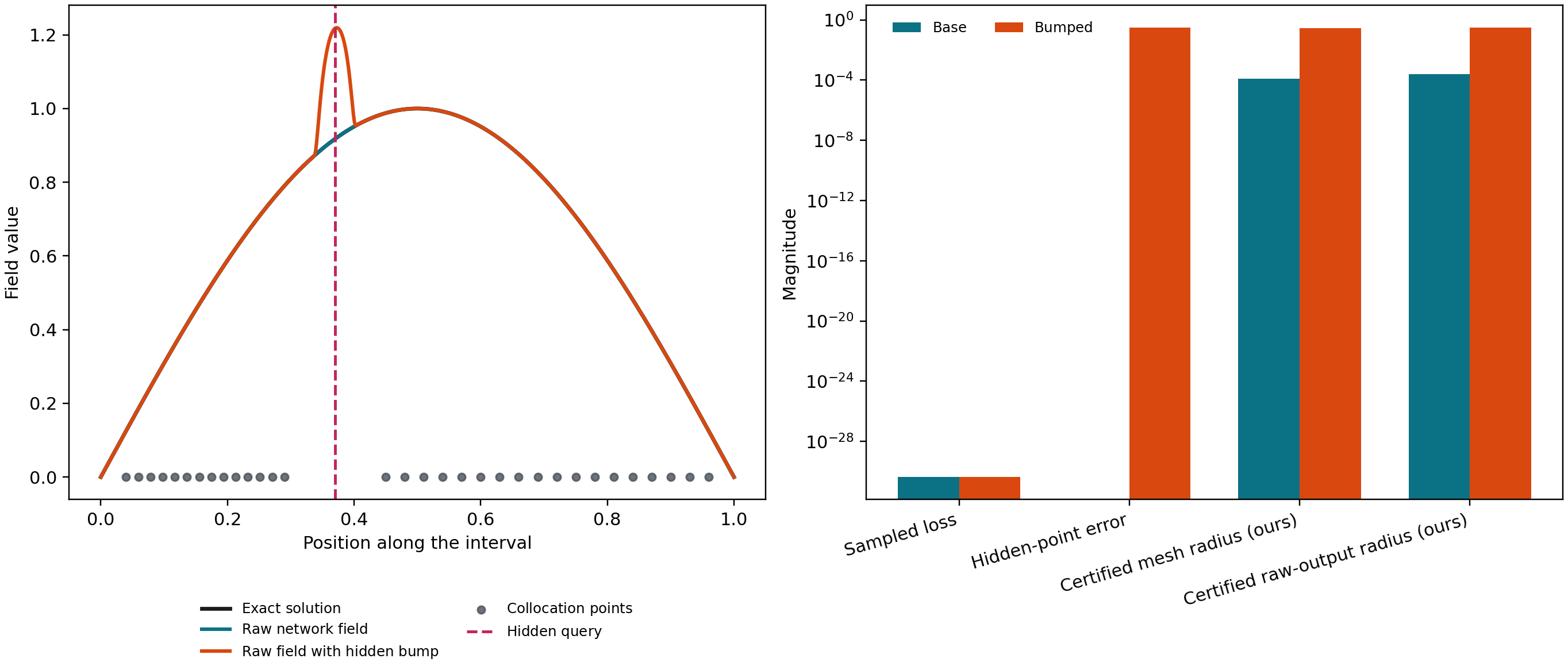}
    \caption{Finite collocation does not certify an unsampled point value. Left: a hidden bump inserted between collocation points leaves the sampled values unchanged at the collocation set but changes the hidden query value. Right: the sampled loss remains unchanged, while the hidden-point error grows and the compatible certified radii reflect this change. This illustrates the obstruction formalized in \Cref{prop:finite_collocation_no_point_control}.}
    \label{fig:compatibility-failure}
\end{figure}

\subsection{Linear calibration on a manufactured benchmark}
\label{subsec:exp-linear-calibration}

The first controlled linear test evaluates the identities from \Cref{sec:linear_main} and the radius decomposition from \Cref{subsec:width_decomposition}. The benchmark includes both nodal and off-grid queries and is evaluated over a refinement sequence. \Cref{fig:linear-calibration} shows three effects. First, both the reconstruction-centered and corrected-center certificates contract under refinement and remain above the corresponding errors. Second, the corrected center agrees with the compatible discrete target to the resolution shown, as predicted by \Cref{prop:exact_correction_recovers_discrete_target}. Third, the median width decomposition separates the signed discrete correction, the transfer/comparison term, and the off-grid evaluation term instead of aggregating them into a non-decomposed interval width. The full-column-rank overdetermined extension, including the effective-residual refinement, is given in \Cref{appsubsec:exp-rectangular-extension}. Together, these results are consistent with calibration to the chosen numerical system rather than only empirical association with the observed error.

Certified early stopping for primal and adjoint solves for this benchmark is reported separately in \Cref{fig:inexact-linear}.

\begin{figure}[t]
    \centering
    \adjustimage{width=\columnwidth,clip,trim=0 {.5\height} 0 0}{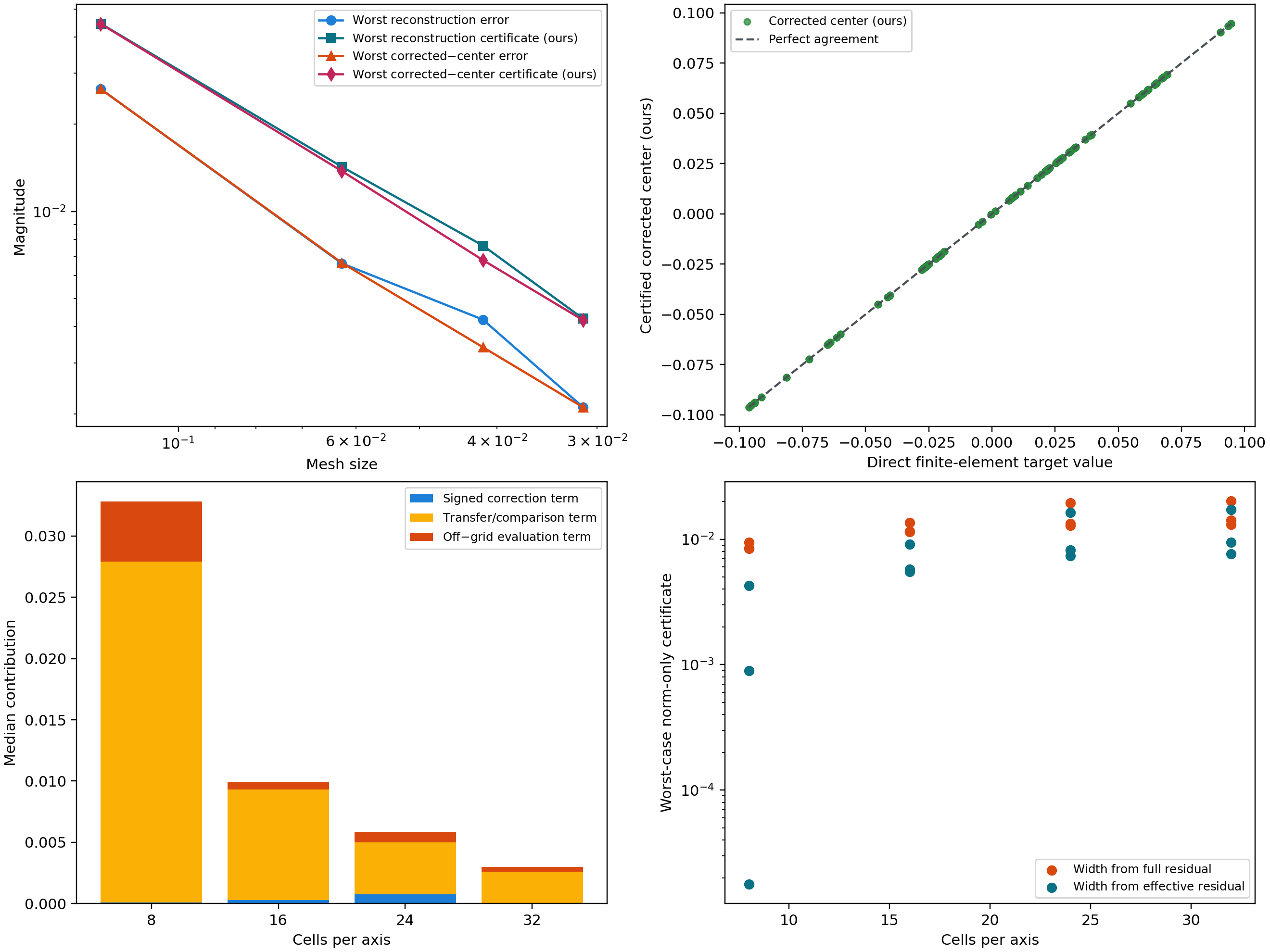}

    \vspace{0.6em}

    \adjustimage{width=0.495\columnwidth,clip,trim=0 0 {.5\width} {.5\height}}{linear_calibration.png}
    \caption{Linear calibration on a manufactured benchmark. Top left: worst off-grid reconstruction errors and their certificates under refinement. Top right: the corrected center agrees with the compatible discrete target. Bottom: median radius decomposition into the signed correction, transfer/comparison term, and off-grid evaluation term.}
    \label{fig:linear-calibration}
\end{figure}

\subsection{Point queries in three dimensions}
\label{subsec:exp-3d-pointwise}

A manufactured three-dimensional refinement study was used to separate the point-query behavior of the signed transfer term from that of the norm route in \Cref{subsec:coercive_galerkin_specialization}. The reported quantities are interior nodal point queries and fixed-radius averages of the same scalar displacement component. For the point queries, the comparison state is chosen so that the evaluation term vanishes at the queried nodes, leaving the signed term \(\abs{g_{i,h}^{\top}\tau_h^\pi}\) and the norm route \(\sigma_{i,h}T_h\) as the relevant transfer quantities.

\Cref{fig:3d-pointwise} shows three effects. Panel A reports the maximum sensitivity over the query set. The point-query sensitivity grows under refinement, with observed slope \(-0.55\), while the fixed-radius averages remain nearly flat over the same meshes. Panel B compares the maximum pointwise bias \(E_{i,h}^{\mathrm{pt}}\) with the maximum signed term \(\abs{g_{i,h}^{\top}\tau_h^\pi}\); the largest discrepancy over the refinement sequence is \(7.16\times 10^{-14}\). Panel C compares the same pointwise bias with the maximum norm route \(\sigma_{i,h}T_h\). The norm route remains valid but wider, with ratio \(17.1\) on the finest mesh. This is the point-query behavior that motivates the use of fixed-radius averages in the large-scale three-dimensional study below.

\begin{figure}[t]
    \centering
    \includegraphics[width=\columnwidth]{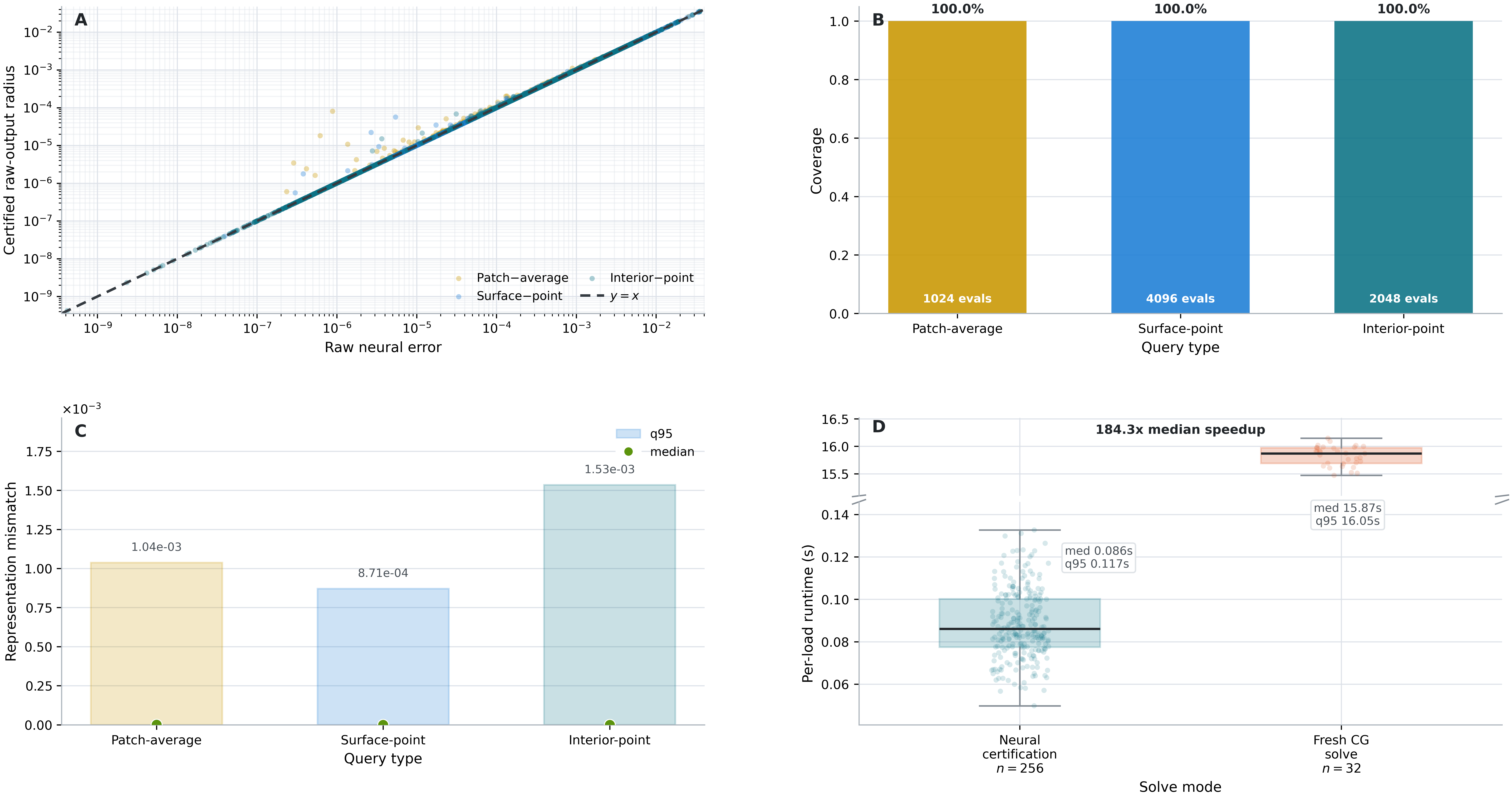}
    \caption{Point queries in three dimensions. A: maximum sensitivity for nodal point queries and fixed-radius averages over the refinement sequence; the observed point-query slope is \(-0.55\). B: maximum pointwise bias and maximum signed term \(\abs{g_{i,h}^{\top}\tau_h^\pi}\); the largest discrepancy over the refinement sequence is \(7.16\times 10^{-14}\). C: maximum pointwise bias and maximum norm route \(\sigma_{i,h}T_h\); on the finest mesh the norm route is \(17.1\times\) larger.}
    \label{fig:3d-pointwise}
\end{figure}

\subsection{Large-scale SimJEB validation: query geometry, Green sensitivities, correction, and transfer}
\label{subsec:exp-simjeb}

The pointwise claims are tested on manufactured nodal and off-grid queries, including the refinement study in \Cref{subsec:exp-3d-pointwise}. The SimJEB study then addresses whether the compatible-query algebra remains query-specific and computable in a large three-dimensional benchmark when the reported quantities are patch-average displacement functionals assembled in the same admissible coefficient basis. Those quantities fall directly under the scalar compatible-query setting of \eqref{eq:general_discrete_query}.

The benchmark is SimJEB, a public finite-element simulation dataset of crowdsourced jet-engine bracket designs; each entry includes a CAD file, tetrahedral mesh, triangular surface mesh, and structural simulation results for four load cases \cite{WhalenBeyeneMueller2021SimJEB}. We use public bracket 148. A compatible \(P_1\) tetrahedral small-strain linear-elasticity system was reconstructed from the released mesh, material parameters and load specification, support conditions, and nodal displacement file. After support elimination, the reduced system is square. The vertical load case is used as the primary test problem. The four reported scalar queries are patch-average surface \(u_z\) values on a loaded interface patch, a near-support patch, a high-deflection bridge-arm patch, and a remote low-sensitivity patch. The benchmark geometry and the query-dependent adjoint Green maps are given directly in \Cref{fig:simjeb-setup-green}; the raw-versus-corrected displacement fields are provided in \Cref{fig:simjeb-raw-corrected}.

\begin{figure*}[t]
    \centering
    \adjustimage{max width=0.48\textwidth,max totalheight=0.23\textheight}{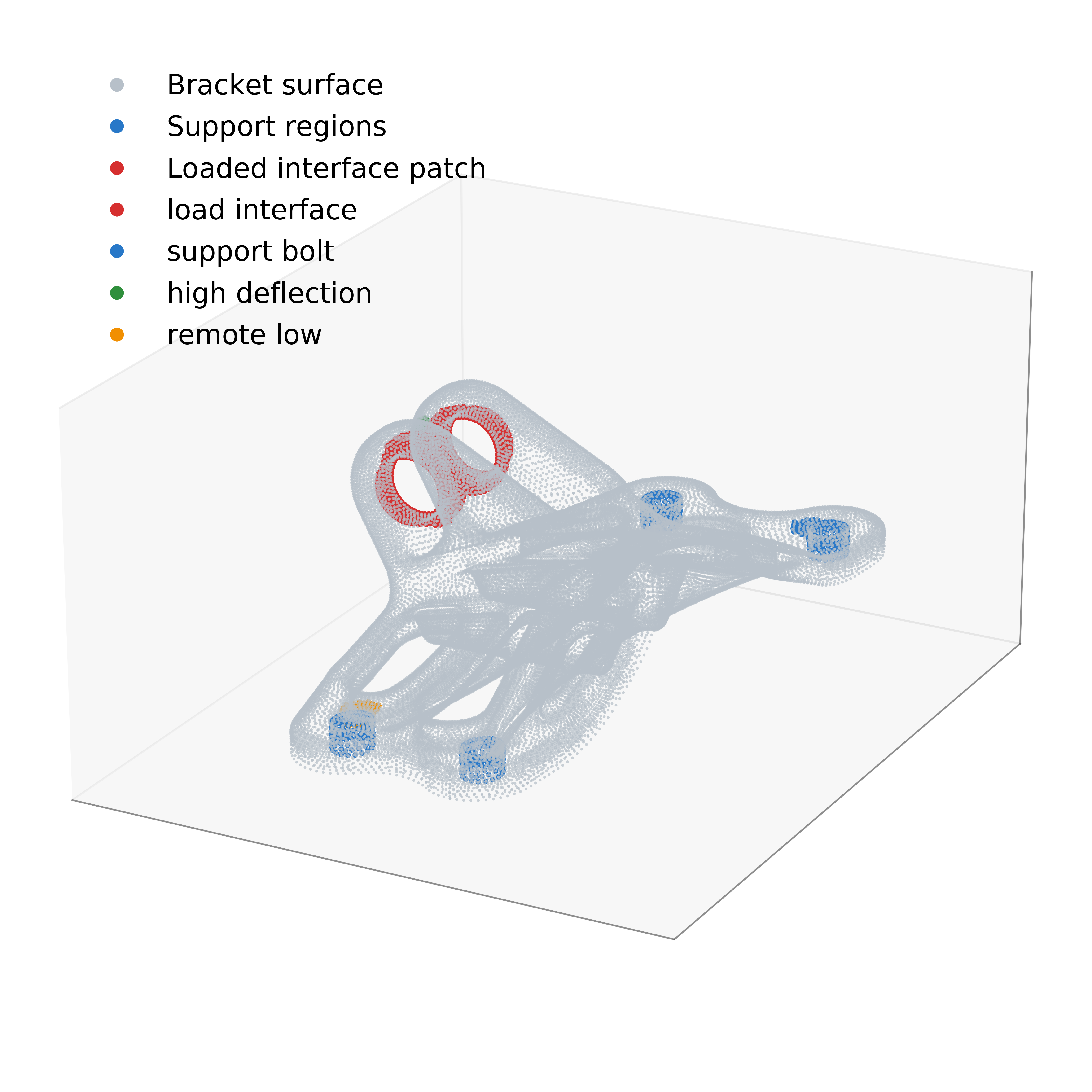}

    \vspace{0.4em}

    \adjustimage{max width=0.62\textwidth,max totalheight=0.30\textheight}{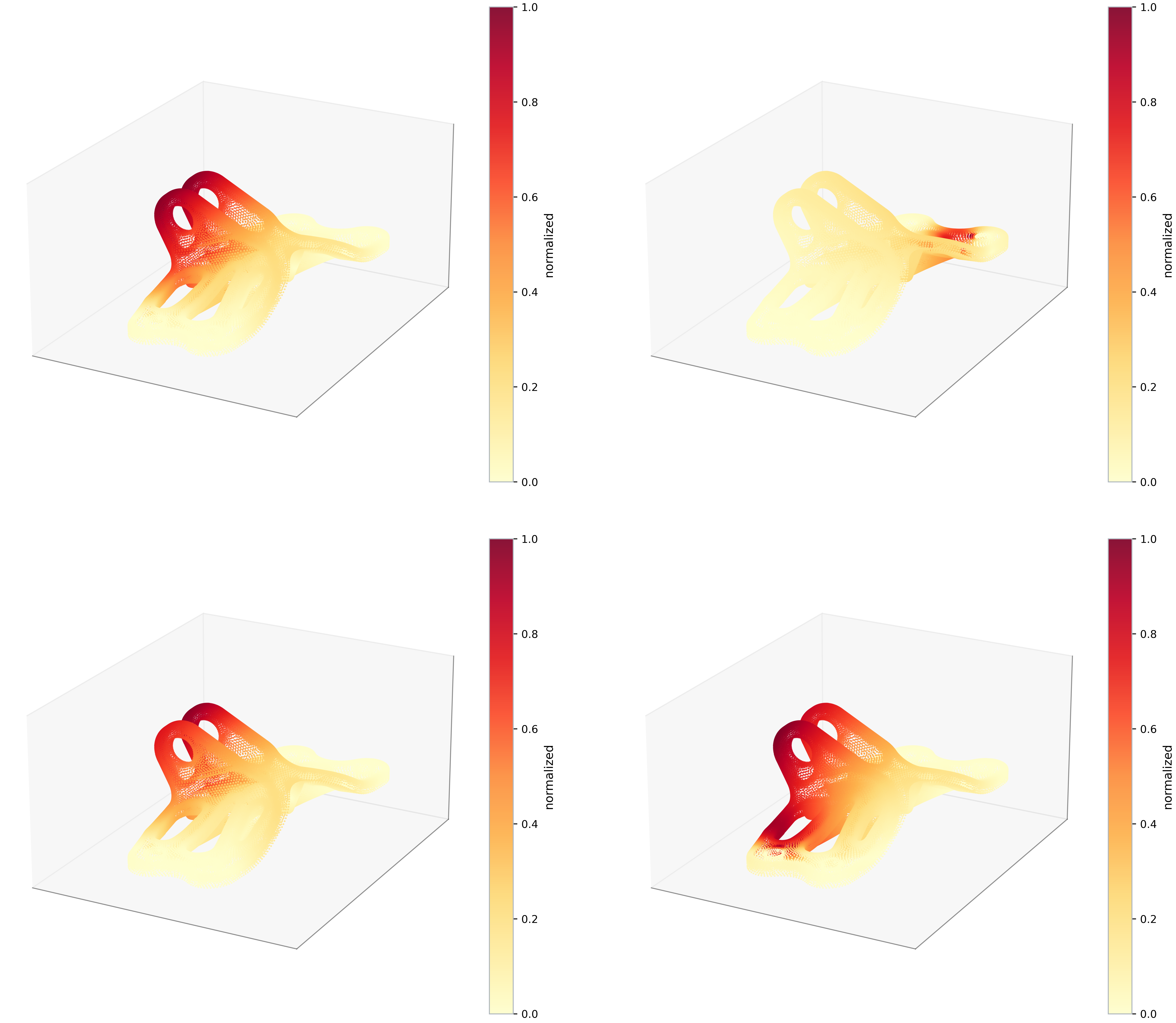}
    \caption{SimJEB bracket 148 used in the large-scale validation. Top: geometry, support regions, loaded interface patch, and the four scalar compatible query patches. Bottom: normalized adjoint Green magnitudes for the four queries. The spatial concentration differs strongly by query and explains the observed variation in sensitivities, norm-only widths, and localization radii.}
    \label{fig:simjeb-setup-green}
\end{figure*}

The Green maps quantify the querywise sensitivities used in the decomposition, localization, and randomized-sensitivity results. They indicate the physical structure of the querywise sensitivities \(\sigma_{i,h}\) used later in the analysis. The load-interface and high-deflection queries exhibit broader and stronger concentration, whereas the near-support and remote queries are more localized.

Starting from zero displacement, \(12\) CG iterations on the reduced compatible system produced the inexact state \(U_\theta\). The compatible residual \(r_h=K_hU_\theta-f_h\) was then formed using the same assembled stiffness matrix, eliminated support constraints, and reconstructed load vector. For each patch query, an adjoint solve gave the signed correction \(g_i^\top r_h\) and the corresponding norm-only certificate.

\Cref{fig:simjeb-decomposition} reports the patch-averaged discrete error, the signed correction, and the wider norm-only width for the four patch queries. The signed correction matches the actual discrete error at all four patches, as predicted by \Cref{prop:admissible_adjoint_green_identity}. The norm-only width is consistently conservative, and its conservatism is query-dependent in direct agreement with the heterogeneous Green concentration in \Cref{fig:simjeb-setup-green}. The near-support and remote patches have smaller norm widths, while the load-interface and high-deflection patches have larger sensitivities.

\begin{figure}[t]
    \centering
    \fitfigure{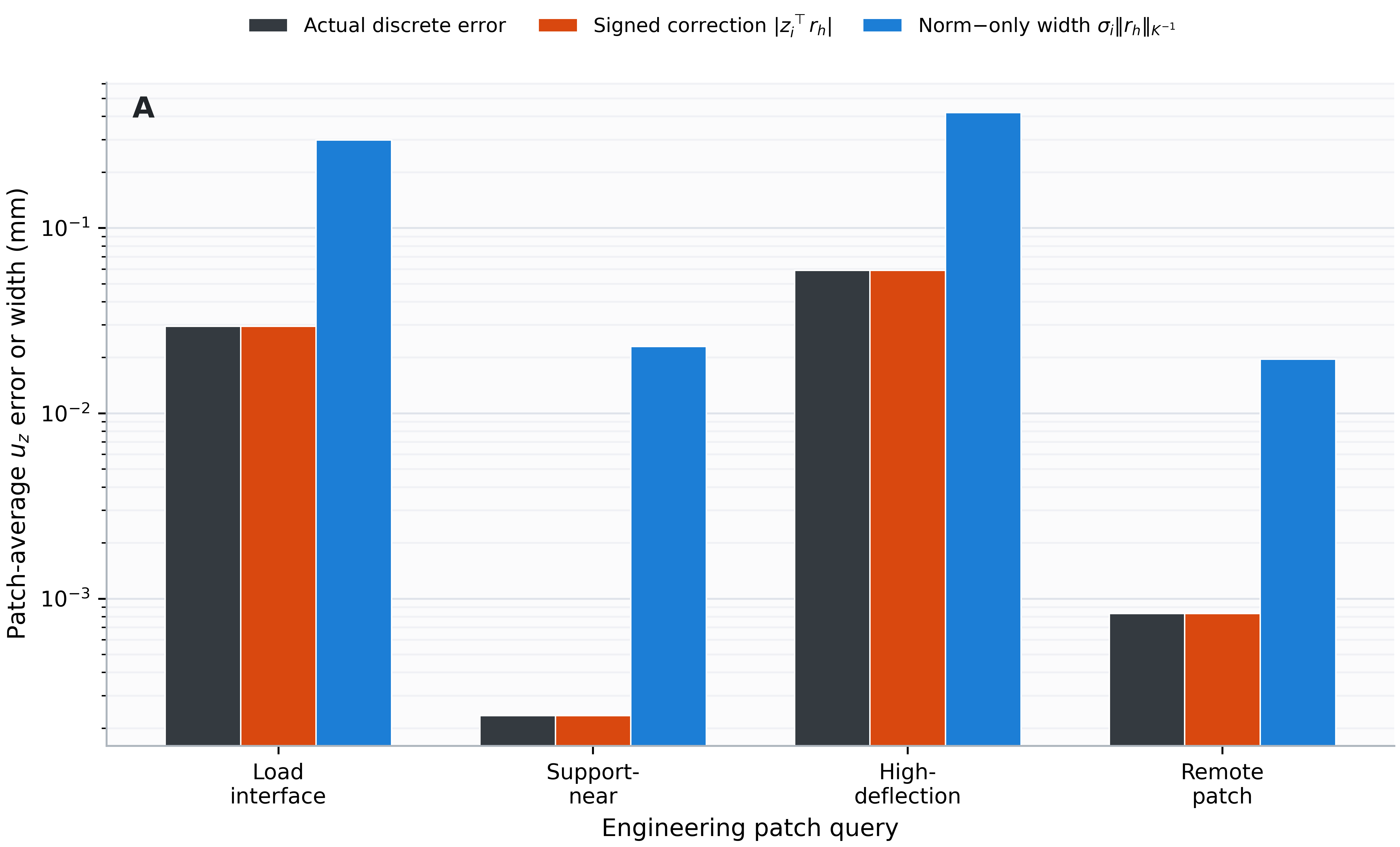}
    \caption{Querywise certificate decomposition on the SimJEB bracket. For each patch query, the actual discrete error is compared with the signed correction and the wider norm-only width. The signed correction tracks the actual discrete error, while the norm-only width remains conservative and varies substantially across queries.}
    \label{fig:simjeb-decomposition}
\end{figure}

The transfer layer was then tested using the released SimJEB vertical-case nodal displacement field as an external admissible comparison state \(U_h^{\mathrm{rel}}\). With
\[
    \tau_h^{\mathrm{rel}}:=K_hU_h^{\mathrm{rel}}-f_h,
\]
the comparison identity from \Cref{prop:admissible_discretization_bias} gives
\[
    \ell_i^\top(U_h^\star-U_h^{\mathrm{rel}})
    =
    -g_i^\top\tau_h^{\mathrm{rel}}.
\]
This does not certify the released field as a continuum reference solution. It tests the discrete comparison identity against an independently released admissible field on the same mesh. Because the figure reports absolute values, the sign convention does not affect the displayed comparison.

\Cref{fig:simjeb-released-transfer} shows that the signed comparison identity holds to numerical precision for all four patch queries. The maximum identity gap is \(2.67\times10^{-12}\,\mathrm{mm}\). The worst released-field patch bias is \(8.11\times10^{-8}\,\mathrm{mm}\), attained at the high-deflection patch. The norm comparison bound remains conservative: the four norm bounds are \(2.62\times10^{-4}\), \(2.02\times10^{-5}\), \(3.69\times10^{-4}\), and \(1.73\times10^{-5}\,\mathrm{mm}\) for the load-interface, near-support, high-deflection, and remote patches, respectively. These results separate two effects. The signed comparison identity is sharp for the released comparison field, whereas the Cauchy--Schwarz norm bound is valid but substantially more conservative.

\begin{figure}[H]
    \centering
    \fitfigure{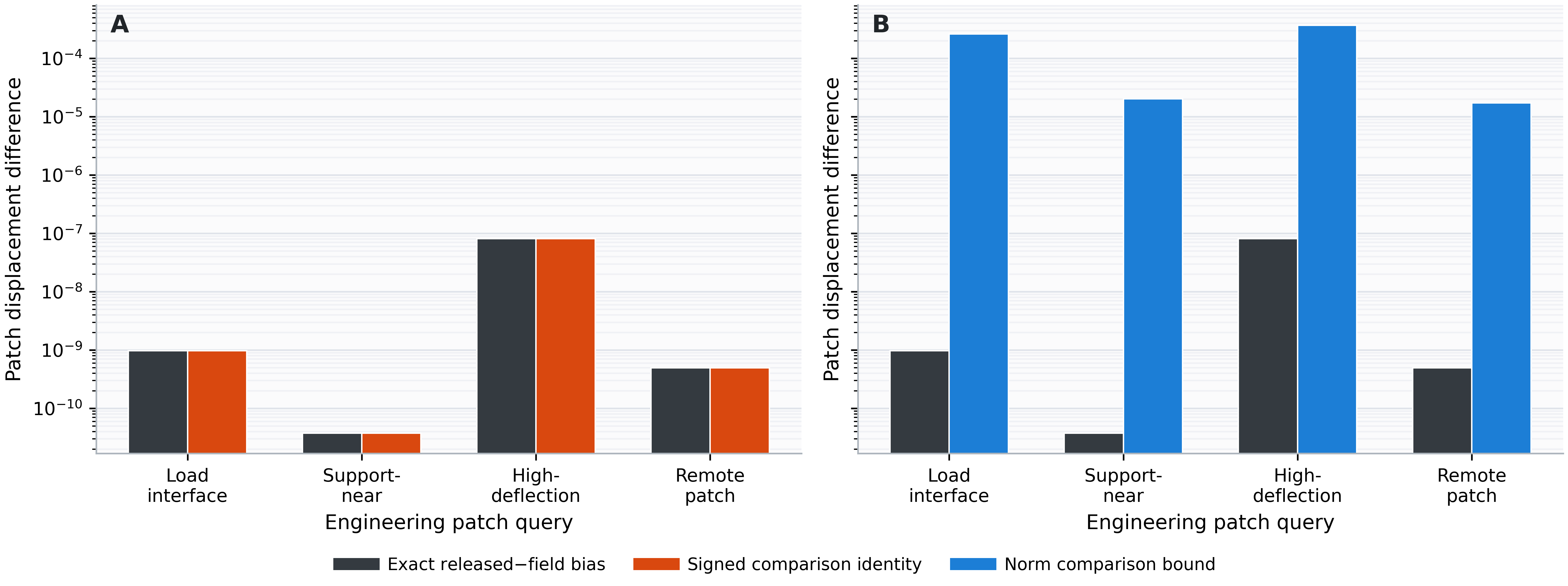}
    \caption{Released-field transfer validation on SimJEB bracket 148. The released vertical-case nodal displacement field is used as an external admissible comparison state \(U_h^{\mathrm{rel}}\) on the same mesh. A: the exact released-field patch bias and the signed comparison identity agree to numerical precision for all four patch queries. B: the corresponding norm comparison bound is conservative and query-dependent. All plotted quantities are absolute patch-average displacement differences in \(\mathrm{mm}\).}
    \label{fig:simjeb-released-transfer}
\end{figure}

\subsection{Projected neural certification on the SimJEB load family}
\label{subsec:exp-simjeb-projected-neural}

The previous SimJEB study starts from an inexact compatible solve and evaluates the coefficient-space identities on one load case. The present test instead starts from a raw neural predictor and evaluates the raw-output interval \eqref{eq:actual_pinn_output_interval} after explicit projection into the admissible space. The benchmark is again SimJEB bracket 148, now with the four-load family
\[
F_\mu
=
\mu_1 F_{\mathrm{vertical}}
+
\mu_2 F_{\mathrm{horizontal}}
+
\mu_3 F_{\mathrm{diagonal}}
+
\mu_4 F_{\mathrm{torsion}}.
\]
For each load, the neural model produces a raw coefficient vector \(\widetilde U_\theta(\mu)\), which is then projected to an admissible compatible state \(U_\theta(\mu)\in\Uhadm\). Certification uses the residual of this projected state, while the reported interval remains centered at the unprojected neural output. The test set contains \(256\) loads and \(28\) scalar queries: four patch-average surface queries, sixteen off-grid surface-point queries, and eight off-grid interior-point queries, for a total of \(7168\) query evaluations. For evaluation, a fresh compatible CG solve on the same reduced system provides the discrete target \(U_h^\star(\mu)\). This experiment therefore evaluates the discrete stage of the raw-output certificate, rather than the continuous transfer layer.

\begin{figure}[t]
    \centering
    \includegraphics[width=\columnwidth]{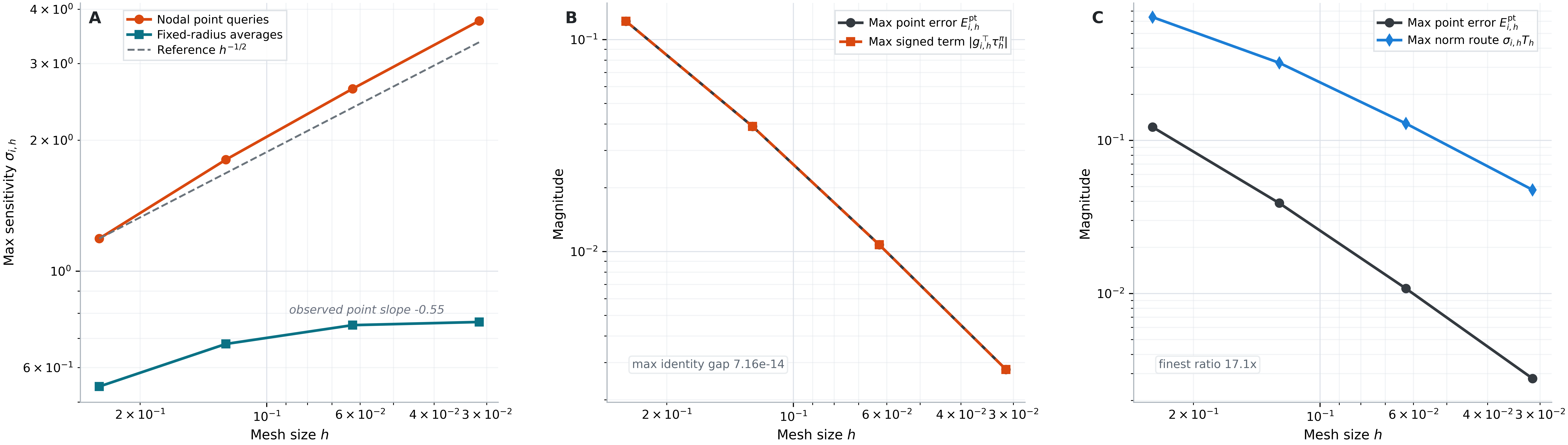}
    \caption{Projected neural certification on the SimJEB load family. A: certified raw-output radius against raw neural error over all \(7168\) test query evaluations; the dashed line is \(y=x\). B: coverage by query type. C: median and q95 representation mismatch \(\zeta_{i,h}^{\theta}\) by query type. D: per-load runtime for projected-neural certification and for a fresh compatible CG solve. All quantities are reported with respect to the compatible discrete target.}
    \label{fig:simjeb-projected-neural}
\end{figure}

\Cref{fig:simjeb-projected-neural} summarizes the resulting behavior. Panel A compares the raw neural error with the certified raw-output radius over all test query evaluations. The points lie on or above the line \(y=x\), and most are concentrated near that line. This is consistent with the corrected center matching the compatible discrete target to solver tolerance, so the numerical width in \eqref{eq:actual_pinn_output_radius} is largely governed by \(\abs{\utheta(x_i)-c_{i,h}^{\mathrm{corr}}}\). Panel B reports coverage \(1.0000\) for patch-average, surface-point, and interior-point queries. Panel C shows that the representation mismatch is small for most evaluations but not uniformly negligible; its upper tail is larger for interior queries than for the other query classes. Panel D compares the per-load runtime with a fresh compatible CG solve.

\begin{table}[t]
    \centering
    \caption{Projected neural certification on the SimJEB load family. Coverage and error statistics are reported over \(256\) test loads, \(28\) certified queries, and \(7168\) query evaluations. Runtime statistics use \(256\) projected-neural evaluations and \(32\) fresh compatible CG solves.}
    \begin{fittabular}{@{}cccccccc@{}}
        \toprule
        Coverage & Med. raw err. & q95 raw err. & Med. raw rad. & q95 raw rad. & q95 $\zeta_{i,h}^{\theta}$ & Max corr.-center err. & Med. speedup \\
        \midrule
        1.0000 & $3.6451\times 10^{-4}$ & $3.9033\times 10^{-3}$ & $3.6479\times 10^{-4}$ & $3.9033\times 10^{-3}$ & $1.1042\times 10^{-3}$ & $3.5227\times 10^{-12}$ & $184.3\times$ \\
        \bottomrule
    \end{fittabular}
    \label{tab:simjeb-projected-neural}
\end{table}

The aggregate values are reported in \Cref{tab:simjeb-projected-neural}. Over all \(7168\) test query evaluations, the raw-output interval contains the compatible discrete target in every case. The median raw error and median raw-output radius are \(3.6451\times 10^{-4}\) and \(3.6479\times 10^{-4}\), while the corresponding q95 values are \(3.9033\times 10^{-3}\) and \(3.9033\times 10^{-3}\). The q95 representation mismatch is \(1.1042\times 10^{-3}\). The corrected-center error remains at solver tolerance, with maximum \(3.5227\times 10^{-12}\). The median projected-neural certification runtime is \(8.61\times 10^{-2}\,\mathrm{s}\), compared with \(1.59\times 10^{1}\,\mathrm{s}\) for a fresh compatible CG solve, corresponding to a median speedup of \(184.3\times\).

\subsection{Localization and the necessity of an explicit tail term}
\label{subsec:exp-localization}

\Cref{fig:localization-toy} uses a dumbbell-domain benchmark to test the projector-based local/tail split from \Cref{prop:local_tail_projector_certificate}. The three query locations are chosen to induce qualitatively different adjoint geometries: near the source, in the bridge, and in the remote lobe. The local-only estimate is not uniformly reliable across these cases. Near the source, the local contribution dominates and the tail contribution decays rapidly. In the bridge and remote-lobe cases, a nontrivial tail remains necessary over a substantial range of radii. The certified local+tail estimate remains consistent with the actual point error across the sweep, whereas the local-only and tail-only terms are each insufficient when used alone. This is the computational implication of \Cref{prop:local_tail_projector_certificate}: locality is query-dependent and a deterministic certificate requires an explicit tail term.

\begin{figure}[t]
    \centering
    \fitfigure{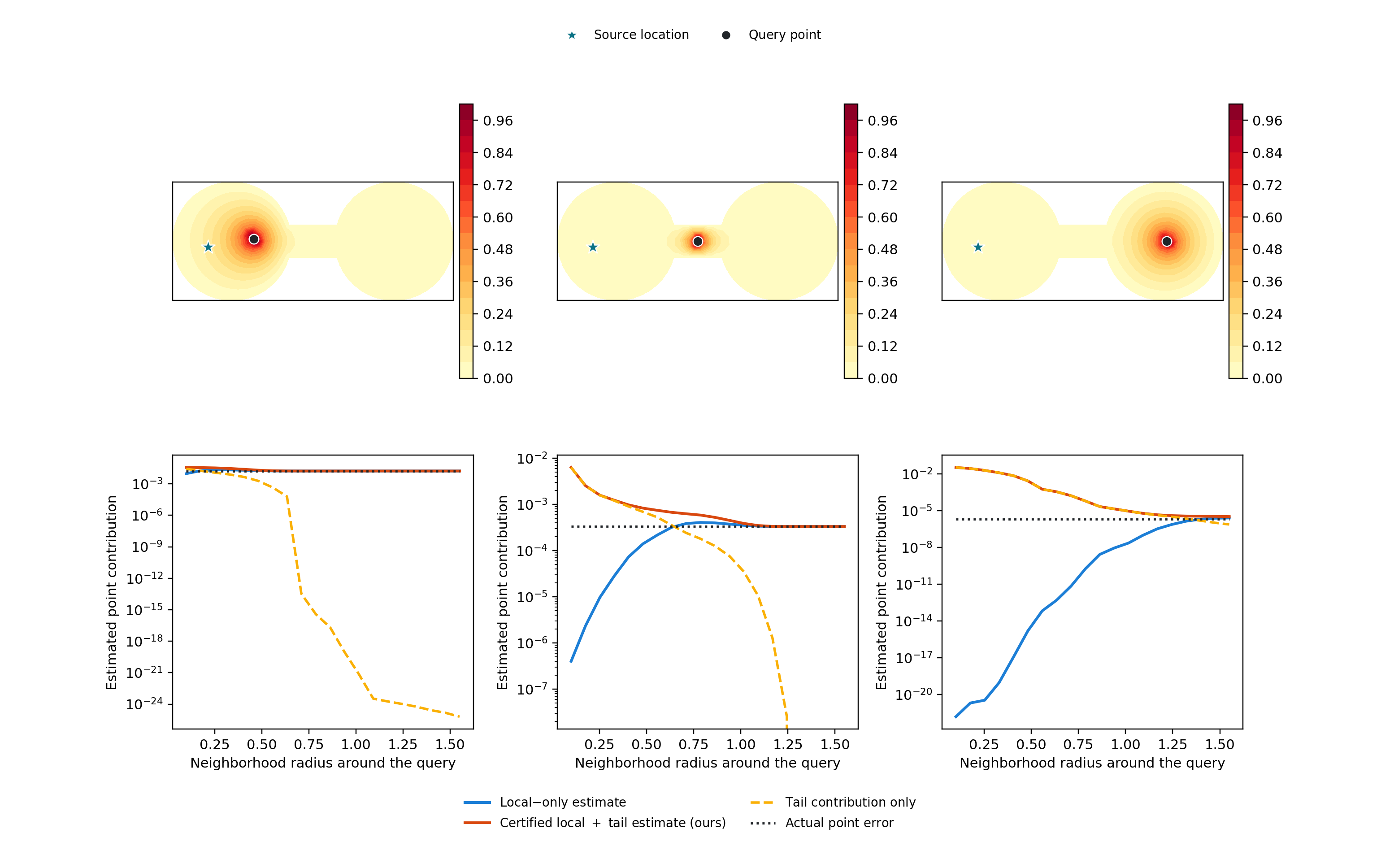}
    \caption{Localization on a dumbbell-domain benchmark. Top row: normalized adjoint Green influence maps for three query locations. Bottom row: radius sweeps comparing the local-only term, the certified local+tail estimate, the tail-only term, and the actual point error. The required radius is strongly query-dependent, and a tail term is essential for the bridge and remote-lobe cases.}
    \label{fig:localization-toy}
\end{figure}

The same mechanism appears on the large-scale SimJEB queries. The Green maps in \Cref{fig:simjeb-setup-green} already show the corresponding query dependence, while the bound-term density maps and full SimJEB radius sweeps are retained in \Cref{fig:simjeb-bound-term-maps,fig:simjeb-localization}.

\subsection{Simultaneous randomized sensitivity estimation}
\label{subsec:exp-randomized}

The exact computation of all query sensitivities becomes increasingly expensive as the number of queries grows. \Cref{fig:simjeb-randomized} therefore evaluates the Gaussian probing construction from \Cref{subsec:randomized_sensitivity} on the SimJEB bracket. The simultaneous upper bound attains the nominal \(95\%\) all-query coverage across the tested probe counts. The inflation factor decreases rapidly when moving from four to sixteen probes and then plateaus, indicating that additional probes give limited further reduction. The representative trial shows the expected one-sided bound: the randomized quantities remain above the exact sensitivities. The runtime panel shows the corresponding cost of the probe solves. The randomized construction therefore reduces sensitivity estimation to a fixed number of probe solves while retaining a simultaneous coverage statement. The querywise spread of the randomized inflation factors is consistent with the heterogeneous Green concentration already visible in \Cref{fig:simjeb-setup-green}.

Additional synthetic probing diagnostics and direct runtime scaling against exact adjoints are reported in \Cref{fig:toy-randomized,fig:probe-scaling}.

\begin{figure}[t]
    \centering
    \fitfigure{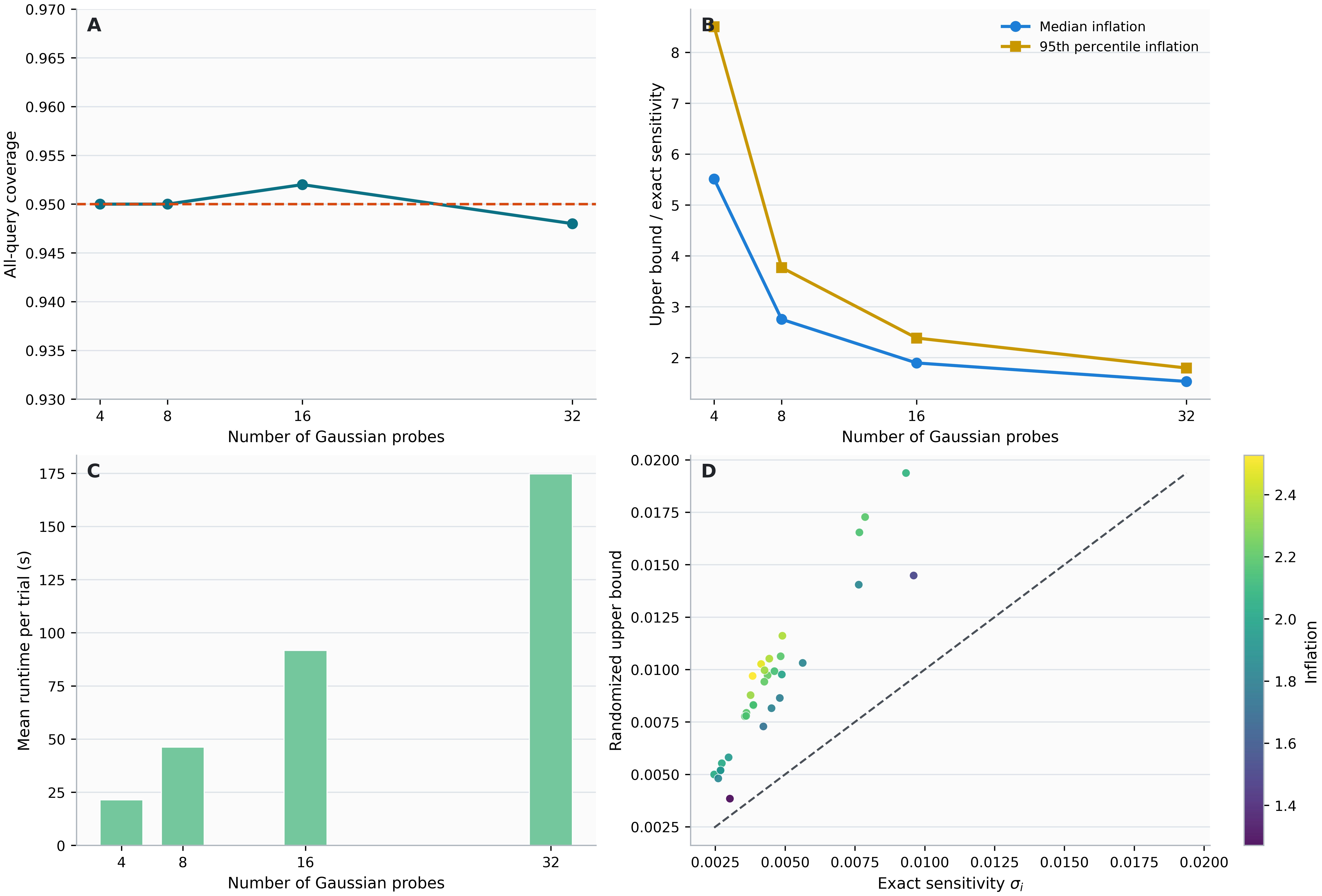}
    \caption{Simultaneous randomized sensitivity estimation on the SimJEB bracket. Top left: empirical all-query coverage against the nominal \(95\%\) target. Top right: median and \(95\)th-percentile inflation of the randomized upper bounds relative to exact sensitivities. Bottom left: probe runtime as a function of the number of Gaussian probes. Bottom right: a representative trial at \(m=16\), showing the one-sided upper-bound property query by query.}
    \label{fig:simjeb-randomized}
\end{figure}

\subsection{Nonlinear certification via verified exact roots}
\label{subsec:exp-nonlinear-root}

The nonlinear exact-root experiment tests \Cref{thm:nonlinear_root_interval}. \Cref{fig:nonlinear-root} shows that the corrected certificate contracts under refinement and remains substantially tighter than the norm-only version. The width decomposition further shows that, on this problem, the comparison-to-truth term dominates the nonlinear remainder term by several orders of magnitude. The remainder is nevertheless retained because it is required for verification in the certified root neighborhood. The verified-ball diagnostics show that the certified root neighborhood stays small relative to the base residual scale, and the per-query plot confirms that the reported radii remain comparable to the corrected errors query by query. The stationary-target variant exhibits the same qualitative behavior and is provided in \Cref{fig:nonlinear-ls}.

\begin{figure}[t]
    \centering
    \fitfigure{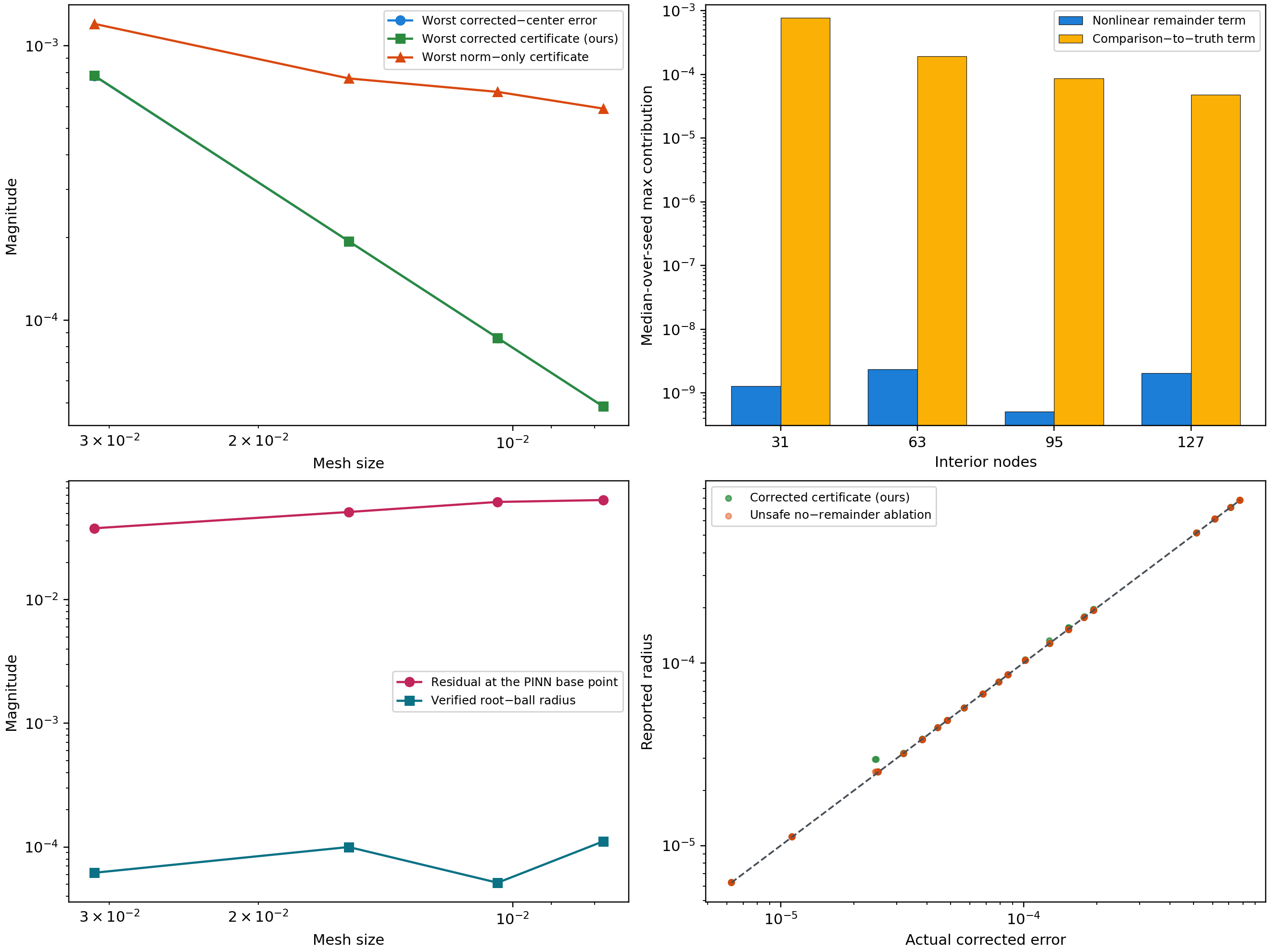}
    \caption{Nonlinear exact-root certificates. Top left: worst corrected-center and norm-only errors and certificates under refinement (worst corrected-center error perfectly overlaps with ours). Top right: width decomposition into the nonlinear remainder term and the comparison-to-truth term. Bottom left: verified-ball diagnostics showing the residual at the PINN base point and the certified root-ball radius. Bottom right: per-query corrected error versus reported radius.}
    \label{fig:nonlinear-root}
\end{figure}

Across these tests, the empirical results are consistent with the certificate chain. Compatibility is necessary. The linear corrected center recovers the compatible discrete target. The width decomposition separates discrete error, sensitivity, and transfer effects instead of aggregating them into a single scalar. On the SimJEB load family, the projected-neural study shows that the raw-output interval remains valid after projection into the admissible space, with a nonzero upper-tail mismatch term and per-load certification cost well below that of a fresh compatible solve. Localization requires an explicit tail. On SimJEB bracket 148, the released-field comparison experiment verifies the signed transfer identity against an external public displacement field and shows the expected conservatism of the norm comparison bound. Gaussian probing preserves simultaneous coverage at lower cost. The nonlinear exact-root construction remains non-vacuous after the remainder term is included. \Cref{app:additional_experiments} reports corresponding results for inexact linear algebra, nonlinear stationary targets, supplementary SimJEB diagnostics, and additional load-family studies.

\section{Conclusion}
\label{sec:conclusion}

In this work, we introduce a constructive framework for certified error bounds of numerical PINNs. The analysis shows that pointwise residual minimization alone is insufficient to guarantee convergence over the full domain, motivating the need for compatibility between the learned solution, the discrete numerical system, and the query mechanism used for evaluation. This leads naturally to discretized domains with sampled query points and certified transfer operators to the continuous setting. Mesh-based $P_1$ finite element spaces provide a particularly effective realization of this principle, although the proposed framework is not restricted to this choice.

The resulting certification decomposes the total error into explicit and interpretable contributions associated with the discrete residual, the discrete-to-continuous transfer, and the evaluation comparison. For linear problems, these quantities can be computed exactly and deterministically, while localized, randomized, and nonlinear extensions broaden the framework to more challenging settings. The theoretical results are supported by experiments on synthetic benchmarks and on SimJEB, a public dataset containing geometrically complex three-dimensional domains.

More broadly, this work establishes a bridge between neural PDE solvers and certified numerical analysis by combining learned approximation with rigorous verification. Rather than replacing classical numerical methods, the framework identifies how neural operators can be integrated into reliable scientific computing pipelines while retaining explicit control of approximation error. Promising directions for future work include combining the computed sensitivities with conformal prediction for sharper calibration, extending the framework to ill-posed and low-regularity systems, and coupling the certification layer with geometry-aware graph, hypergraph, and sheaf neural architectures. In particular, tangent-bundle and sheaf neural constructions suggest ways to generate vector-valued admissible fields on curved or heterogeneous geometries \cite{BattiloroWangRiessDiLorenzoRibeiro2024TangentBundle}, while transferability results for hypergraph neural networks suggest a possible route toward certified amortized prediction across changing mesh, load, and interaction structures \cite{HayhoeRiessZavlanosPreciadoRibeiro2024Hypergraph}. A further long-term direction is to express compatibility itself in a sheaf-theoretic local-to-global language, especially for multiphysics or heterogeneous-domain problems where local admissible data must agree across interfaces \cite{GhristRiess2022TarskiLaplacian}.

\bibliographystyle{elsarticle-num}
\bibliography{references}

\clearpage
\appendix

\crefalias{section}{appendix}
\crefalias{subsection}{subappendix}
\crefalias{subsubsection}{subappendix}
\renewcommand{\thesubsection}{\Alph{section}.\arabic{subsection}}
\renewcommand{\thesubsubsection}{\Alph{section}.\arabic{subsection}.\arabic{subsubsection}}
\renewcommand{\thetheorem}{\Alph{section}.\arabic{theorem}}
\renewcommand{\theproposition}{\thetheorem}
\renewcommand{\thelemma}{\thetheorem}
\renewcommand{\thecorollary}{\thetheorem}
\renewcommand{\thedefinition}{\thetheorem}
\renewcommand{\theassumption}{\thetheorem}
\renewcommand{\theremark}{\thetheorem}

\section{Additional Experiments and Reporting Details}
\label{app:additional_experiments}

This appendix gives experimental material that supports the certificate results and clarifies the numerical interpretation of the intervals.

\subsection{Certification under inexact linear algebra}
\label{appsubsec:exp-inexact}

\Cref{fig:inexact-linear} isolates the algebraic allowance terms from \Cref{prop:inexact_primal_correction,prop:inexact_adjoint_certificate}. On the manufactured benchmark, the certified primal and adjoint allowances decay in step with the actual corrected-center shift as the CG tolerance is tightened. The sensitivity inflation is close to one once the solves exit the low-accuracy regime, and the full-column-rank overdetermined example shows the same conservative behavior.

\Cref{fig:simjeb-inexact} repeats the primal-correction study on the SimJEB bracket. The observed center shift decreases with the number of correction iterations, and the certified allowance remains above it at each patch query and for the max-over-queries statistic. These plots illustrate certified early stopping: the remaining linear-algebra error can be reported explicitly rather than absorbed without decomposition into the interval.

\begin{figure}[H]
    \centering
    \fitfigure{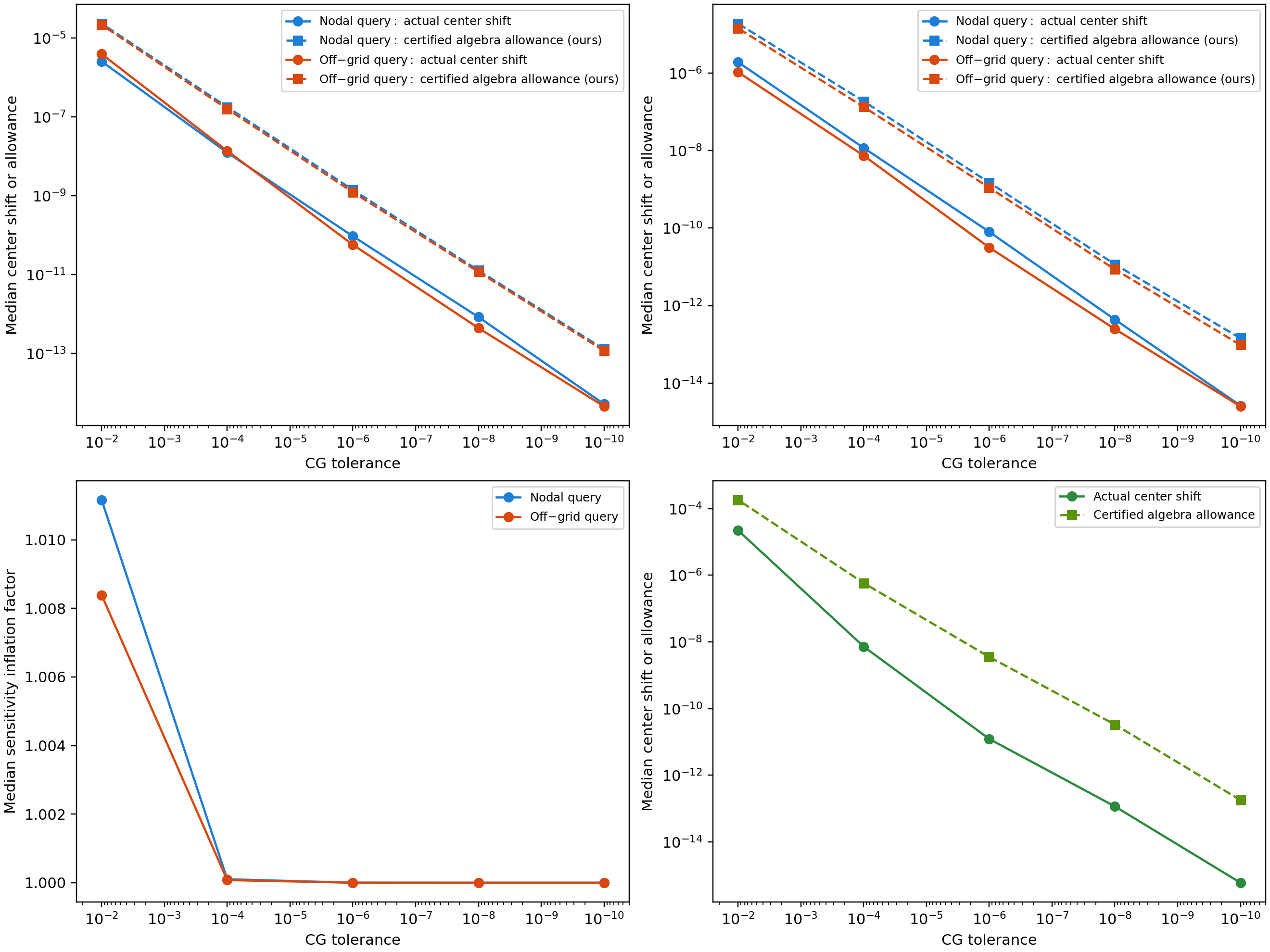}
    \caption{Inexact linear algebra on the manufactured benchmark. Top left: certified primal algebra allowance versus actual center shift. Top right: certified adjoint allowance versus actual center shift. Bottom left: sensitivity inflation caused by the inexact adjoint solve. Bottom right: corresponding primal-correction study for a full-column-rank overdetermined case.}
    \label{fig:inexact-linear}
\end{figure}

\subsection{Full-column-rank overdetermined extension}
\label{appsubsec:exp-rectangular-extension}

This subsection presents the manufactured linear example that isolates the full-column-rank overdetermined extension from \Cref{appsec:linear_rectangular_extension}. This panel isolates the effective-residual refinement in the overdetermined setting.

\begin{figure}[H]
    \centering
    \adjustimage{width=0.72\columnwidth,clip,trim={.5\width} 0 0 {.5\height}}{linear_calibration.png}
    \caption{Full-column-rank overdetermined extension on the manufactured linear benchmark. The norm-only width based on the effective residual is consistently smaller than the width based on the full residual.}
    \label{fig:linear-calibration-rectangular-extension}
\end{figure}

\begin{figure}[H]
    \centering
    \fitfigure{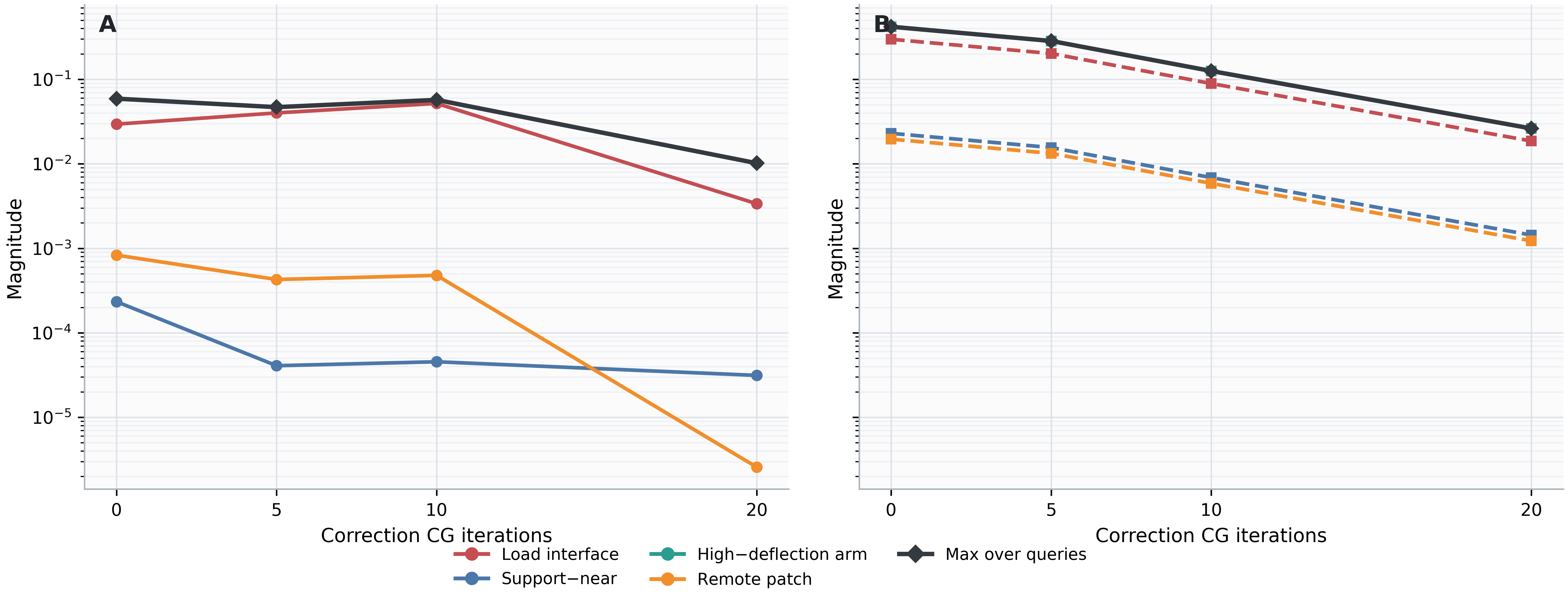}
    \caption{Early-stopped primal correction on the SimJEB bracket. Left: observed corrected-center shift versus the number of correction CG iterations at the four patch queries and for the max-over-queries statistic. Right: the corresponding certified primal allowances.}
    \label{fig:simjeb-inexact}
\end{figure}

\subsection{Nonlinear least-squares stationary targets}
\label{appsubsec:exp-nonlinear-ls}

\Cref{fig:nonlinear-ls} evaluates the stationary-target certificate from \Cref{thm:nonlinear_ls_interval}. The corrected certificate contracts with refinement and is substantially tighter than the norm-only version. The stationarity remainder is small on this problem relative to the comparison-to-truth term, but it is retained because the theorem is stated around a verified stationary point rather than an exact discrete root. The verified stationary-ball radii remain small, and the per-query radii stay aligned with the corrected errors.

\begin{figure}[H]
    \centering
    \fitfigure{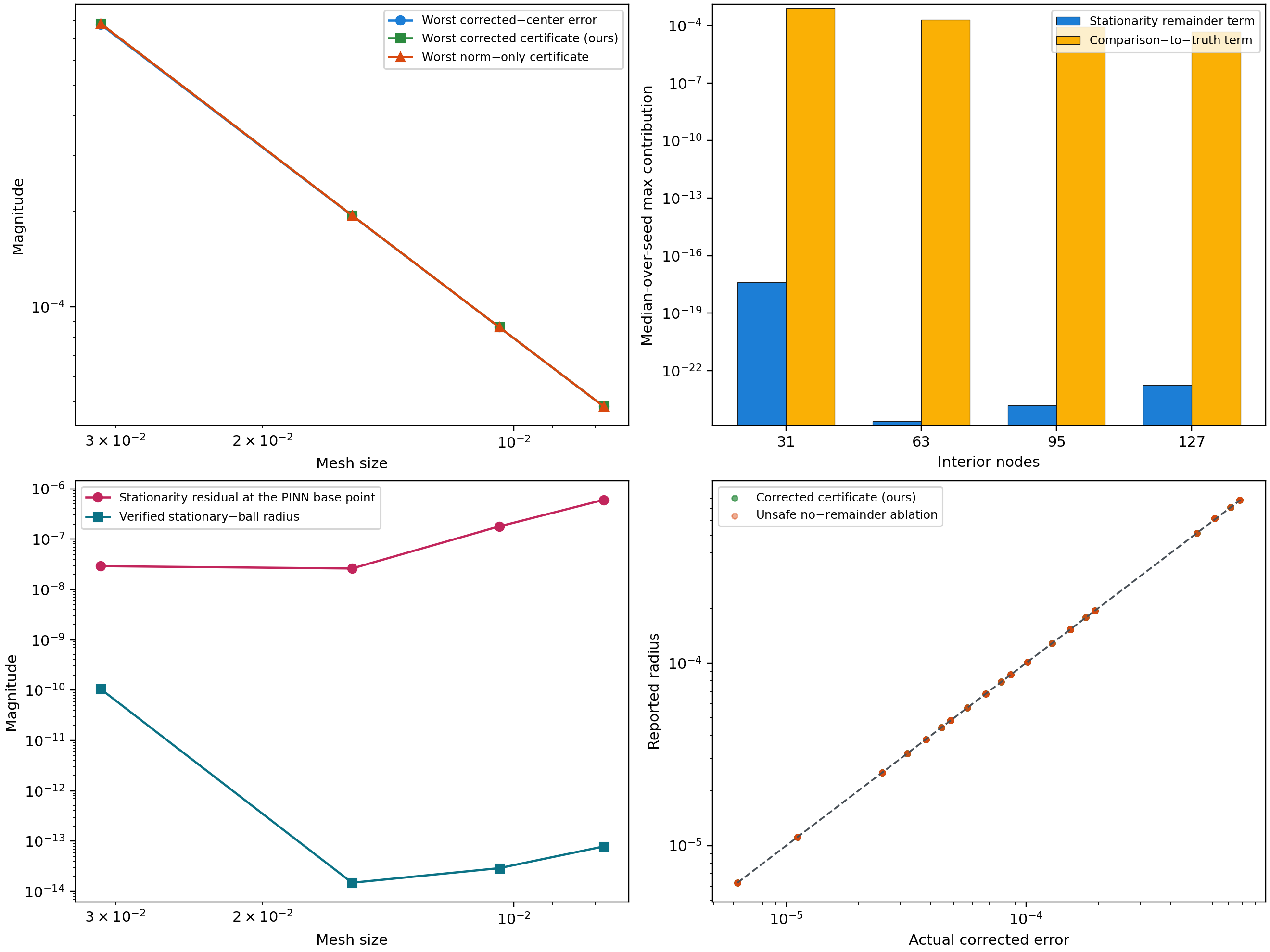}
    \caption{Nonlinear least-squares stationary-target certificates. Top left: worst corrected-center and norm-only errors and certificates under refinement. Top right: width decomposition into the stationarity remainder and the comparison-to-truth term. Bottom left: verified stationary-ball diagnostics. Bottom right: per-query corrected error versus reported radius.}
    \label{fig:nonlinear-ls}
\end{figure}

\subsection{Supplementary SimJEB field and localization diagnostics}
\label{appsubsec:exp-simjeb-extra}

\Cref{subsec:exp-simjeb} gives the bracket/query geometry and query-dependent Green maps for the SimJEB benchmark. This subsection gives the raw-versus-corrected field comparison, bound-term density maps, and full localization sweeps, which support interpretation but are not required for the primary theorem-level narrative.

\Cref{fig:simjeb-raw-corrected} compares the uncorrected early-stopped field with the corrected field that coincides with the compatible discrete target. The difference field is spatially structured rather than uniform, which is why the decomposition is reported patch by patch.

\Cref{fig:simjeb-localization} performs a radius sweep for the four patch queries. The certified local+tail estimate remains above the actual error across the sweep, whereas the local-only and tail-only terms are each inadequate on parts of the radius range. The load-interface and high-deflection queries retain a sizeable tail over long radii, while the near-support and remote queries localize more rapidly.

\begin{figure}[H]
    \centering
    \fitfigure{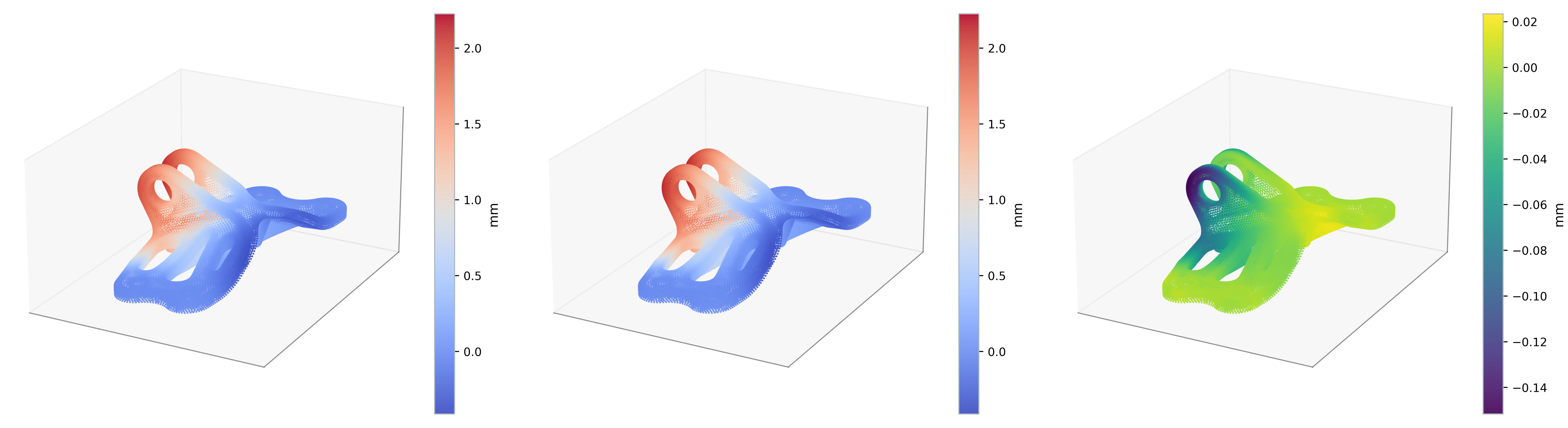}
    \caption{Uncorrected-versus-corrected displacement field on the SimJEB bracket. Left: uncorrected early-stopped field. Middle: corrected field equal to the compatible discrete target. Right: spatially structured difference between the two fields.}
    \label{fig:simjeb-raw-corrected}
\end{figure}

\begin{figure*}[!t]
    \centering
    \adjustimage{max width=\textwidth,max totalheight=0.82\textheight}{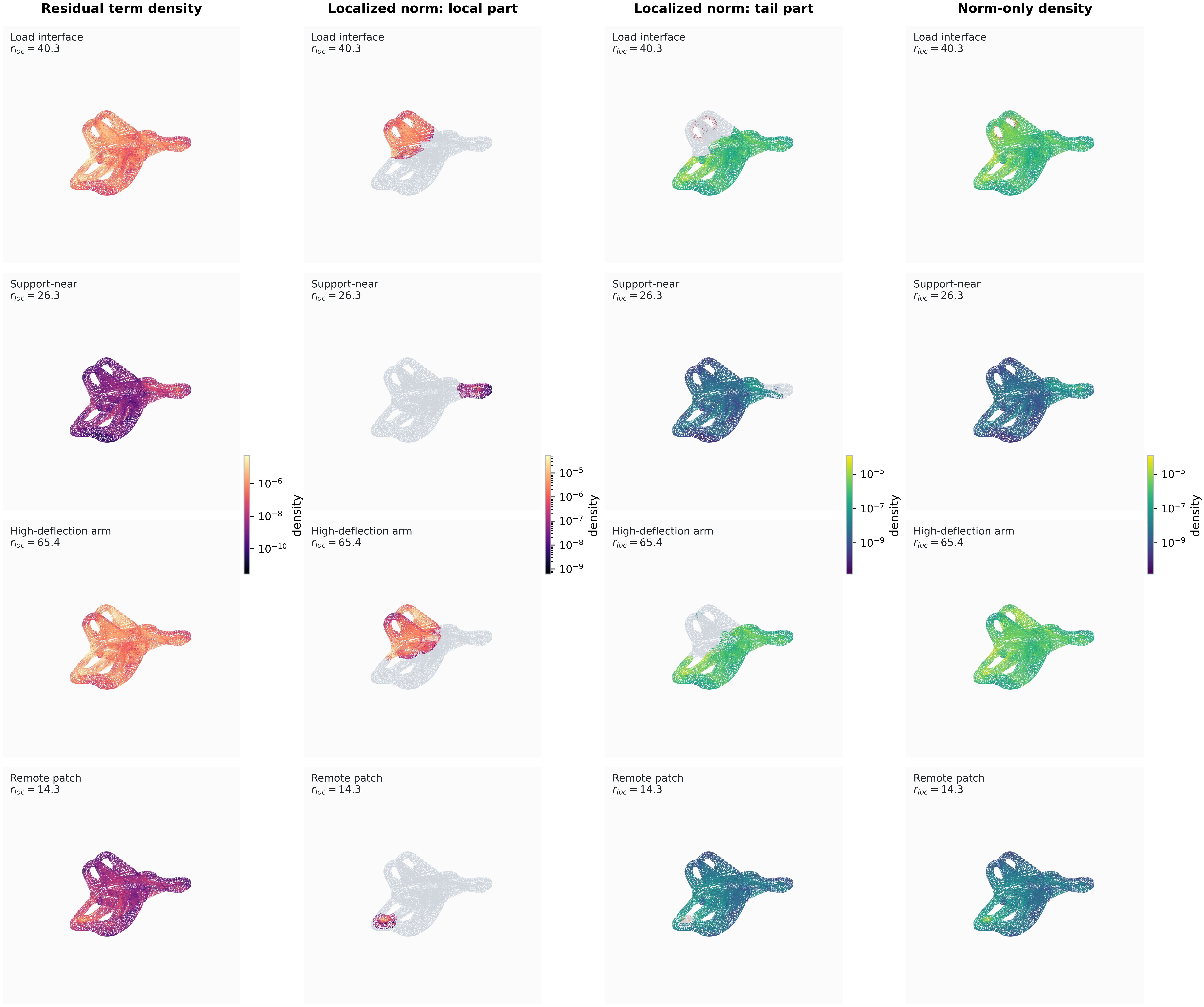}
    \caption{Spatial density of the SimJEB certificate terms for the four patch queries. Rows correspond to the load interface, support-near patch, high-deflection arm, and remote patch; \(r_{\mathrm{loc}}\) is the selected query-dependent localization radius. Columns show, from left to right, the residual correction density, the localized norm contribution inside the selected ball, the complementary tail contribution outside it, and the unsplit norm-only density. Brighter colors indicate larger contribution; gray marks zero contribution in the displayed split. The maps visualize certificate terms, not displacement.}
    \label{fig:simjeb-bound-term-maps}
\end{figure*}

\begin{figure}[H]
    \centering
    \fitfigure{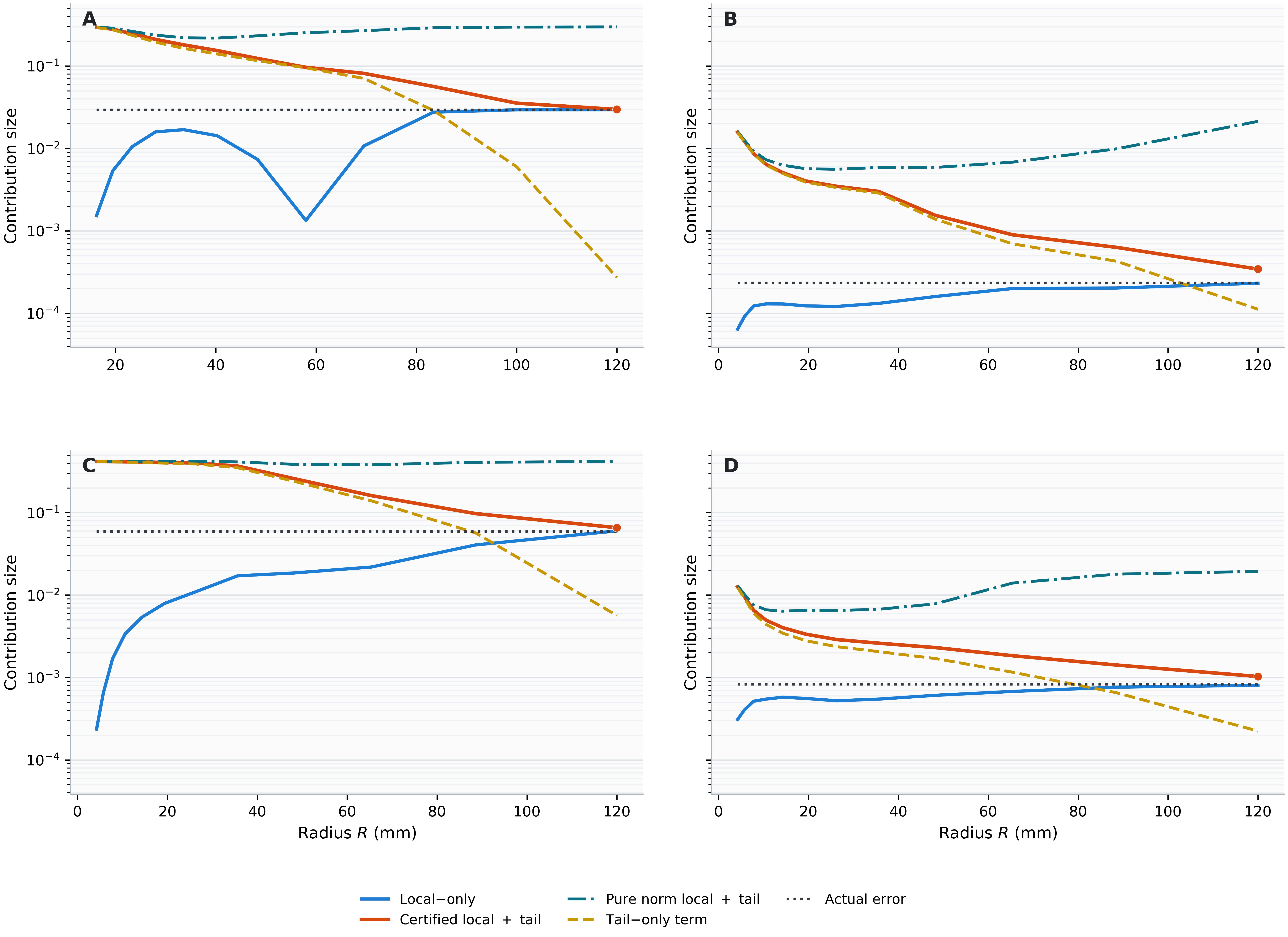}
    \caption{Localization radius sweeps on the SimJEB bracket. Each panel compares the local-only term, the certified local+tail estimate, the pure norm local+tail quantity, the tail-only term, and the actual error for one patch query.}
    \label{fig:simjeb-localization}
\end{figure}

\subsection{Supplementary diagnostics for randomized sensitivity estimation}
\label{appsubsec:exp-randomized-extra}

The synthetic probing tests complement the SimJEB results. \Cref{fig:toy-randomized} shows the same two qualitative trends on a controlled benchmark: the simultaneous upper bound attains the nominal coverage and the inflation factor decays with the number of probes. \Cref{fig:probe-scaling} isolates the computational tradeoff. Exact adjoints scale with the number of queries, whereas probing cost is essentially fixed by the number of probe solves and becomes advantageous for large query sets. The second panel shows the corresponding cost--tightness tradeoff: more probes cost more but reduce the inflation of the \(95\)th-percentile bound.

\begin{figure}[H]
    \centering
    \fitfigure{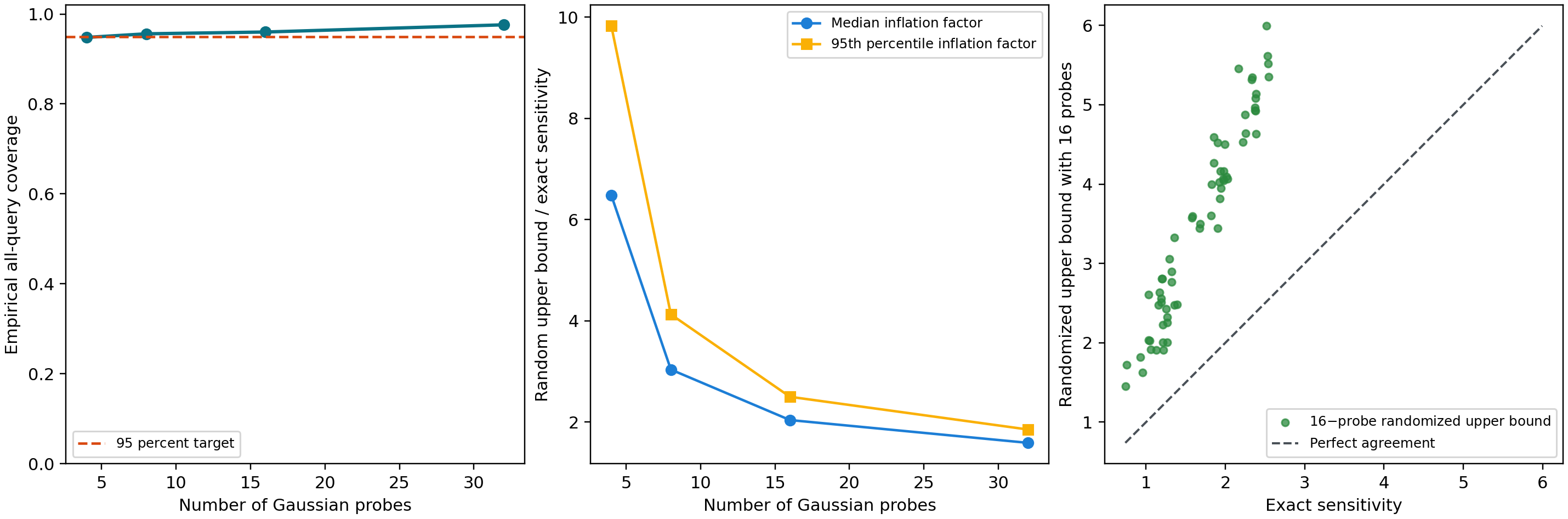}
    \caption{Synthetic randomized-sensitivity diagnostics. Left: empirical simultaneous coverage against the nominal target. Middle: inflation of the randomized upper bound relative to exact sensitivities. Right: representative trial showing the one-sided upper-bound behavior.}
    \label{fig:toy-randomized}
\end{figure}

\begin{figure}[H]
    \centering
    \fitfigure{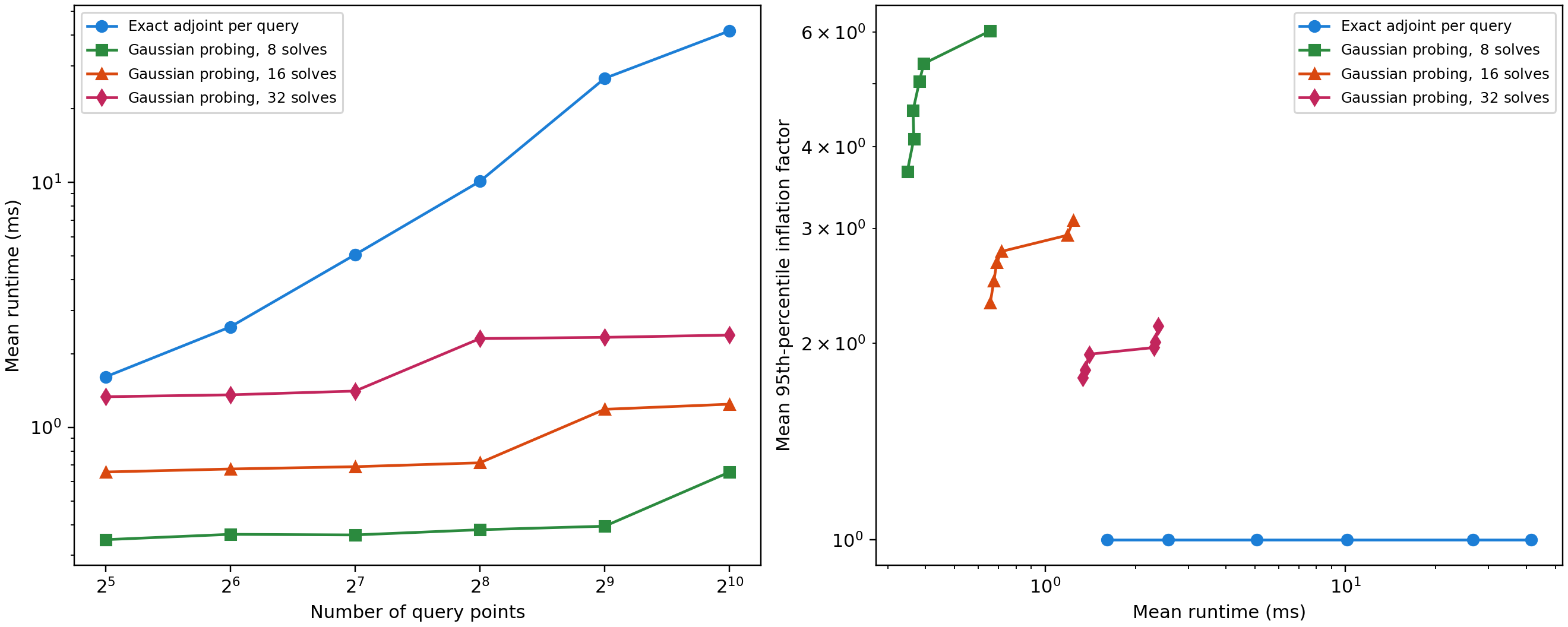}
    \caption{Runtime scaling for exact adjoints and Gaussian probing. Left: runtime versus the number of query points for exact per-query adjoints and fixed probe budgets. Right: the cost--tightness tradeoff between runtime and inflation.}
    \label{fig:probe-scaling}
\end{figure}

\subsection{Amortized load families}
\label{appsubsec:exp-amortized}

The load-family experiments address the many-instance setting discussed in \Cref{rem:pinn_role_linear_exact_correction}. \Cref{fig:simjeb-amortized-family} studies a linear reduced load-family surrogate. Increasing the output rank reduces both the certified test quantities and the residual energy, with the signed correction remaining close to the actual 95th-percentile maximum query error and the norm-only width remaining wider.

\Cref{fig:simjeb-amortized-nn} reports the corresponding physics-trained neural load-family model. The training and validation losses stabilize, the certified 95th-percentile maximum signed correction tracks the actual 95th-percentile maximum query error, and rank three gives the smallest reported certified test quantity among the tested neural models while substantially reducing residual energy relative to rank two. These results characterize certificate behavior in the amortized regime, where the neural model produces many compatible coefficient vectors and the certificate verifies each one.

\begin{figure}[H]
    \centering
    \fitfigure{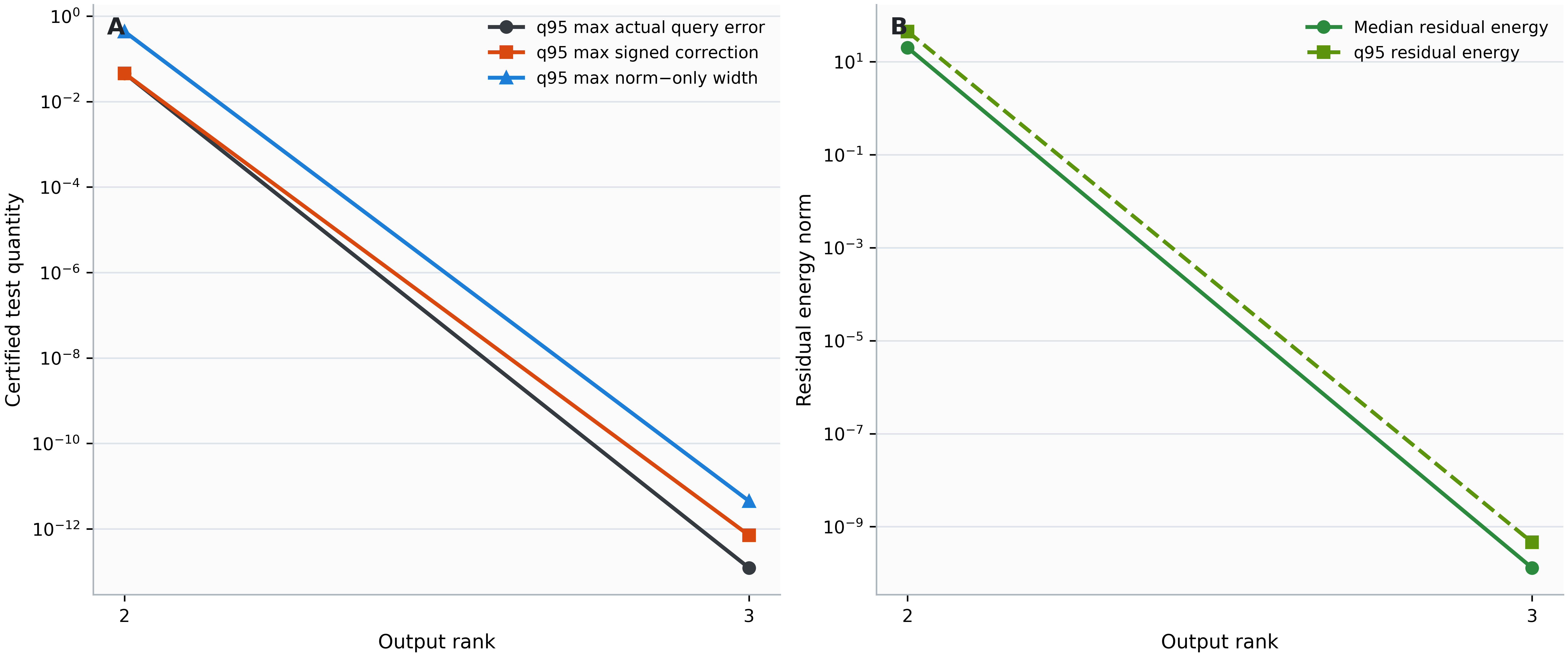}
    \caption{Amortized reduced load-family surrogate on the SimJEB bracket. Left: q95 certified test quantities versus output rank. Right: residual energy versus output rank.}
    \label{fig:simjeb-amortized-family}
\end{figure}

\begin{figure}[H]
    \centering
    \fitfigure{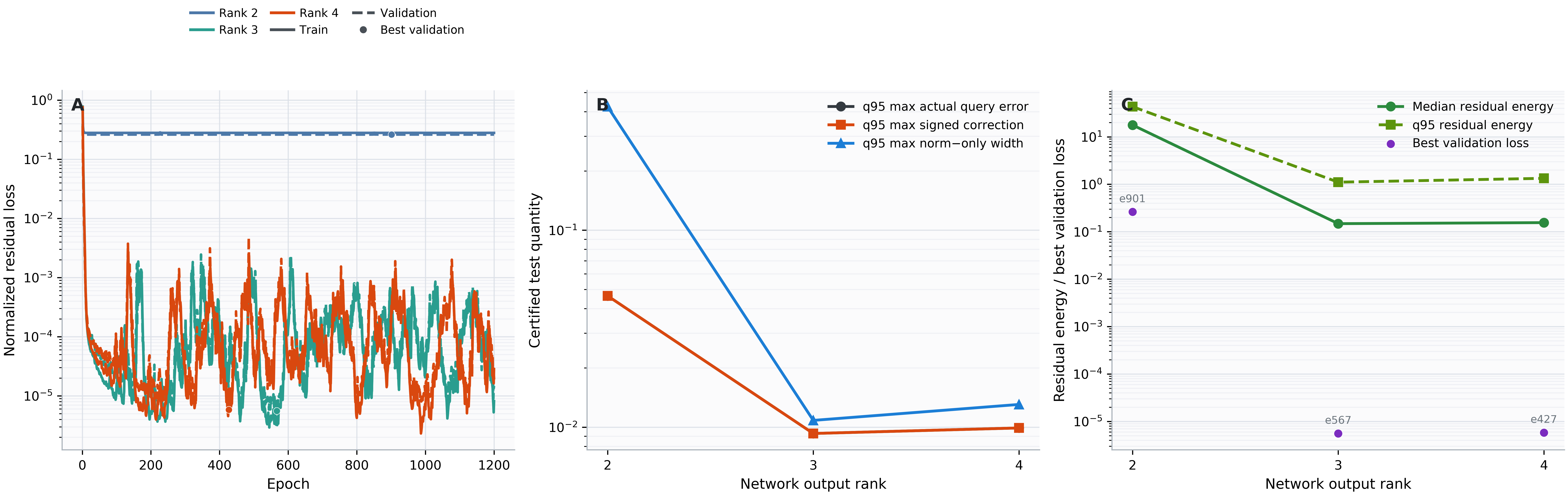}
    \caption{Amortized physics-trained neural load-family model. Left: representative training and validation residual losses. Middle: q95 certified test quantities versus network output rank. Right: residual energy and best validation fit versus network output rank.}
    \label{fig:simjeb-amortized-nn}
\end{figure}

\subsection{Experimental configuration summary}
\label{appsubsec:experimental-configuration}
\Cref{tab:app-experimental-configuration} gathers the configuration used in the numerical experiments throughout the paper.

For the main linear manufactured benchmark, the data are generated on \(\Omega=(0,1)^2\) with homogeneous Dirichlet boundary conditions from
\[
    -\Delta u^\star=f,
\]
using the manufactured solution
\[
\begin{aligned}
    u^\star(x,y)
    &=
    x(1-x)y(1-y)
    \exp\!\left(
    -\frac{(x-0.7)^2+(y-0.35)^2}{0.12^2}
    \right) \\
    &\quad
    +0.1\sin(3\pi x)\sin(2\pi y).
\end{aligned}
\]
The forcing is generated by applying the operator to this solution, \(f=-\Delta u^\star\), so \(u^\star\) is the exact continuous solution by construction. The compatible \(P_1\) systems are assembled on structured meshes with \(8\), \(16\), \(24\), and \(32\) cells per coordinate direction. For each mesh, we use both the finite-element solution \(u_h^{\mathrm{FE}}\) and the nodal interpolant \(u_h^\pi=\mathcal J_hu^\star\). The trained boundary-vanishing coordinate MLP is evaluated through its nodal values in the admissible \(P_1\) space, and the fixed query set contains \(12\) interior nodal queries and \(12\) off-grid point queries.

\begin{table*}[t]
    \centering
    \caption{Main experimental configurations.}
    \scriptsize
    \begin{adjustbox}{max width=\textwidth}
    \begin{tabular}{@{}p{0.14\textwidth}p{0.19\textwidth}p{0.18\textwidth}p{0.19\textwidth}p{0.16\textwidth}p{0.16\textwidth}@{}}
        \toprule
        Experiment & Problem/data & Model or certified state & Objective or solve & Mesh/data details & Queries/notes \\
        \midrule
        Linear manufactured calibration
        &
        Manufactured 2D unit-square problem with analytic solution containing localized Gaussian and sinusoidal components
        &
        Boundary-vanishing coordinate PINN (MLP); direct nodal values in the compatible discrete space
        &
        Normalized mean-squared compatible residual of the reduced FE system; Adam, learning rate \(2\times10^{-3}\), 600 epochs; 3 hidden layers, width 96, Tanh; seeds 0, 1, 2
        &
        Mesh levels 8, 16, 24, and 32 cells per axis
        &
        12 nodal queries and 12 fixed off-grid queries
        \\
        \addlinespace
        Point queries in three dimensions
        &
        Manufactured 3D refinement study on nested meshes
        &
        No separately trained neural model; compatible Galerkin target and admissible comparison state
        &
        Refinement study on the compatible linear system
        &
        Nested three-dimensional meshes
        &
        Interior nodal point queries and fixed-radius averages of the same scalar displacement component
        \\
        \addlinespace
        Nonlinear exact-root certification
        &
        Manufactured 1D semilinear Dirichlet problem with \(\lambda=20\)
        &
        Boundary-vanishing coordinate PINN (MLP); direct nodal values in the discrete nonlinear state
        &
        Normalized mean-squared discrete semilinear residual; Adam, learning rate \(2\times10^{-3}\), 1400 epochs; 3 hidden layers, width 96, Tanh; seeds 0, 1, 2
        &
        Meshes with 31, 63, 95, and 127 interior nodes
        &
        8 fixed query points
        \\
        \addlinespace
        Main SimJEB large-scale validation
        &
        SimJEB bracket 148 (``butterfly''), vertical load case
        &
        No separately trained neural model; certified state is an admissible approximate FE state
        &
        12 conjugate-gradient iterations from zero displacement on the compatible reduced linear-elasticity system
        &
        \(P_1\) tetrahedral reconstruction; 129{,}260 volume nodes; 640{,}197 tetrahedra; 50{,}994 surface nodes; 385{,}998 free DOFs; \(E=113{,}800.0\), \(\nu=0.342\), density \(4.43\times10^{-9}\)
        &
        Four patch-average surface \(u_z\) queries: load interface, support-near, high-deflection, and remote
        \\
        \addlinespace
        SimJEB load-family appendix study
        &
        Same SimJEB bracket with load family
        \(F_\mu=\mu_1F_{\mathrm{vertical}}+\mu_2F_{\mathrm{horizontal}}+\mu_3F_{\mathrm{diagonal}}+\mu_4F_{\mathrm{torsion}}\)
        &
        Coefficient-output neural model (MLP); reduced compatible basis coefficients; output ranks 2, 3, and 4
        &
        Normalized compatible residual over the load family; AdamW, learning rate \(2\times10^{-3}\), batch size 64, 1200 epochs; 3 hidden layers, width 128, GELU; best checkpoint by validation loss
        &
        512 training, 128 validation, and 256 test random scaled load combinations
        &
        Same four patch-average \(u_z\) queries; rank 3 was the strongest practical reported model
        \\
        \bottomrule
    \end{tabular}
    \end{adjustbox}
    \label{tab:app-experimental-configuration}
\end{table*}

\section{Auxiliary Analytical Material}
\label{appsec:main_relocated_material}

\subsection{Continuous residual-to-point stability route}
\label{appsec:continuous_residual_route}

This subsection gives an alternative route that is not used in the certificate chain. Instead of certifying a compatible discrete target and then transferring to the continuous solution, one can work directly with a continuous residual norm strong enough to control point evaluation.

Assume in this subsection that \(\mathcal A:V\to Y\) is linear and bounded, where
\[
    Y:=Y_\Omega\times Y_{\partial\Omega}
\]
is a Hilbert residual space with norm \(\norm{\cdot}_Y\) and inner product \(\langle\cdot,\cdot\rangle_Y\). Define
\begin{equation}
    \mathcal{A}u:=(\Lop u,\Bop u)\in Y,
    \qquad
    b:=(f,g)\in Y.
\end{equation}
Also define the continuous residual map
\begin{equation}
    \Fop(u):=\mathcal{A}u-b.
\end{equation}
For a fixed query point \(x\in\OmegaBar\), let
\[
    \ell_x(u):=u(x).
\]
If point evaluation is controlled directly by the continuous residual, define
\begin{equation}
    S_x
    :=
    \sup_{\substack{w\in V\\ \mathcal{A}w\neq 0}}
    \frac{\abs{\ell_x(w)}}{\norm{\mathcal{A}w}_Y}.
\end{equation}

\begin{proposition}[Continuous residual-to-point stability]
\label{prop:continuous_residual_bound}
Assume \(S_x<\infty\). Then every admissible field \(u\in V\) satisfies
\begin{equation}
    \abs{u(x)-\ustar(x)}
    \le
    S_x\norm{\Fop(u)}_Y.
\end{equation}
\end{proposition}

If, in addition, there exists a residual-side Green representative \(G_x\in Y\) such that
\begin{equation}
    \ell_x(w)=\langle G_x,\mathcal{A}w\rangle_Y
    \qquad
    \text{for all }w\in V,
\end{equation}
then \(S_x\le \norm{G_x}_Y\), and the pointwise error admits the exact representation
\begin{equation}
    u(x)-\ustar(x)
    =
    \langle G_x,\Fop(u)\rangle_Y.
\end{equation}
For any orthogonal projector \(\Pi_{x,R}:Y\to Y\), this yields the local/tail split
\begin{equation}
    \begin{multlined}
    \abs{u(x)-\ustar(x)}
    \le
    \abs{\langle \Pi_{x,R}G_x,\Fop(u)\rangle_Y}\\
    +
    \norm{(I-\Pi_{x,R})G_x}_Y\,
    \norm{(I-\Pi_{x,R})\Fop(u)}_Y.
    \end{multlined}
    \label{eq:dual_green_local_tail_bound}
\end{equation}
This is the direct continuous analogue of the discrete local/tail split used in the primary formulation.

\subsection{Full-column-rank overdetermined linear extension}
\label{appsec:linear_rectangular_extension}

This subsection gives the full-column-rank overdetermined linear extension. The compatible discrete target is now defined by least squares after admissible reduction, while rank-deficient reduced operators remain outside the certified scope.

\begin{assumption}[Full-column-rank overdetermined reduced linear system]
\label{ass:discrete_rectangular_extension}
The reduced operator \(\widehat A_h\in\R^{M_h\times n_h}\) has full column rank with \(M_h\ge n_h\). Rank-deficient reduced operators are outside the certified scope.
\end{assumption}

Define the compatible discrete target by
\begin{equation}
    y_h^\star
    :=
    \argmin_{y\in\R^{n_h}}
    \frac12\norm{\widehat A_hy-\widehat b_h}_{W_h}^2,
    \qquad
    U_h^\star:=U_{0,h}+Z_hy_h^\star.
    \label{eq:rectangular_discrete_target}
\end{equation}
The associated normal equations are
\begin{equation}
    \widehat A_h^\top W_h(\widehat A_hy_h^\star-\widehat b_h)=0.
    \label{eq:rectangular_normal_equations}
\end{equation}

\begin{proposition}[Exact overdetermined residual-to-error identity]
\label{prop:rectangular_exact_identity}
Assume \Cref{ass:discrete_rectangular_extension}. Then
\begin{equation}
    \Utheta-U_h^\star
    =
    Z_h\widehat H_h^{-1}\widehat A_h^\top W_hr_h.
    \label{eq:rectangular_exact_discrete_identity}
\end{equation}
\end{proposition}

\begin{proposition}[Overdetermined adjoint Green identity]
\label{prop:rectangular_adjoint_green_identity}
Assume \Cref{ass:discrete_rectangular_extension}. For every \(\ell_h\in\R^{N_h}\),
\begin{equation}
    \ell_h^\top(\Utheta-U_h^\star)=g_{\ell,h}^\top r_h.
    \label{eq:rectangular_pointwise_identity}
\end{equation}
Moreover,
\begin{equation}
    \sigma_{\ell,h}
    =
    \left(
    \widehat\ell_h^\top\widehat H_h^{-1}\widehat\ell_h
    \right)^{1/2}.
    \label{eq:rectangular_sigma_identity}
\end{equation}
Consequently,
\begin{equation}
    \abs{\ell_h^\top(\Utheta-U_h^\star)}
    \le
    \sigma_{\ell,h}\norm{r_h}_{W_h}.
    \label{eq:rectangular_pointwise_bound}
\end{equation}
\end{proposition}

\begin{proposition}[Minimal Green representative]
\label{prop:minimal_green_representative}
For every reduced query vector \(\widehat\ell_h\in\R^{n_h}\), the residual-side Green vector \(g_{\ell,h}\) is the unique minimizer of
\[
    \min_{g\in\R^{M_h}}
    \frac12 g^\top W_h^{-1} g
    \qquad
    \text{subject to}
    \qquad
    \widehat A_h^\top g=\widehat\ell_h.
\]
\end{proposition}

\begin{proposition}[Effective residual and natural error scale]
\label{prop:projected_residual_identity}
Assume \Cref{ass:discrete_rectangular_extension}. Let
\[
    e_{y,h}:=y_\theta-y_h^\star,
    \qquad
    P_h^{\mathrm{eff}}
    :=
    \widehat A_h\widehat H_h^{-1}\widehat A_h^\top W_h
    \in\R^{M_h\times M_h}.
\]
Then \(P_h^{\mathrm{eff}}\) is the \(\langle\cdot,\cdot\rangle_{W_h}\)-orthogonal projector onto \(\operatorname{range}(\widehat A_h)\), and
\begin{equation}
    \widehat A_h e_{y,h}=P_h^{\mathrm{eff}} r_h.
    \label{eq:projected_residual_exact}
\end{equation}
Hence only the effective residual \(P_h^{\mathrm{eff}}r_h\) enters the admissible discrete error. Moreover,
\begin{equation}
    \norm{e_{y,h}}_{\widehat H_h}
    =
    \norm{P_h^{\mathrm{eff}}r_h}_{W_h}
    \le
    \norm{r_h}_{W_h},
    \label{eq:natural_error_scale}
\end{equation}
\begin{equation}
    \ell_h^\top(\Utheta-U_h^\star)
    =
    g_{\ell,h}^\top P_h^{\mathrm{eff}}r_h,
    \label{eq:pointwise_projected_residual_identity}
\end{equation}
and therefore
\begin{equation}
    \abs{\ell_h^\top(\Utheta-U_h^\star)}
    \le
    \sigma_{\ell,h}\norm{P_h^{\mathrm{eff}}r_h}_{W_h}
    \le
    \sigma_{\ell,h}\norm{r_h}_{W_h}.
    \label{eq:pointwise_projected_residual_bound}
\end{equation}
\end{proposition}

\begin{proposition}[Overdetermined projector-based local/tail certificate]
\label{prop:rectangular_local_tail_projector_certificate}
Fix a query point \(x_i\) and a projector \(\Pi_{i,R}:\R^{M_h}\to\R^{M_h}\) with \(\Pi_{i,R}^2=\Pi_{i,R}\). Define
\begin{equation}
    \begin{aligned}
    \bar r_h&:=P_h^{\mathrm{eff}}r_h,\\
    g_{i,h}^{\mathrm{loc}}(R)&:=\Pi_{i,R}^\top g_{i,h},\\
    g_{i,h}^{\mathrm{tail}}(R)&:=(I-\Pi_{i,R})^\top g_{i,h},
    \end{aligned}
\end{equation}
and
\begin{equation}
    \bar r_h^{\mathrm{loc}}(R):=\Pi_{i,R}\bar r_h,
    \qquad
    \bar r_h^{\mathrm{tail}}(R):=(I-\Pi_{i,R})\bar r_h.
\end{equation}
Then
\begin{equation}
    \ell_{i,h}^{\top}(\Utheta-U_h^\star)
    =
    \bigl(g_{i,h}^{\mathrm{loc}}(R)\bigr)^\top \bar r_h^{\mathrm{loc}}(R)
    +
    \bigl(g_{i,h}^{\mathrm{tail}}(R)\bigr)^\top \bar r_h^{\mathrm{tail}}(R),
    \label{eq:rectangular_exact_local_tail_identity}
\end{equation}
and therefore
\begin{equation}
    \begin{multlined}
    \abs{\ell_{i,h}^{\top}(\Utheta-U_h^\star)}
    \le
    \abs{\bigl(g_{i,h}^{\mathrm{loc}}(R)\bigr)^\top \bar r_h^{\mathrm{loc}}(R)}\\
    +
    \norm{g_{i,h}^{\mathrm{tail}}(R)}_{W_h^{-1}}
    \norm{\bar r_h^{\mathrm{tail}}(R)}_{W_h}.
    \end{multlined}
    \label{eq:rectangular_certified_local_green_tail_bound}
\end{equation}
If \(P_h^{\mathrm{eff}}r_h\) is not formed explicitly, one may conservatively replace \(\bar r_h\) by \(r_h\).
\end{proposition}

For inexact linear algebra in this full-column-rank overdetermined setting, the least-squares definition of the compatible discrete target is intrinsic, but correction and adjoint quantities should still be computed matrix-free by weighted QR, LSQR/LSMR, or preconditioned Krylov methods rather than by explicitly forming \(\widehat H_h^{-1}\). The certified allowance formulas in \Cref{appsec:linear_auxiliary_material} can then be applied directly.

\subsection{Auxiliary results for linear certification}
\label{appsec:linear_auxiliary_material}

\begin{proposition}[Role of the residual weight]
\label{prop:weight_metric_role}
Let \(W_h\in\R^{M_h\times M_h}\) be symmetric positive definite. Then the reduced normal matrix
\[
    \widehat H_h=\widehat A_h^\top W_h\widehat A_h
\]
is symmetric positive definite if and only if \(\widehat A_h\) has full column rank. Thus \(W_h\) changes the residual geometry but not the rank condition needed for the compatible discrete target.
\end{proposition}

The overdetermined projector identity from \Cref{appsec:linear_rectangular_extension} also implies the following coefficient-norm consequence:
\begin{equation}
    \norm{\Utheta-U_h^\star}_2
    \le
    \norm{Z_h}_2\,
    \lambda_{\min}(\widehat H_h)^{-1/2}
    \norm{r_h}_{W_h}.
    \label{eq:euclidean_coefficient_error_bound}
\end{equation}

\begin{corollary}[Contraction condition for the linear certificates]
\label{cor:interval_shrinkage_conditions}
For a fixed query \(x_i\),
\begin{equation}
    \rho_{i,h}^{\mathrm{corr}}\to 0
    \quad\Longleftrightarrow\quad
    \sigma_{i,h}T_h+\varepsilon_{i,h}^{\mathrm{eval}}\to 0.
    \label{eq:necessary_sufficient_contraction}
\end{equation}
If
\[
    \sigma_{i,h}T_h=O(h^\alpha),
    \qquad
    \varepsilon_{i,h}^{\mathrm{eval}}=O(h^\beta),
\]
then
\[
    \rho_{i,h}^{\mathrm{corr}}=O(h^{\min\{\alpha,\beta\}}).
\]
If, in addition,
\[
    \abs{g_{i,h}^{\top}r_h}=O(h^\gamma),
\]
then
\[
    \rho_{i,h}^{\mathrm{pred}}=O(h^{\min\{\alpha,\beta,\gamma\}}).
\]
If the reported center is an unprojected neural output and
\[
    \zeta_{i,h}^{\theta}=O(h^\delta),
\]
then
\[
    \rho_{i,h}^{\theta}=O(h^{\min\{\alpha,\beta,\gamma,\delta\}}).
\]
If only the norm-only discrete term is used and
\[
    \sigma_{i,h}\norm{r_h}_{W_h}=O(h^\eta),
\]
then
\[
    \rho_{i,h}^{\mathrm{norm}}=O(h^{\min\{\alpha,\beta,\eta\}}).
\]
\end{corollary}

\subsubsection{Inexact linear algebra for certification}

Inexact linear algebra should distinguish the square reduced square-system setting from the full-column-rank overdetermined extension. If \(\widehat A_h\) is square and invertible, then the square-system formulas show that the exact primal correction and adjoint sensitivities are obtained from
\[
    \widehat A_h e_\theta = r_h,
    \qquad
    \widehat A_h^\top z_{i,h} = \widehat\ell_{i,h},
\]
with \(g_{i,h}=z_{i,h}\). In that regime the computation should not be performed through \(\widehat H_h=\widehat A_h^\top W_h\widehat A_h\), because that squares the condition number without changing the exact corrected center.

If \(\widehat A_h\) is full-column-rank overdetermined, the normal matrix is intrinsic to the compatible discrete target, but it should still be accessed matrix-free through weighted QR, LSQR/LSMR, preconditioned Krylov methods, or the residual-side constrained problem of \Cref{prop:minimal_green_representative}, rather than by explicitly forming \(\widehat H_h^{-1}\).

\begin{proposition}[Inexact primal correction certificate]
\label{prop:inexact_primal_correction}
Suppose \(\widetilde e_\theta\in\R^{n_h}\) approximates the exact correction solving
\[
    \widehat H_h e_\theta=\widehat A_h^\top W_h r_h,
\]
and define the linear-solve residual
\[
    \eta_h^{\mathrm{prim}}
    :=
    \widehat A_h^\top W_h r_h-\widehat H_h\widetilde e_\theta.
\]
Assume that
\begin{equation}
    \delta_h^{\mathrm{prim}}
    \ge
    \left(
    (\eta_h^{\mathrm{prim}})^\top\widehat H_h^{-1}\eta_h^{\mathrm{prim}}
    \right)^{1/2}.
    \label{eq:primal_solve_error_bound}
\end{equation}
Define
\[
    \widetilde c_{i,h}^{\mathrm{corr,prim}}
    :=
    \ell_{i,h}^{\top}(\Utheta-Z_h\widetilde e_\theta).
\]
Then
\begin{equation}
    \left|
    \widetilde c_{i,h}^{\mathrm{corr,prim}}
    -
    \ell_{i,h}^{\top}U_h^\star
    \right|
    \le
    \sigma_{i,h}\delta_h^{\mathrm{prim}}.
    \label{eq:primal_correction_query_error}
\end{equation}
Consequently,
\begin{equation}
    \begin{multlined}
    \ustar(x_i)\in\\
    \left[
    \widetilde c_{i,h}^{\mathrm{corr,prim}}
    -
    \sigma_{i,h}\bigl(\delta_h^{\mathrm{prim}}+T_h\bigr)
    -
    \varepsilon_{i,h}^{\mathrm{eval}},\right.\\
    \left.
    \widetilde c_{i,h}^{\mathrm{corr,prim}}
    +
    \sigma_{i,h}\bigl(\delta_h^{\mathrm{prim}}+T_h\bigr)
    +
    \varepsilon_{i,h}^{\mathrm{eval}}
    \right].
    \end{multlined}
    \label{eq:inexact_primal_correction_interval}
\end{equation}
\end{proposition}

Any upper bound \(\bar\sigma_{i,h}\ge\sigma_{i,h}\) may replace \(\sigma_{i,h}\) in \eqref{eq:inexact_primal_correction_interval}; this is useful when \(\sigma_{i,h}\) itself is obtained by an inexact or randomized procedure.

\begin{proposition}[Inexact adjoint certificate]
\label{prop:inexact_adjoint_certificate}
Let \(\widetilde q_{i,h}\) approximate the exact adjoint solution \(q_{i,h}\), and define
\[
    \widetilde g_{i,h}
    :=
    W_h\widehat A_h\widetilde q_{i,h},
    \qquad
    s_{i,h}
    :=
    \widehat\ell_{i,h}-\widehat H_h\widetilde q_{i,h}.
\]
Assume that
\begin{equation}
    \delta_{i,h}^{\mathrm{adj}}
    \ge
    \left(
    s_{i,h}^{\top}\widehat H_h^{-1}s_{i,h}
    \right)^{1/2}.
    \label{eq:adjoint_solve_error_bound}
\end{equation}
Define
\[
    \widetilde\sigma_{i,h}^{+}
    :=
    \|\widetilde g_{i,h}\|_{W_h^{-1}}
    +
    \delta_{i,h}^{\mathrm{adj}}
\]
and
\[
    \begin{multlined}
    \widetilde c_{i,h}^{\mathrm{corr}}
    :=\\
    \ell_{i,h}^{\top}\Utheta-
    \widetilde g_{i,h}^{\top}r_h.
    \end{multlined}
\]
Then
\[
    \|g_{i,h}-\widetilde g_{i,h}\|_{W_h^{-1}}
    \le
    \delta_{i,h}^{\mathrm{adj}},
\]
\[
    \sigma_{i,h}\le \widetilde\sigma_{i,h}^{+},
\]
and
\begin{equation}
    \begin{multlined}
    \ustar(x_i)\in\\
    \left[
    \widetilde c_{i,h}^{\mathrm{corr}}
    -
    \delta_{i,h}^{\mathrm{adj}}\|r_h\|_{W_h}
    -
    \widetilde\sigma_{i,h}^{+}T_h
    -
    \varepsilon_{i,h}^{\mathrm{eval}},\right.\\
    \left.
    \widetilde c_{i,h}^{\mathrm{corr}}
    +
    \delta_{i,h}^{\mathrm{adj}}\|r_h\|_{W_h}
    +
    \widetilde\sigma_{i,h}^{+}T_h
    +
    \varepsilon_{i,h}^{\mathrm{eval}}
    \right].
    \end{multlined}
    \label{eq:inexact_solve_certified_interval}
\end{equation}
\end{proposition}

A sufficient computable choice for \(\delta_h^{\mathrm{prim}}\) or \(\delta_{i,h}^{\mathrm{adj}}\) follows from any coercivity bound
\begin{equation}
    \widehat H_h \succeq \alpha_h P_h,
\end{equation}
where \(P_h\) is symmetric positive definite and \(\alpha_h>0\). Since
\[
    \widehat H_h^{-1}\preceq \alpha_h^{-1}P_h^{-1},
\]
one has
\[
    \left(
    (\eta_h^{\mathrm{prim}})^\top\widehat H_h^{-1}\eta_h^{\mathrm{prim}}
    \right)^{1/2}
    \le
    \alpha_h^{-1/2}\norm{\eta_h^{\mathrm{prim}}}_{P_h^{-1}},
\]
and
\[
    \left(
    s_{i,h}^\top\widehat H_h^{-1}s_{i,h}
    \right)^{1/2}
    \le
    \alpha_h^{-1/2}\norm{s_{i,h}}_{P_h^{-1}}.
\]
Thus the right-hand sides are valid certified choices for
\(\delta_h^{\mathrm{prim}}\) and \(\delta_{i,h}^{\mathrm{adj}}\). In particular, with \(P_h=I\), any verified lower eigenvalue bound \(\lambda_{\min}(\widehat H_h)\ge \alpha_h>0\) is sufficient.

\subsubsection{Verified exact-root ball for square systems}

\begin{proposition}[Verified exact-root ball by Newton--Kantorovich]
\label{prop:newton_kantorovich_verified_ball}
Assume \(M_h=n_h\). Let
\[
    \begin{aligned}
    J_\theta&:=J_h(y_\theta),
    &B_\theta&:=J_\theta^{-1},\\
    r_h&:=F_h(y_\theta),
    &\eta_\theta&:=\|B_\theta r_h\|_{X_h}.
    \end{aligned}
\]
Assume that on the closed ball \(B_{X_h}(y_\theta,R)\),
\begin{equation}
    \begin{multlined}
    \|B_\theta(J_h(u)-J_h(v))\|_{X_h\to X_h}
    \le
    L_\theta^B\|u-v\|_{X_h}\\
    \text{for all }u,v\in B_{X_h}(y_\theta,R).
    \end{multlined}
    \label{eq:kantorovich_lipschitz}
\end{equation}
If
\[
    2L_\theta^B\eta_\theta<1,
\]
define
\begin{equation}
    t_\theta^{\mathrm{root}}
    :=
    \frac{1-\sqrt{1-2L_\theta^B\eta_\theta}}{L_\theta^B},
    \label{eq:kantorovich_radius}
\end{equation}
with \(t_\theta^{\mathrm{root}}=\eta_\theta\) when \(L_\theta^B=0\). If
\[
    t_\theta^{\mathrm{root}}\le R,
\]
then \(F_h\) has a unique zero \(y_h^\star\) in
\[
    B_{X_h}(y_\theta,t_\theta^{\mathrm{root}}).
\]
In particular,
\[
    \|y_\theta-y_h^\star\|_{X_h}
    \le
    t_\theta^{\mathrm{root}}.
\]
\end{proposition}

\section{Nonlinear Stationary-Target Variant}
\label{appsec:nonlinear_stationary_variant}

This appendix isolates the nonlinear least-squares stationary-target construction so that the exact-root chain remains focused.

\subsection{Certification of nonlinear least-squares stationary targets}
\label{subsec:nonlinear_least_squares}

For this section, assume in addition that \(F_h\) is twice continuously differentiable on the relevant neighborhood. Define
\begin{equation}
    \Phi_h(y):=\frac12\norm{F_h(y)}_{W_h}^2,
    \qquad
    S_h(y):=DF_h(y)^\top W_hF_h(y)\in\R^{n_h}.
    \label{eq:nonlinear_stationarity_map}
\end{equation}
A nonlinear least-squares target \(y_{h,\mathrm{ls}}^\star\) satisfies
\begin{equation}
    S_h(y_{h,\mathrm{ls}}^\star)=0.
\end{equation}

\begin{remark}[Meaning of the nonlinear least-squares target]
\label{rem:nonlinear_ls_target_meaning}
The condition
\[
    S_h(y_{h,\mathrm{ls}}^\star)=0
\]
certifies stationarity, not global optimality. The least-squares certificates below are therefore certificates relative to the verified stationary target in the certified neighborhood. If the intended target is the global minimizer of \(\Phi_h\), then a separate global optimality argument is required. If \(K_h(y_{h,\mathrm{ls}}^\star)\) is positive definite on the certified neighborhood, the stationary target is a strict local minimizer in that neighborhood.
\end{remark}

For any base point \(y\), define
\begin{equation}
    K_h(y):=DS_h(y).
\end{equation}
Explicitly, for \(v\in\R^{n_h}\),
\begin{equation}
    K_h(y)v
    =
    J_h(y)^\top W_hJ_h(y)v
    +
    \bigl[D^2F_h(y)[v]\bigr]^\top W_hF_h(y).
    \label{eq:nonlinear_ls_hessian}
\end{equation}
If \(K_h(y)\) is invertible, define \(q_{\ell,h}^{\mathrm{ls}}(y)\) by
\begin{equation}
    K_h(y)^\top q_{\ell,h}^{\mathrm{ls}}(y)=\widehat\ell_h.
\end{equation}
For a deviation \(d\in\R^{n_h}\), define
\begin{equation}
    R_h^{\mathrm{ls}}(y;d)
    :=
    S_h(y-d)-S_h(y)+K_h(y)d.
    \label{eq:ls_base_point_remainder}
\end{equation}
If \(DS_h\) is Lipschitz on \(B_{X_h}(y,t)\) with constant \(L_h^{\mathrm{ls}}(y,t)\) in the operator norm from \((\R^{n_h},\norm{\cdot}_{X_h})\) to \((\R^{n_h},\norm{\cdot}_{X_h^\ast})\), then
\begin{equation}
    \norm{R_h^{\mathrm{ls}}(y;d)}_{X_h^\ast}
    \le
    \frac{L_h^{\mathrm{ls}}(y,t)}{2}\norm{d}_{X_h}^2,
    \qquad
    \norm{d}_{X_h}\le t,
\end{equation}
because
\begin{equation}
    R_h^{\mathrm{ls}}(y;d)
    =
    \int_0^1 \bigl(K_h(y)-K_h(y-sd)\bigr)d\,\dd s.
\end{equation}

\begin{proposition}[Local least-squares correction at a base point]
\label{prop:nonlinear_ls_base_point}
Let \(y\) be a base point at which \(K_h(y)\) is invertible. Assume the stationary target \(y_{h,\mathrm{ls}}^\star\) satisfies
\begin{equation}
    \norm{y-y_{h,\mathrm{ls}}^\star}_{X_h}\le t_h^{\mathrm{ls}}(y),
\end{equation}
and that a certified remainder bound
\begin{equation}
    \mathcal{R}_h^{\mathrm{ls}}(y;t_h^{\mathrm{ls}}(y))
    \ge
    \norm{R_h^{\mathrm{ls}}(y;y-y_{h,\mathrm{ls}}^\star)}_{X_h^\ast}
\end{equation}
is available. Then, for every reduced functional \(\widehat\ell_h\in\R^{n_h}\),
\begin{equation}
    \abs{
    \widehat\ell_h^\top (y-y_{h,\mathrm{ls}}^\star)
    -
    \bigl(q_{\ell,h}^{\mathrm{ls}}(y)\bigr)^\top S_h(y)
    }
    \le
    \norm{q_{\ell,h}^{\mathrm{ls}}(y)}_{X_h}\,
    \mathcal{R}_h^{\mathrm{ls}}(y;t_h^{\mathrm{ls}}(y)).
    \label{eq:nonlinear_ls_base_point_certificate}
\end{equation}
In particular, for query \(x_i\),
\begin{equation}
    \abs{
    L_{i,h}(y)-L_{i,h}(y_{h,\mathrm{ls}}^\star)
    -
    \bigl(q_{i,h}^{\mathrm{ls}}(y)\bigr)^\top S_h(y)
    }
    \le
    \norm{q_{i,h}^{\mathrm{ls}}(y)}_{X_h}\,
    \mathcal{R}_h^{\mathrm{ls}}(y;t_h^{\mathrm{ls}}(y)).
    \label{eq:local_ls_query_certificate}
\end{equation}
\end{proposition}

\begin{proposition}[Least-squares discrete-to-continuous comparison]
\label{prop:nonlinear_ls_discrete_to_continuous}
Let \(y_h^\pi\in\R^{n_h}\) be an admissible comparison vector, set
\begin{equation}
    U_h^\pi:=U_{0,h}+Z_h y_h^\pi,
    \qquad
    s_h^\pi:=S_h(y_h^\pi),
\end{equation}
and assume
\begin{equation}
    \varepsilon_{i,h}^{\mathrm{eval}}
    \ge
    \abs{L_{i,h}(y_h^\pi)-\ustar(x_i)}.
\end{equation}
Assume also that \(K_h(y_h^\pi)\) is invertible and that the same stationary target \(y_{h,\mathrm{ls}}^\star\) satisfies
\begin{equation}
    \norm{y_h^\pi-y_{h,\mathrm{ls}}^\star}_{X_h}\le t_h^{\mathrm{ls}}(y_h^\pi),
\end{equation}
together with a certified remainder bound
\begin{equation}
    \mathcal{R}_h^{\mathrm{ls}}(y_h^\pi;t_h^{\mathrm{ls}}(y_h^\pi))
    \ge
    \norm{R_h^{\mathrm{ls}}(y_h^\pi;y_h^\pi-y_{h,\mathrm{ls}}^\star)}_{X_h^\ast}.
\end{equation}
Then
\begin{equation}
    \abs{L_{i,h}(y_{h,\mathrm{ls}}^\star)-\ustar(x_i)}
    \le
    \mu_{i,h}^{\mathrm{ls,disc}},
\end{equation}
where
\begin{equation}
    \mu_{i,h}^{\mathrm{ls,disc}}
    :=
    \abs{
    \bigl(q_{i,h}^{\mathrm{ls}}(y_h^\pi)\bigr)^\top s_h^\pi
    }
    +
    \norm{q_{i,h}^{\mathrm{ls}}(y_h^\pi)}_{X_h}\,
    \mathcal{R}_h^{\mathrm{ls}}(y_h^\pi;t_h^{\mathrm{ls}}(y_h^\pi))
    +
    \varepsilon_{i,h}^{\mathrm{eval}}.
    \label{eq:nonlinear_ls_disc_bias}
\end{equation}
\end{proposition}

\begin{theorem}[Nonlinear least-squares pointwise certificate]
\label{thm:nonlinear_ls_interval}
Assume the hypotheses of \Cref{prop:nonlinear_ls_base_point} at the computable base point \(y_\theta\) and of \Cref{prop:nonlinear_ls_discrete_to_continuous} at the comparison point \(y_h^\pi\), with the same stationary target \(y_{h,\mathrm{ls}}^\star\). Define
\begin{equation}
    c_{i,h}^{\mathrm{ls,corr}}
    :=
    L_{i,h}(y_\theta)
    -
    \bigl(q_{i,h}^{\mathrm{ls}}(y_\theta)\bigr)^\top S_h(y_\theta).
    \label{eq:nonlinear_ls_corrected_center}
\end{equation}
Then
\begin{equation}
    \begin{multlined}
    \abs{c_{i,h}^{\mathrm{ls,corr}}-\ustar(x_i)}
    \le
    \norm{q_{i,h}^{\mathrm{ls}}(y_\theta)}_{X_h}\,
    \mathcal{R}_h^{\mathrm{ls}}(y_\theta;t_h^{\mathrm{ls}}(y_\theta))\\
    +
    \mu_{i,h}^{\mathrm{ls,disc}}.
    \end{multlined}
\end{equation}
Hence
\begin{equation}
    \ustar(x_i)\in
    \left[
    c_{i,h}^{\mathrm{ls,corr}}-\rho_{i,h}^{\mathrm{ls}},
    \;
    c_{i,h}^{\mathrm{ls,corr}}+\rho_{i,h}^{\mathrm{ls}}
    \right],
\end{equation}
where
\begin{equation}
    \begin{multlined}
    \rho_{i,h}^{\mathrm{ls}}
    :=\\
    \norm{q_{i,h}^{\mathrm{ls}}(y_\theta)}_{X_h}\,
    \mathcal{R}_h^{\mathrm{ls}}(y_\theta;t_h^{\mathrm{ls}}(y_\theta))\\
    +
    \mu_{i,h}^{\mathrm{ls,disc}}.
    \end{multlined}
\end{equation}
Moreover,
\begin{equation}
    \begin{multlined}
    \abs{L_{i,h}(y_\theta)-\ustar(x_i)}
    \le
    \norm{q_{i,h}^{\mathrm{ls}}(y_\theta)}_{X_h}
    \Bigl(
    \norm{S_h(y_\theta)}_{X_h^\ast}
    +
    \mathcal{R}_h^{\mathrm{ls}}(y_\theta;t_h^{\mathrm{ls}}(y_\theta))
    \Bigr)\\
    +
    \mu_{i,h}^{\mathrm{ls,disc}},
    \end{multlined}
\end{equation}
and therefore
\begin{equation}
    \begin{multlined}
    \abs{u_\theta(x_i)-\ustar(x_i)}
    \le
    \zeta_{i,h}^{\theta}\\
    +
    \norm{q_{i,h}^{\mathrm{ls}}(y_\theta)}_{X_h}
    \Bigl(
    \norm{S_h(y_\theta)}_{X_h^\ast}
    +
    \mathcal{R}_h^{\mathrm{ls}}(y_\theta;t_h^{\mathrm{ls}}(y_\theta))
    \Bigr)\\
    +
    \mu_{i,h}^{\mathrm{ls,disc}}.
    \end{multlined}
\end{equation}
\end{theorem}

A radius \(t_h^{\mathrm{ls}}(y_\theta)\) can be certified by applying a square Newton argument to the stationarity map \(S_h\). If \(K_h(y)\) remains positive definite on the verified ball, the stationary point is a strict local minimizer of \(\Phi_h\) in that ball.

\section{Proofs}
\label{app:proofs}

\subsection{Proofs for \Cref{sec:framework} and \Cref{appsec:main_relocated_material}}
\label{app:framework_proofs}

\subsubsection{Proof of \Cref{prop:finite_collocation_no_point_control}}

\begin{proof}
Choose an open ball \(B\Subset \Omega\setminus S\) centered at \(x_0\). Let \(\varphi\in C_c^\infty(B)\) satisfy \(\varphi(x_0)=1\). Because \(\varphi\) vanishes on a neighborhood of every sample point \(z_j\), all derivatives of \(\varphi\) vanish at each \(z_j\); in particular, every derivative actually used by the sampled loss vanishes there. Therefore the sampled loss takes the same value at \(u\) and \(u+\beta\varphi\) for every \(u\) and \(\beta\in\R\), while the value at \(x_0\) shifts by \(\beta\). Thus deterministic certification of the unsampled point value is impossible from finitely many collocation samples alone.
\end{proof}

\subsubsection{Proof of \Cref{prop:continuous_residual_bound}}

\begin{proof}
Let \(w:=u-\ustar\). Since \(\mathcal{A}\ustar=b\),
\begin{equation}
    \mathcal{A}w=\mathcal{A}u-b=\Fop(u).
\end{equation}
The definition of \(S_x\) gives
\begin{equation}
    \abs{u(x)-\ustar(x)}
    =\abs{\ell_x(w)}
    \le S_x\norm{\mathcal{A}w}_Y
    =S_x\norm{\Fop(u)}_Y.
\end{equation}
\end{proof}

\subsubsection{Proof of \Cref{prop:weight_metric_role}}

\begin{proof}
For any \(y\in\R^{n_h}\),
\[
    y^\top \widehat H_h y
    =
    (\widehat A_h y)^\top W_h(\widehat A_h y)
    =
    \norm{\widehat A_h y}_{W_h}^2.
\]
Because \(W_h\) is symmetric positive definite, \(\norm{\widehat A_h y}_{W_h}=0\) if and only if \(\widehat A_h y=0\). Therefore \(y^\top \widehat H_h y>0\) for every nonzero \(y\) if and only if \(\widehat A_h y\neq 0\) for every nonzero \(y\), which is exactly full column rank of \(\widehat A_h\).
\end{proof}

\subsubsection{Proof of \Cref{prop:minimal_green_representative}}

\begin{proof}
The Lagrangian is
\begin{equation}
    \mathcal{L}(g,\lambda)
    =
    \frac12g^\top W_h^{-1}g
    -
    \lambda^\top(\widehat A_h^\top g-\widehat\ell_h).
\end{equation}
Stationarity in \(g\) gives \(g=W_h\widehat A_h\lambda\). The constraint gives
\begin{equation}
    \widehat A_h^\top W_h\widehat A_h\lambda=\widehat\ell_h,
\end{equation}
so \(\lambda=q_{\ell,h}\) and \(g=g_{\ell,h}\). Strict convexity gives uniqueness.
\end{proof}

\subsection{Proofs for \Cref{sec:linear_main} and \Cref{appsec:linear_rectangular_extension}}
\label{appsec:linear_main_proofs}

\subsubsection{Proof of \Cref{prop:rectangular_exact_identity}}

\begin{proof}
Let \(e_y:=y_\theta-y_h^\star\). By \eqref{eq:rectangular_normal_equations},
\begin{align}
    \widehat H_he_y
    &=
    \widehat A_h^\top W_h
    \bigl(\widehat A_hy_\theta-\widehat A_hy_h^\star\bigr) \\
    &=
    \widehat A_h^\top W_h
    \bigl(\widehat A_hy_\theta-\widehat b_h\bigr)
    =
    \widehat A_h^\top W_hr_h.
\end{align}
Since \(\widehat H_h\) is invertible,
\begin{equation}
    e_y=\widehat H_h^{-1}\widehat A_h^\top W_hr_h.
\end{equation}
Multiplication by \(Z_h\) proves \eqref{eq:rectangular_exact_discrete_identity}.
\end{proof}

\subsubsection{Proof of \Cref{prop:rectangular_adjoint_green_identity}}

\begin{proof}
By \Cref{prop:rectangular_exact_identity},
\begin{align}
    \ell_h^\top(\Utheta-U_h^\star)
    &=
    \widehat\ell_h^\top\widehat H_h^{-1}\widehat A_h^\top W_hr_h \\
    &=
    q_{\ell,h}^\top\widehat A_h^\top W_hr_h
    =
    (W_h\widehat A_hq_{\ell,h})^\top r_h
    =
    g_{\ell,h}^\top r_h.
\end{align}
Furthermore,
\begin{align}
    \norm{g_{\ell,h}}_{W_h^{-1}}^2
    &=
    q_{\ell,h}^\top\widehat A_h^\top W_h\widehat A_hq_{\ell,h} \\
    &=
    q_{\ell,h}^\top\widehat H_hq_{\ell,h}
    =
    q_{\ell,h}^\top\widehat\ell_h
    =
    \widehat\ell_h^\top\widehat H_h^{-1}\widehat\ell_h.
\end{align}
The inequality follows from Cauchy--Schwarz in the dual weighted norms.
\end{proof}

\subsubsection{Proof of \Cref{prop:admissible_exact_identity}}

\begin{proof}
Because \(y_h^\star=\widehat A_h^{-1}\widehat b_h\),
\[
    r_h
    =
    \widehat A_hy_\theta-\widehat b_h
    =
    \widehat A_h(y_\theta-y_h^\star).
\]
Hence
\[
    y_\theta-y_h^\star=\widehat A_h^{-1}r_h.
\]
Multiplication by \(Z_h\) proves \eqref{eq:admissible_exact_discrete_identity}.
\end{proof}

\subsubsection{Proof of \Cref{prop:admissible_adjoint_green_identity}}

\begin{proof}
Using \Cref{prop:admissible_exact_identity},
\[
    \ell_h^\top(\Utheta-U_h^\star)
    =
    \widehat\ell_h^\top\widehat A_h^{-1}r_h
    =
    (\widehat A_h^{-\top}\widehat\ell_h)^\top r_h
    =
    g_{\ell,h}^\top r_h.
\]
Furthermore,
\[
    \norm{g_{\ell,h}}_{W_h^{-1}}^2
    =
    q_{\ell,h}^\top\widehat H_hq_{\ell,h}
    =
    q_{\ell,h}^\top\widehat\ell_h
    =
    \widehat\ell_h^\top\widehat H_h^{-1}\widehat\ell_h.
\]
The inequality follows from Cauchy--Schwarz in the dual weighted norms.
\end{proof}

\subsubsection{Proof of \Cref{prop:admissible_discretization_bias}}

\begin{proof}
Because \(y_h^\star=\widehat A_h^{-1}\widehat b_h\),
\[
    \tau_h^\pi
    =
    \widehat A_h y_h^\pi-\widehat b_h
    =
    \widehat A_h(y_h^\pi-y_h^\star).
\]
Therefore
\[
    \ell_{i,h}^{\top}(U_h^\star-U_h^\pi)
    =
    \widehat\ell_{i,h}^{\top}(y_h^\star-y_h^\pi)
    =
    -\widehat\ell_{i,h}^{\top}\widehat A_h^{-1}\tau_h^\pi
    =
    -g_{i,h}^{\top}\tau_h^\pi,
\]
which proves \eqref{eq:sharp_comparison_bias}. Adding and subtracting \(\ell_{i,h}^\top U_h^\pi\) yields \eqref{eq:exact_comparison_bias_identity}, and the bounds follow from Cauchy--Schwarz together with \eqref{eq:certified_consistency_bound}--\eqref{eq:certified_eval_error}.
\end{proof}

\subsubsection{Proof of \Cref{thm:linear_deterministic_interval}}

\begin{proof}
By \Cref{prop:admissible_adjoint_green_identity},
\begin{equation}
    \ell_{i,h}^{\top}(\Utheta-U_h^\star)=g_{i,h}^{\top}r_h.
\end{equation}
Thus
\begin{equation}
    c_{i,h}^{\mathrm{corr}}=\ell_{i,h}^{\top}U_h^\star.
\end{equation}
By \Cref{prop:admissible_discretization_bias},
\begin{equation}
    \abs{c_{i,h}^{\mathrm{corr}}-\ustar(x_i)}
    \le
    \mu_{i,h}.
\end{equation}
which proves \eqref{eq:corrected_interval}. The representation-centered interval follows from
\begin{equation}
    \abs{\ell_{i,h}^{\top}\Utheta-\ustar(x_i)}
    \le
    \abs{g_{i,h}^{\top}r_h}+\mu_{i,h}.
\end{equation}
The neural-output interval follows from
\begin{equation}
    \abs{\utheta(x_i)-\ustar(x_i)}
    \le
    \abs{\utheta(x_i)-c_{i,h}^{\mathrm{corr}}}
    +
    \abs{c_{i,h}^{\mathrm{corr}}-\ustar(x_i)}.
\end{equation}
\[
    \rho_{i,h}^{\theta}
    =
    \abs{\utheta(x_i)-c_{i,h}^{\mathrm{corr}}}
    +
    \mu_{i,h}
    \le
    \zeta_{i,h}^{\theta}
    +
    \abs{g_{i,h}^{\top}r_h}
    +
    \mu_{i,h}
    =
    \zeta_{i,h}^{\theta}
    +
    \rho_{i,h}^{\mathrm{pred}}.
\]
\end{proof}

\subsubsection{Proof of \Cref{prop:exact_correction_recovers_discrete_target}}

\begin{proof}
Because \(e_\theta\) solves \(\widehat A_h e_\theta=r_h\) and \(r_h=\widehat A_h(y_\theta-y_h^\star)\), invertibility of \(\widehat A_h\) gives
\[
    e_\theta=y_\theta-y_h^\star.
\]
Hence
\[
    \Utheta-Z_h e_\theta
    =
    U_{0,h}+Z_hy_\theta-Z_h(y_\theta-y_h^\star)
    =
    U_{0,h}+Z_hy_h^\star
    =
    U_h^\star.
\]
The identity for \(c_{i,h}^{\mathrm{corr}}\) follows from \eqref{eq:corrected_center}.
\end{proof}

\subsection{Proofs for \Cref{sec:linear_refinements} and \Cref{appsec:main_relocated_material}}
\label{app:linear_refinement_proofs}

\subsubsection{Proof of \Cref{prop:projected_residual_identity}}

\begin{proof}
By \Cref{prop:rectangular_exact_identity},
\[
    e_{y,h}
    =
    \widehat H_h^{-1}\widehat A_h^\top W_h r_h.
\]
Multiplication by \(\widehat A_h\) gives \eqref{eq:projected_residual_exact}. The matrix \(P_h^{\mathrm{eff}}\) is idempotent because
\[
    (P_h^{\mathrm{eff}})^2
    =
    \begin{aligned}[t]
    &\widehat A_h\widehat H_h^{-1}\widehat A_h^\top W_h
    \widehat A_h\widehat H_h^{-1}\widehat A_h^\top W_h\\
    &=
    \widehat A_h\widehat H_h^{-1}\widehat H_h
    \widehat H_h^{-1}\widehat A_h^\top W_h
    \end{aligned}
    =
    P_h^{\mathrm{eff}},
\]
and it is \(W_h\)-self-adjoint because
\[
    (P_h^{\mathrm{eff}})^\top W_h
    =
    W_h\widehat A_h\widehat H_h^{-1}\widehat A_h^\top W_h
    =
    W_h P_h^{\mathrm{eff}}.
\]
Since \(\operatorname{range}(P_h^{\mathrm{eff}})\subseteq \operatorname{range}(\widehat A_h)\) and
\[
    P_h^{\mathrm{eff}}(\widehat A_h y)
    =
    \widehat A_h\widehat H_h^{-1}\widehat A_h^\top W_h\widehat A_h y
    =
    \widehat A_h y
\]
for every \(y\in\R^{n_h}\), it is the \(W_h\)-orthogonal projector onto \(\operatorname{range}(\widehat A_h)\).

Next,
\[
    \norm{e_{y,h}}_{\widehat H_h}^2
    =
    \begin{aligned}[t]
    &e_{y,h}^\top \widehat H_h e_{y,h}
    =
    e_{y,h}^\top \widehat A_h^\top W_h \widehat A_h e_{y,h}\\
    &=
    \norm{\widehat A_h e_{y,h}}_{W_h}^2
    =
    \norm{P_h^{\mathrm{eff}}r_h}_{W_h}^2,
    \end{aligned}
\]
which proves \eqref{eq:natural_error_scale}. Because \(P_h^{\mathrm{eff}}\) is a \(W_h\)-orthogonal projector,
\[
    \norm{P_h^{\mathrm{eff}}r_h}_{W_h}\le \norm{r_h}_{W_h}.
\]

For the pointwise identity,
\[
    \ell_h^\top(\Utheta-U_h^\star)
    =
    \begin{aligned}[t]
    &\widehat\ell_h^\top e_{y,h}
    =
    q_{\ell,h}^\top \widehat H_h e_{y,h}\\
    &=
    q_{\ell,h}^\top \widehat A_h^\top W_h \widehat A_h e_{y,h}
    =
    g_{\ell,h}^\top \widehat A_h e_{y,h}.
    \end{aligned}
\]
Using \eqref{eq:projected_residual_exact} gives
\[
    \ell_h^\top(\Utheta-U_h^\star)=g_{\ell,h}^\top P_h^{\mathrm{eff}}r_h,
\]
and \eqref{eq:pointwise_projected_residual_bound} follows by Cauchy--Schwarz.

Finally,
\[
    \norm{\Utheta-U_h^\star}_2
    =
    \begin{aligned}[t]
    &\norm{Z_h e_{y,h}}_2
    \le
    \norm{Z_h}_2 \norm{e_{y,h}}_2\\
    &\le
    \norm{Z_h}_2 \lambda_{\min}(\widehat H_h)^{-1/2}\norm{e_{y,h}}_{\widehat H_h},
    \end{aligned}
\]
and \eqref{eq:euclidean_coefficient_error_bound} follows from \eqref{eq:natural_error_scale}.
\end{proof}

\subsubsection{Proof of \Cref{prop:coercive_galerkin_transfer}}

\begin{proof}
The reduced Galerkin coefficients satisfy \(\widehat A_h y_h^{\mathrm{FE}}=\widehat b_h\). By the definition of \(y_h^\star\) in the square reduced setting, this implies \(y_h^\star=y_h^{\mathrm{FE}}\), hence \(U_h^\star=U_h^{\mathrm{FE}}\). The identity for \(c_{i,h}^{\mathrm{corr}}\) follows from \eqref{eq:corrected_center}. The interval \eqref{eq:coercive_galerkin_corrected_interval} is immediate.
\end{proof}

\subsubsection{Proof of \Cref{thm:interpolation_transfer_bound}}

\begin{proof}
Let \(y_h^\pi\) and \(y_h^{\mathrm{FE}}\) denote the reduced coefficients of \(u_h^\pi\) and \(u_h^{\mathrm{FE}}\). Since \(\widehat A_h y_h^{\mathrm{FE}}=\widehat b_h\),
\[
    \tau_h^\pi
    =
    \widehat A_h y_h^\pi-\widehat b_h
    =
    \widehat A_h(y_h^\pi-y_h^{\mathrm{FE}}).
\]
Because \(\widehat A_h\) is symmetric positive definite and \(W_h=\widehat A_h^{-1}\),
\[
    \norm{\tau_h^\pi}_{W_h}
    =
    \norm{y_h^\pi-y_h^{\mathrm{FE}}}_{\widehat A_h}
    =
    \norm{u_h^\pi-u_h^{\mathrm{FE}}}_{a}.
\]
The standard Céa-type quasi-optimality estimate and the standard \(P_1\) interpolation estimate cited in \Cref{subsec:coercive_galerkin_specialization} give
\[
    \norm{u_h^\pi-u_h^{\mathrm{FE}}}_{a}
    \le
    C_{\mathrm{int}} h \abs{u^\star}_{H^2(\Omega)}.
\]
The standard local pointwise interpolation estimate gives
\[
    \abs{u_h^\pi(x_i)-u^\star(x_i)}
    \le
    C_{\mathrm{pt}} h^2 \abs{u^\star}_{W^{2,\infty}(\omega_i)}
\]
with equality zero at mesh nodes. Inserting these two estimates into
\Cref{prop:admissible_discretization_bias} gives \eqref{eq:coercive_transfer_split}.
\end{proof}

\subsubsection{Proof of \Cref{cor:coercive_galerkin_certificate_split}}

\begin{proof}
Combine \Cref{thm:linear_deterministic_interval,thm:interpolation_transfer_bound}.
\end{proof}

\subsubsection{Proof of \Cref{thm:1d_fully_computable_transfer}}

\begin{proof}
Set
\[
    e:=u^\star-u_h^{\mathrm{FE}}.
\]
Since \(u_h^{\mathrm{FE}}\) is the conforming Galerkin solution, Galerkin orthogonality gives
\[
    a(e,v_h)=0
    \qquad\text{for all }v_h\in V_h.
\]
Taking \(v_h=\mathcal J_h e\) yields
\[
    \norm{e}_{a}^2
    =
    a\bigl(e,e-\mathcal J_h e\bigr).
\]
Because \(\kappa\) is constant on each element \(K\), because \(u_h^{\mathrm{FE}}\) is affine on each \(K\), and because \(e-\mathcal J_h e\) vanishes at every mesh node, elementwise integration by parts gives
\[
    \norm{e}_{a}^2
    =
    \sum_{K\in\Th}
    \int_K
    \bigl(f-c\,u_h^{\mathrm{FE}}\bigr)
    \bigl(e-\mathcal J_h e\bigr)\,\dd x
    =
    \sum_{K\in\Th}
    \int_K
    R_K
    \bigl(e-\mathcal J_h e\bigr)\,\dd x.
\]
Hence
\[
    \norm{e}_{a}^2
    \le
    \sum_{K\in\Th}
    \norm{R_K}_{L^2(K)}
    \norm{e-\mathcal J_h e}_{L^2(K)}.
\]
Now \(e-\mathcal J_h e\in H_0^1(K)\) on each interval \(K\), so the sharp one-dimensional Poincar\'e inequality gives
\[
    \norm{e-\mathcal J_h e}_{L^2(K)}
    \le
    \frac{h_K}{\pi}
    \norm{(e-\mathcal J_h e)'}_{L^2(K)}.
\]
Also, since \(\mathcal J_h e\) is the affine interpolant of \(e\) on \(K\),
\[
    \int_K
    (e-\mathcal J_h e)'(\mathcal J_h e)'\,\dd x
    =
    (\mathcal J_h e)'
    \bigl[
    e-\mathcal J_h e
    \bigr]_{\partial K}
    =
    0,
\]
and therefore
\[
    \norm{(e-\mathcal J_h e)'}_{L^2(K)}
    \le
    \norm{e'}_{L^2(K)}.
\]
Using \(\kappa|_K=\kappa_K\) then gives
\[
    \norm{e-\mathcal J_h e}_{L^2(K)}
    \le
    \frac{h_K}{\pi\sqrt{\kappa_K}}
    \norm{\kappa^{1/2}e'}_{L^2(K)}.
\]
Substituting into the previous estimate and applying Cauchy--Schwarz over the mesh yields
\[
    \norm{e}_{a}^2
    \le
    \eta_h^{\mathrm{tr}}
    \left(
    \sum_{K\in\Th}
    \norm{\kappa^{1/2}e'}_{L^2(K)}^2
    \right)^{1/2}
    \le
    \eta_h^{\mathrm{tr}}\norm{e}_{a}.
\]
This proves
\[
    \norm{u^\star-u_h^{\mathrm{FE}}}_{a}
    \le
    \eta_h^{\mathrm{tr}}.
\]

For the pointwise estimate, use \(e(a)=e(b)=0\). Then
\[
    e(x_i)
    =
    \int_a^{x_i} e'(x)\,\dd x
    =
    -\int_{x_i}^{b} e'(x)\,\dd x.
\]
Weighted Cauchy--Schwarz therefore gives
\[
    \abs{e(x_i)}
    \le
    \min\left\{
    \left(\int_a^{x_i}\kappa(x)^{-1}\,\dd x\right)^{1/2},
    \left(\int_{x_i}^{b}\kappa(x)^{-1}\,\dd x\right)^{1/2}
    \right\}
    \norm{\kappa^{1/2}e'}_{L^2(a,b)}.
\]
Since \(\norm{\kappa^{1/2}e'}_{L^2(a,b)}\le \norm{e}_{a}\), the bound
\[
    \abs{u^\star(x_i)-u_h^{\mathrm{FE}}(x_i)}
    \le
    C_{i,h}^{\mathrm{pt}}\,\eta_h^{\mathrm{tr}}
\]
follows from the energy estimate already proved.

Finally, \Cref{prop:coercive_galerkin_transfer} gives
\[
    U_h^\star=U_h^{\mathrm{FE}},
    \qquad
    c_{i,h}^{\mathrm{corr}}=u_h^{\mathrm{FE}}(x_i).
\]
Choosing \(U_h^\pi=U_h^{\mathrm{FE}}\) gives \(\tau_h^\pi=0\), hence one may take \(T_h=0\), and \eqref{eq:1d_computable_transfer_terms} follows from the pointwise estimate above. The interval \eqref{eq:1d_computable_transfer_interval} is then immediate from \Cref{thm:linear_deterministic_interval}.
\end{proof}

\subsubsection{Proof of \Cref{cor:interval_shrinkage_conditions}}

\begin{proof}
The identity \(\rho_{i,h}^{\mathrm{corr}}=\sigma_{i,h}T_h+\varepsilon_{i,h}^{\mathrm{eval}}\) is \eqref{eq:corrected_radius_decomposition}, so \eqref{eq:necessary_sufficient_contraction} is immediate. The rate statements follow by substitution. The unprojected-output claim uses \eqref{eq:actual_pinn_output_radius}. The norm-only claim uses \eqref{eq:prediction_centered_radius_chain}.
\end{proof}

\subsubsection{Proof of \Cref{prop:local_tail_projector_certificate}}

\begin{proof}
By \Cref{prop:admissible_adjoint_green_identity},
\[
    \ell_{i,h}^{\top}(\Utheta-U_h^\star)
    =
    g_{i,h}^\top r_h.
\]
Since \(\Pi_{i,R}^2=\Pi_{i,R}\) and \((I-\Pi_{i,R})^2=I-\Pi_{i,R}\),
\[
    \bigl(g_{i,h}^{\mathrm{loc}}(R)\bigr)^\top r_h^{\mathrm{loc}}(R)
    =
    g_{i,h}^\top \Pi_{i,R} r_h,
\]
and similarly
\[
    \bigl(g_{i,h}^{\mathrm{tail}}(R)\bigr)^\top r_h^{\mathrm{tail}}(R)
    =
    g_{i,h}^\top (I-\Pi_{i,R})r_h.
\]
Adding these identities gives \eqref{eq:exact_local_tail_identity}. The mixed bound \eqref{eq:certified_local_green_tail_bound} follows by Cauchy--Schwarz on the tail term, and \eqref{eq:pure_norm_local_green_tail_bound} follows by applying Cauchy--Schwarz to both terms.
\end{proof}

\subsubsection{Proof of \Cref{prop:rectangular_local_tail_projector_certificate}}

\begin{proof}
By \Cref{prop:projected_residual_identity},
\[
    \ell_{i,h}^{\top}(\Utheta-U_h^\star)
    =
    g_{i,h}^\top \bar r_h.
\]
Since \(\Pi_{i,R}^2=\Pi_{i,R}\) and \((I-\Pi_{i,R})^2=I-\Pi_{i,R}\),
\[
    \bigl(g_{i,h}^{\mathrm{loc}}(R)\bigr)^\top \bar r_h^{\mathrm{loc}}(R)
    =
    g_{i,h}^\top \Pi_{i,R}\bar r_h,
\]
and similarly
\[
    \bigl(g_{i,h}^{\mathrm{tail}}(R)\bigr)^\top \bar r_h^{\mathrm{tail}}(R)
    =
    g_{i,h}^\top (I-\Pi_{i,R})\bar r_h.
\]
Adding these identities gives \eqref{eq:rectangular_exact_local_tail_identity}. The mixed bound \eqref{eq:rectangular_certified_local_green_tail_bound} follows by Cauchy--Schwarz on the tail term.
\end{proof}

\subsubsection{Proof of \Cref{prop:inexact_primal_correction}}

\begin{proof}
Let \(e_\theta\) denote the exact correction. Since
\[
    \widehat H_h(e_\theta-\widetilde e_\theta)
    =
    \eta_h^{\mathrm{prim}},
\]
one has
\[
    e_\theta-\widetilde e_\theta
    =
    \widehat H_h^{-1}\eta_h^{\mathrm{prim}}.
\]
Therefore
\[
    \widehat\ell_{i,h}^{\top}(e_\theta-\widetilde e_\theta)
    =
    \widehat\ell_{i,h}^{\top}\widehat H_h^{-1}\eta_h^{\mathrm{prim}}.
\]
Cauchy--Schwarz in the \(\widehat H_h\)-duality pairing yields
\[
    \left|
    \widehat\ell_{i,h}^{\top}(e_\theta-\widetilde e_\theta)
    \right|
    \le
    \begin{aligned}[t]
    &\left(
    \widehat\ell_{i,h}^{\top}\widehat H_h^{-1}\widehat\ell_{i,h}
    \right)^{1/2}\\
    &\times
    \left(
    (\eta_h^{\mathrm{prim}})^\top\widehat H_h^{-1}\eta_h^{\mathrm{prim}}
    \right)^{1/2}
    \end{aligned}
    \le
    \sigma_{i,h}\delta_h^{\mathrm{prim}}.
\]
Since
\[
    U_h^\star=\Utheta-Z_he_\theta,
\]
it follows that
\[
    \widetilde c_{i,h}^{\mathrm{corr,prim}}
    -
    \ell_{i,h}^{\top}U_h^\star
    =
    \widehat\ell_{i,h}^{\top}(e_\theta-\widetilde e_\theta).
\]
This proves \eqref{eq:primal_correction_query_error}. Combining it with
\[
    |\ell_{i,h}^{\top}U_h^\star-\ustar(x_i)|
    \le
    \sigma_{i,h}T_h+\varepsilon_{i,h}^{\mathrm{eval}}
\]
proves the interval.
\end{proof}

\subsubsection{Proof of \Cref{prop:inexact_adjoint_certificate}}

\begin{proof}
The exact adjoint \(q_{i,h}\) satisfies
\[
    \widehat H_hq_{i,h}=\widehat\ell_{i,h}.
\]
Therefore
\[
    \widehat H_h(q_{i,h}-\widetilde q_{i,h})
    =
    s_{i,h}.
\]
Since
\[
    g_{i,h}-\widetilde g_{i,h}
    =
    W_h\widehat A_h(q_{i,h}-\widetilde q_{i,h}),
\]
one has
\[
    \|g_{i,h}-\widetilde g_{i,h}\|_{W_h^{-1}}^2
    =
    (q_{i,h}-\widetilde q_{i,h})^\top
    \widehat A_h^\top W_h\widehat A_h
    (q_{i,h}-\widetilde q_{i,h})
\]
and hence
\[
    \|g_{i,h}-\widetilde g_{i,h}\|_{W_h^{-1}}^2
    =
    (q_{i,h}-\widetilde q_{i,h})^\top
    \widehat H_h
    (q_{i,h}-\widetilde q_{i,h})
    =
    s_{i,h}^\top\widehat H_h^{-1}s_{i,h}
    \le
    (\delta_{i,h}^{\mathrm{adj}})^2.
\]
The triangle inequality gives
\[
    \sigma_{i,h}
    =
    \|g_{i,h}\|_{W_h^{-1}}
    \le
    \|\widetilde g_{i,h}\|_{W_h^{-1}}
    +
    \|g_{i,h}-\widetilde g_{i,h}\|_{W_h^{-1}}
    \le
    \widetilde\sigma_{i,h}^{+}.
\]

Let
\[
    c_{i,h}^{\mathrm{corr}}
    =
    \ell_{i,h}^{\top}\Utheta-g_{i,h}^{\top}r_h
\]
be the exact corrected center. Then
\[
    \left|
    \widetilde c_{i,h}^{\mathrm{corr}}
    -
    c_{i,h}^{\mathrm{corr}}
    \right|
    =
    |(g_{i,h}-\widetilde g_{i,h})^\top r_h|
    \le
    \delta_{i,h}^{\mathrm{adj}}\|r_h\|_{W_h}.
\]
Also,
\[
    |c_{i,h}^{\mathrm{corr}}-\ustar(x_i)|
    \le
    \sigma_{i,h}T_h+\varepsilon_{i,h}^{\mathrm{eval}}
    \le
    \widetilde\sigma_{i,h}^{+}T_h+\varepsilon_{i,h}^{\mathrm{eval}}.
\]
Combining these two inequalities proves the interval.
\end{proof}

\subsubsection{Proof of \Cref{prop:newton_kantorovich_verified_ball}}

\begin{proof}
Apply the Newton--Kantorovich theorem to \(F_h\) at \(y_\theta\), with inverse \(B_\theta=J_h(y_\theta)^{-1}\), residual size \(\eta_\theta=\|B_\theta F_h(y_\theta)\|_{X_h}\), and Lipschitz bound \(L_\theta^B\) on the verified ball. The radius \(t_\theta^{\mathrm{root}}\) is the smaller nonnegative solution of
\[
    \eta_\theta+\frac{L_\theta^B}{2}t^2=t.
\]
The condition \(t_\theta^{\mathrm{root}}\le R\) ensures that the Newton--Kantorovich ball lies inside the domain where the Lipschitz bound is valid. The theorem gives existence and uniqueness of a zero in \(B_{X_h}(y_\theta,t_\theta^{\mathrm{root}})\), and hence the stated radius bound.
\end{proof}

\subsection{Proofs for \Cref{sec:nonlinear}}
\label{app:nonlinear_proofs}

\subsubsection{Proof of \Cref{prop:nonlinear_root_base_point}}

\begin{proof}
Set \(d:=y-y_h^\star\). Since \(F_h(y_h^\star)=0\),
\begin{equation}
    0
    =
    F_h(y-d)
    =
    F_h(y)-J_h(y)d+R_h^{\mathrm{root}}(y;d).
\end{equation}
Hence
\begin{equation}
    J_h(y)d
    =
    F_h(y)+R_h^{\mathrm{root}}(y;d).
\end{equation}
Using \(H_h^{\mathrm{root}}(y)q_{\ell,h}^{\mathrm{root}}(y)=\widehat\ell_h\),
\begin{align}
    \widehat\ell_h^\top d
    &=
    \bigl(q_{\ell,h}^{\mathrm{root}}(y)\bigr)^\top
    J_h(y)^\top W_h J_h(y)d \\
    &=
    \bigl(g_{\ell,h}^{\mathrm{root}}(y)\bigr)^\top
    \bigl(F_h(y)+R_h^{\mathrm{root}}(y;d)\bigr).
\end{align}
Therefore
\begin{equation}
    \widehat\ell_h^\top d
    -
    \bigl(g_{\ell,h}^{\mathrm{root}}(y)\bigr)^\top F_h(y)
    =
    \bigl(g_{\ell,h}^{\mathrm{root}}(y)\bigr)^\top R_h^{\mathrm{root}}(y;d),
\end{equation}
and Cauchy--Schwarz in the \(W_h\)-duality pairing gives \eqref{eq:nonlinear_root_base_point_certificate}. Since
\begin{equation}
    L_{i,h}(y)-L_{i,h}(y_h^\star)=\widehat\ell_{i,h}^\top (y-y_h^\star),
\end{equation}
\eqref{eq:local_root_query_certificate} follows.
\end{proof}

\subsubsection{Proof of \Cref{prop:nonlinear_root_discrete_to_continuous}}

\begin{proof}
Applying \Cref{prop:nonlinear_root_base_point} at the base point \(y=y_h^\pi\) gives
\begin{equation}
    \abs{
    L_{i,h}(y_h^\pi)-L_{i,h}(y_h^\star)
    -
    \bigl(g_{i,h}^{\mathrm{root}}(y_h^\pi)\bigr)^\top \tau_h^\pi
    }
    \le
    \sigma_{i,h}^{\mathrm{root}}(y_h^\pi)\,
    \mathcal{R}_h^{\mathrm{root}}(y_h^\pi;t_h^{\mathrm{root}}(y_h^\pi)).
\end{equation}
Hence
\begin{align}
    \abs{L_{i,h}(y_h^\star)-\ustar(x_i)}
    &\le
    \abs{L_{i,h}(y_h^\star)-L_{i,h}(y_h^\pi)}
    +
    \abs{L_{i,h}(y_h^\pi)-\ustar(x_i)} \\
    &\le
    \abs{
    \bigl(g_{i,h}^{\mathrm{root}}(y_h^\pi)\bigr)^\top \tau_h^\pi
    } \\
    &\quad
    +
    \sigma_{i,h}^{\mathrm{root}}(y_h^\pi)\,
    \mathcal{R}_h^{\mathrm{root}}(y_h^\pi;t_h^{\mathrm{root}}(y_h^\pi)) \\
    &\quad
    +
    \varepsilon_{i,h}^{\mathrm{eval}}.
\end{align}
The norm-only variant follows from Cauchy--Schwarz.
\end{proof}

\subsubsection{Proof of \Cref{thm:nonlinear_root_interval}}

\begin{proof}
Applying \Cref{prop:nonlinear_root_base_point} at \(y=y_\theta\) gives
\begin{equation}
    \abs{
    c_{i,h}^{\mathrm{root,corr}}
    -
    L_{i,h}(y_h^\star)
    }
    \le
    \sigma_{i,h}^{\mathrm{root}}(y_\theta)\,
    \mathcal{R}_h^{\mathrm{root}}(y_\theta;t_h^{\mathrm{root}}(y_\theta)).
\end{equation}
Combining this with \Cref{prop:nonlinear_root_discrete_to_continuous} proves \eqref{eq:nonlinear_root_interval_bound}. The norm-only bound \eqref{eq:nonlinear_root_norm_only} follows from
\begin{equation}
    \abs{
    \bigl(g_{i,h}^{\mathrm{root}}(y_\theta)\bigr)^\top r_h
    }
    \le
    \sigma_{i,h}^{\mathrm{root}}(y_\theta)\norm{r_h}_{W_h},
\end{equation}
and the unprojected-output bound follows from
\begin{equation}
    \abs{u_\theta(x_i)-\ustar(x_i)}
    \le
    \zeta_{i,h}^{\theta}
    +
    \abs{L_{i,h}(y_\theta)-\ustar(x_i)}.
\end{equation}
\end{proof}

\subsection{Proofs for \Cref{appsec:nonlinear_stationary_variant}}
\label{app:nonlinear_stationary_proofs}

\subsubsection{Proof of \Cref{prop:nonlinear_ls_base_point}}

\begin{proof}
Set \(d:=y-y_{h,\mathrm{ls}}^\star\). Since \(S_h(y_{h,\mathrm{ls}}^\star)=0\),
\begin{equation}
    0
    =
    S_h(y-d)
    =
    S_h(y)-K_h(y)d+R_h^{\mathrm{ls}}(y;d).
\end{equation}
Thus
\begin{equation}
    K_h(y)d
    =
    S_h(y)+R_h^{\mathrm{ls}}(y;d).
\end{equation}
Using \(K_h(y)^\top q_{\ell,h}^{\mathrm{ls}}(y)=\widehat\ell_h\),
\begin{align}
    \widehat\ell_h^\top d
    &=
    \bigl(q_{\ell,h}^{\mathrm{ls}}(y)\bigr)^\top K_h(y)d \\
    &=
    \bigl(q_{\ell,h}^{\mathrm{ls}}(y)\bigr)^\top S_h(y)
    +
    \bigl(q_{\ell,h}^{\mathrm{ls}}(y)\bigr)^\top R_h^{\mathrm{ls}}(y;d).
\end{align}
Duality between \(X_h\) and \(X_h^\ast\) proves \eqref{eq:nonlinear_ls_base_point_certificate}. The query-point version follows from
\begin{equation}
    L_{i,h}(y)-L_{i,h}(y_{h,\mathrm{ls}}^\star)=\widehat\ell_{i,h}^\top(y-y_{h,\mathrm{ls}}^\star).
\end{equation}
\end{proof}

\subsubsection{Proof of \Cref{prop:nonlinear_ls_discrete_to_continuous}}

\begin{proof}
Applying \Cref{prop:nonlinear_ls_base_point} at the base point \(y=y_h^\pi\) gives
\begin{equation}
    \abs{
    L_{i,h}(y_h^\pi)-L_{i,h}(y_{h,\mathrm{ls}}^\star)
    -
    \bigl(q_{i,h}^{\mathrm{ls}}(y_h^\pi)\bigr)^\top s_h^\pi
    }
    \le
    \norm{q_{i,h}^{\mathrm{ls}}(y_h^\pi)}_{X_h}\,
    \mathcal{R}_h^{\mathrm{ls}}(y_h^\pi;t_h^{\mathrm{ls}}(y_h^\pi)).
\end{equation}
Combining this with
\begin{equation}
    \abs{L_{i,h}(y_h^\pi)-\ustar(x_i)}
    \le
    \varepsilon_{i,h}^{\mathrm{eval}}
\end{equation}
proves the claim.
\end{proof}

\subsubsection{Proof of \Cref{thm:nonlinear_ls_interval}}

\begin{proof}
Applying \Cref{prop:nonlinear_ls_base_point} at \(y=y_\theta\) gives
\begin{equation}
    \abs{
    c_{i,h}^{\mathrm{ls,corr}}
    -
    L_{i,h}(y_{h,\mathrm{ls}}^\star)
    }
    \le
    \norm{q_{i,h}^{\mathrm{ls}}(y_\theta)}_{X_h}\,
    \mathcal{R}_h^{\mathrm{ls}}(y_\theta;t_h^{\mathrm{ls}}(y_\theta)).
\end{equation}
Combining this with \Cref{prop:nonlinear_ls_discrete_to_continuous} proves the interval bound. The norm-only estimate follows from
\begin{equation}
    \abs{
    \bigl(q_{i,h}^{\mathrm{ls}}(y_\theta)\bigr)^\top S_h(y_\theta)
    }
    \le
    \norm{q_{i,h}^{\mathrm{ls}}(y_\theta)}_{X_h}
    \norm{S_h(y_\theta)}_{X_h^\ast},
\end{equation}
and the unprojected-output bound follows from
\begin{equation}
    \abs{u_\theta(x_i)-\ustar(x_i)}
    \le
    \zeta_{i,h}^{\theta}
    +
    \abs{L_{i,h}(y_\theta)-\ustar(x_i)}.
\end{equation}
\end{proof}

\end{document}